%% file: Final_Paper.tex
\documentclass[hidelinks, reqno]{amsart}

\usepackage{amsmath}
\usepackage{amsthm}
\usepackage{amssymb}
\usepackage{amsfonts} %Already loaded by amssymb.
\usepackage{amscd, dsfont}
\usepackage{graphicx}
\usepackage[all,arc]{xy}
\usepackage{enumerate}
\usepackage{enumitem}
\usepackage{array}
\usepackage{multirow}
\usepackage[scr]{rsfso} %Better than mathrsfs. Removed euscript. https://sites.math.washington.edu/~lee/Writing/typesetting-script.pdf 
\usepackage{caption}
\usepackage{subcaption}
\usepackage[pagebackref = true, colorlinks, linkcolor = Red, citecolor = Green, bookmarksdepth=3, unicode, linktocpage=true]{hyperref}
\usepackage[capitalize]{cleveref}
\usepackage{comment}
\usepackage{color}
\usepackage[usenames,dvipsnames]{xcolor}
\usepackage{tikz-cd, tikz}% for drawing graph
\usepackage{pgfplots}
\usetikzlibrary{patterns}
\usetikzlibrary{decorations.markings}
\usepackage{bm} %For bold math (keeping italics). For bold upright, use mathbf.
\usepackage{bbm} %Alternative style to mathbb. https://ctan.math.illinois.edu/macros/latex/contrib/bbm/bbm.pdf
\usepackage[utf8]{inputenc}

\newtheorem{theorem}{Theorem}[section]
\newtheorem{proposition}[theorem]{Proposition}
\newtheorem{lemma}[theorem]{Lemma}
\newtheorem{corollary}[theorem]{Corollary}

\theoremstyle{definition}
\newtheorem{definition}[theorem]{Definition}
\newtheorem{example}[theorem]{Example}

\theoremstyle{remark}
\newtheorem{remark}[theorem]{Remark}
\makeatletter
\renewcommand\subsubsection{\@startsection{subsubsection}{3}%
\z@{.5\linespacing\@plus.7\linespacing}{-.5em}%
{\normalfont\bfseries\itshape}} %Changed \normalfont\itshape to \normalfont\bfseries\itshape
\makeatother

\DeclareMathOperator{\SO}{SO}
\DeclareMathOperator{\Id}{Id}
\DeclareMathOperator{\PSL}{PSL}
\DeclareMathOperator{\SL}{SL}

\DeclareMathOperator{\Ad}{Ad}
\DeclareMathOperator{\BP}{BP}

\DeclareMathOperator{\Stab}{Stab}
\DeclareMathOperator{\diam}{diam}
\let\U\relax
\DeclareMathOperator{\U}{U}

\DeclareMathOperator{\supp}{supp}
\DeclareMathOperator{\Lip}{Lip}
\DeclareMathOperator{\Li}{Li}

\DeclareMathOperator{\T}{\operatorname{T}^1}
\DeclareMathOperator{\F}{F}

\let\C\relax
\newcommand{\C}{\mathbb{C}}
\newcommand{\N}{\mathbb{N}}

\newcommand{\R}{\mathbb{R}}
\newcommand{\Z}{\mathbb{Z}}

\renewcommand{\bf}[1]{\mathbf{#1}}

\newcommand{\cal}[1]{\mathcal{#1}}

\newcommand{\dd}{d}

\newcommand{\calM}{\cal M}
\newcommand{\calH}{\cal H}
\newcommand{\calL}{\cal L}
\newcommand{\calR}{\cal R}
\newcommand{\calP}{\cal P}
\newcommand{\calQ}{\cal Q}

\newcommand{\LieG}{\mathfrak{g}}
\newcommand{\LieA}{\mathfrak{a}}
\newcommand{\LieN}{\mathfrak{n}}

\newcommand{\LieM}{\mathfrak{m}}

\newcommand{\HS}{\mathbb{H}^{d+1}}
\newcommand{\Ret}{\mathscr{R}} %Replaced all the return time R's with \Ret
\newcommand{\FRet}{\mathscr{F}}
\newcommand{\Hol}{\mathscr{H}}
\newcommand{\GHol}{\mathscr{G}}
\newcommand{\m}{m^{\mathrm{BMS}}}

\newcommand{\Leb}{m^{\mathrm{Leb}}}
\newcommand{\jj}{j_0}
\renewcommand{\L}{L}

\newcommand{\jI}{{\bm j}}

\Crefname{subsection}{Subsection}{Subsections}

\title[Exponential mixing of frame flows] {Exponential mixing of frame flows for geometrically finite hyperbolic manifolds}

\author{Jialun Li}
\address{Institut f\"ur Mathematik,
Universit\"at Z\"urich,
Winterthurerstrasse 190,
CH-8057 Z\"urich, Switzerland; 
\textit{Current}: CNRS-Centre de math\'ematiques Laurent Schwartz,
\'Ecole Polytechnique,
91128 PALAISEAU, France }
\email{jialun.li@polytechnique.edu} %You don't want to use an institution email?

\author{Wenyu Pan}
\address{Department of Mathematics, University of Toronto, 40 St George St, Toronto, ON M5S 2E4, Canada}
\email{wenyup.pan@utoronto.ca}

\author{Pratyush Sarkar}
\address{Department of Mathematics, UC San Diego, 9500 Gilman Drive, San Diego, CA 92093, USA}
\email{psarkar@ucsd.edu}

\date{\today}

\begin{document}

\begin{abstract}
As a final work to establish that frame flows for geometrically finite hyperbolic manifolds of arbitrary dimensions are exponentially mixing with respect to the Bowen--Margulis--Sullivan measure, this paper focuses on the case with \emph{cusps}. To prove this, we utilize the countably infinite symbolic coding of the geodesic flow of Li--Pan and perform a frame flow version of Dolgopyat's method \`{a} la Sarkar--Winter and Tsujii--Zhang. This requires the local non-integrability condition and the non-concentration property but the challenge in the presence of cusps is that the latter holds only on a large \emph{proper} subset of the limit set. To overcome this, we use a large deviation property for symbolic recurrence to the large subset. It is proved by studying the combinatorics of cusp excursions and using an effective renewal theorem as in the work of Li; the latter uses the exponential decay of the transfer operators for the geodesic flow of Li--Pan. %Applications of the main theorem include an asymptotic formula for matrix coefficients with an exponential error term and exponential equidistributions.
\end{abstract}

\maketitle

\setcounter{tocdepth}{1}
\tableofcontents

\section{Introduction}
\label{sec:Introduction}
Let $\HS$ be the $(d + 1)$-dimensional hyperbolic space for any $d \in \N$. Let $G = \SO(d+1,1)^{\circ}$ endowed with a left $G$-invariant and right $K$-invariant Riemannian metric, which can be identified with the group of orientation-preserving isometries of $\HS$. Let $\Gamma < G$ be a torsion-free discrete subgroup. Consider the hyperbolic manifold $X = \Gamma \backslash \HS \cong \Gamma \backslash G/K$ whose unit tangent bundle is $\T(X) \cong \Gamma \backslash G/M$ and whose (oriented orthonormal) frame bundle is $\F(X) \cong \Gamma \backslash G$ where $M < K$ are compact subgroups of $G$. Let $A = \{a_t: t \in \R\}$ be the one-parameter subgroup of semisimple elements whose right translation action gives the geodesic flow on $\Gamma \backslash G/M$ and the frame flow on $\Gamma \backslash G$. We take our frame flow invariant measure $m^{\mathrm{BMS}}$ to be the Bowen--Margulis--Sullivan probability measure on $\Gamma \backslash G$ which is supported on the non-wandering set. It is the $M$-invariant lift of the Bowen--Margulis--Sullivan probability measure on $\Gamma \backslash G/M$ which is known to be the unique probability measure of maximal entropy. The entropy coincides with the critical exponent $\delta \in [0, d]$ of $\Gamma$.

If $\Gamma < G$ is a lattice, i.e., if $X$ is of \emph{finite} volume, then ${m}^{\mathrm{BMS}}$ coincides with the $G$-invariant probability measure and it is well-known in the literature that the frame flow is exponentially mixing \cite{Rat87,Moo87}. Its proof makes extensive use of spectral gap and representation theory.

In recent times, there has been significant progress regarding the mixing properties with respect to ${m}^{\mathrm{BMS}}$ for hyperbolic manifolds of \emph{infinite} volume. Under the natural hypothesis that $\Gamma < G$ is non-elementary and ${m}^{\mathrm{BMS}}$ is finite, mixing of the geodesic flow is due to Rudolph \cite{Rud82} and Babillot \cite{Bab02}. For frame flows, it is necessary to assume that $\Gamma < G$ is Zariski dense because otherwise it is not even ergodic. Under the natural hypothesis that $\Gamma < G$ is Zariski dense and ${m}^{\mathrm{BMS}}$ is finite, mixing of the frame flow is due to Winter \cite{Win15} (see also \cite{BP74,FS90}). Let us now turn to \emph{exponential} mixing. If $\Gamma < G$ is Zariski dense and geometrically finite, representation theoretic techniques were extended by Mohammadi--Oh \cite{MO15}, so long as the critical exponent $\delta$ is large so that a certain spectral gap holds, to prove exponential mixing of the frame flow (see also \cite{EO21} for the geodesic flow). However, when the critical exponent $\delta$ is small, such a spectral gap does not exist. Hence, it is more fruitful to apply techniques which are more dynamical. If $\Gamma < G$ is non-elementary and geometrically finite, exponential mixing of the geodesic flow is due to Stoyanov \cite{Sto11} when $\Gamma$ does not contain parabolic elements (i.e., convex cocompact) and due to the first two authors Li--Pan \cite{LP22} (see also the work of Khalil \cite{Kha21}) when $\Gamma$ contains parabolic elements. If $\Gamma < G$ is Zariski dense and geometrically finite without parabolic elements, exponential mixing of the frame flow is known by a joint work of the third author Sarkar--Winter \cite{SW21} (see also \cite{CS22}).

The purpose of this paper is to complete the remaining case for exponential mixing of frame flows for geometrically finite hyperbolic manifolds and thereby establish \cref{thm:ExponentialMixing}. That is, we prove \cref{thm:ExponentialMixing} for Zariski dense torsion-free geometrically finite subgroups $\Gamma < G$ with \emph{parabolic elements}. For any $\alpha \in (0, 1]$, we denote by $C^\alpha(\Gamma \backslash G, \R)$ the space of $\alpha$-H\"{o}lder continuous functions on $\Gamma \backslash G$.

\begin{theorem}
\label{thm:ExponentialMixing}
Let $G = \SO(d + 1, 1)^\circ$ for any $d \in \N$ and $\Gamma < G$ be a Zariski dense torsion-free geometrically finite subgroup. Let $\alpha \in (0, 1]$. There exist $\eta_\alpha > 0$ and $C > 0$ (independent of $\alpha$) such that for all $\phi, \psi \in C^\alpha(\Gamma \backslash G, \R)$ and $t > 0$, we have
\begin{align*}
\left|\int_{\Gamma \backslash G} \phi(x) \psi(x a_t) \, dm^{\mathrm{BMS}} - m^{\mathrm{BMS}}(\phi) \cdot m^{\mathrm{BMS}}(\psi)\right| \leq C\|\phi\|_{C^\alpha} \|\psi\|_{C^\alpha} e^{-\eta_\alpha t}.
\end{align*}
\end{theorem}

As indicated above, the proof is a dynamical one and is based on a combination of the works of Stoyanov \cite{Sto11}, Sarkar--Winter \cite{SW21}, Li--Pan \cite{LP22}, and Tsuji--Zhang \cite{TZ23} which build on the framework introduced in the work of Dolgopyat \cite{Dol98}, now commonly called Dolgopyat's method. We emphasize that even when $\Gamma < G$ is a \emph{lattice} with parabolic elements such as the familar settings $\PSL_2(\Z) < \PSL_2(\R)$ or $\PSL_2(\Z[i]) < \PSL_2(\C)$, the work of Li--Pan \cite{LP22} and the work of this paper are the first \emph{dynamical proofs}, to the authors' best knowledge, for exponential mixing of the geodesic flow and the frame flow respectively.

\subsection{Connections, applications, and further directions}
In broad context, frame flows have attracted substantial attention since they serve as typical examples of partially hyperbolic systems. Consider $(\mathcal{M},g)$ a smooth closed oriented Riemannian manifold of dimension $n\geq 3$ with negative sectional curvature. Brin conjectured that if $(\mathcal{M},g)$ is strictly $1/4$-pinched, the frame flow is ergodic. In \cite{BG80}, Brin--Gromov verified the case when $n$ is odd and $n\neq 7$. Recently, Ceki\'c--Lefeuvre--Moroianu--Semmelmann \cite{CLMS21} made progress on the case when $n$ is even or $n=7$. For the quantitative theory, Dolgopyat \cite{Dol02} treated the mixing properties of compact group extensions of hyperbolic diffeomorphisms which are discrete-time versions of frame flows. He proved the equivalence between an infinitesimal non-integrability condition and the exponential mixing of compact group extensions of expanding maps on closed manifolds. In \cite{Sid22}, Siddiqi considered the compact extensions of a certain class of Anosov flows, where he translated the accessibility properties of the extension into Dolgopyat's non-integrability condition. Besides dynamical approaches, there are works using further analytic tools to study frame flows. For example, Guillarmou--K\"uster employed semiclassical or microlocal analysis to study the spectrum of frame flows for $3$-dimensional closed hyperbolic manifolds \cite{GK21}.

More specifically in the context of homogeneous dynamics, studying frame flows have proven to be fruitful due to numerous applications which have been derived in prior works. We state the following selection of them here for the convenience of the reader: decay of matrix coefficients with exponential error term; exponential equidistribution of holonomies; effective equidistribution of horospheres. Let $\Gamma < G$ be as in \cref{thm:ExponentialMixing} for the rest of the subsection.

Fix a Haar measure on $G$. It induces a right $G$-invariant measure on $\Gamma\backslash G$ and also the unstable and stable Burger--Roblin measures on $\Gamma\backslash G$ denoted by $m^{\mathrm{BR}}$ and $m^{\mathrm{BR}_*}$, respectively. Using an effective version of Roblin’s transverse intersection argument as in \cite[Theorem 5.8]{OW16}, \cref{thm:ExponentialMixing} implies the following theorem. The original (ineffective) argument is in \cite{Rob03} (see also \cite{OS13}).

\begin{theorem}
\label{thm:MatrixCoefficient}
Let $\alpha \in (0, 1]$. There exists $\eta_\alpha > 0$ such that for all $\phi, \psi \in C_{\mathrm{c}}^\alpha(\Gamma \backslash G, \R)$, there exists $C > 0$ (depending only on $\supp(\phi)$ and $\supp(\psi)$, and independent of $\alpha$) such that for all $t > 0$, we have
\begin{align*}
\left|e^{(d - \delta)t}\int_{\Gamma \backslash G} \phi(x) \psi(x a_t) \, dx - m^{\mathrm{BR}}(\phi) \cdot m^{\mathrm{BR}_*}(\psi)\right| \leq C\|\phi\|_{C^\alpha} \|\psi\|_{C^\alpha} e^{-\eta_\alpha t}.
\end{align*}
\end{theorem}

For all $T > 0$, define
\begin{align*}
\mathcal{G}(T) = \#\{\gamma: \gamma \text{ is a primitive closed geodesic in } \Gamma \backslash \mathbb H^n \text{ with length at most } T\}.
\end{align*}
For all primitive closed geodesics $\gamma$ in $\Gamma \backslash \mathbb H^n$, its \emph{holonomy} is a conjugacy class $h_\gamma$ in $M$ induced by parallel transport along $\gamma$. Fix the Haar probability measure on $M$. Recall the function $\Li: (2, \infty) \to \mathbb R$ defined by $\Li(x) = \int_2^x \frac{1}{\log(t)} \, dt$ for all $x \in (2, \infty)$. Denote by $\|\cdot\|_{\mathcal{S}^k}$ the $L^2$ Sobolev norm of order $k \in \N$. Since \cite[Lemma 3.8]{DFSU21} holds in our setting, in light of the remark in \cite[Section 1]{SW21}, \cref{thm:ExponentialMixing} implies the following theorem due to the work of Margulis--Mohammadi--Oh \cite{MMO14}.

\begin{theorem}
\label{thm:Equidistribution}
There exist $k \in \N$ and $\eta > 0$ such that for all class functions $\phi \in C^\infty(M, \R)$, we have
\begin{align*}
\sum_{\gamma \in \mathcal{G}(T)} \phi(h_\gamma) = \Li\bigl(e^{\delta T}\bigr) \int_M \phi(m) \, dm + O\bigl(e^{(\delta - \eta)T}\bigr) \qquad \text{as $T \to +\infty$}
\end{align*}
where the implied constant depends only on $\|\phi\|_{\mathcal{S}^k}$.
\end{theorem}

Let $\mathcal{C}_0 \subset \Gamma \backslash G$ be a fixed compact subset from a thick-thin decomposition. For all $\epsilon \in (0, 1)$ and $t_0 \geq 1$, we say that $x \in \Gamma \backslash G$ is $(\epsilon, t_0)$-Diophantine if any of its lift in $G$ has backward endpoint in the limit set $\Lambda_\Gamma$ and $d(\mathcal{C}_0, xa_{-t}) < (1 - \epsilon)t$ for all $t \geq t_0$. Let $N^+ < G$ be the unstable horospherical subgroup endowed with the induced Riemannian metric. Denote by $N_{\mathrm{max},R}^+ \subset N^+$ the image under $\exp: \mathfrak{n}^+ \to N^+$ of the ball of radius $R > 0$ centered at $0$ with respect to the max norm. Fix the Haar measure on $N^+$ compatible with the one on $G$. For any $x \in \Gamma \backslash G$, denote by $\Leb_{xN^+}$ the corresponding Lebesgue measure on $xN^+$ and by $\mu_{xN^+}^{\mathrm{PS}}$ the Patterson--Sullivan measure on $xN^+$. \Cref{thm:ExponentialMixing} implies the following theorem due to the work of Tamam--Warren \cite{TW22}. The latter result in \cref{thm:EffectiveEquidistribution} was earlier obtained by Edwards \cite{Edw22} when $d = 2$, i.e., for surfaces.

\begin{theorem}
\label{thm:EffectiveEquidistribution}
There exists $k \in \N$ such that for all $\epsilon \in (0, 1)$, there exists $\eta > 0$ such that for all $t_0 \geq 1$, the following holds. For all $\phi \in C_{\mathrm{c}}^\infty(\Gamma \backslash G, \R)$, there exists $C > 0$ (depending only on $\supp(\phi)$) such that:
\begin{enumerate}
\item for all $(\epsilon, t_0)$-Diophantine $x \in \Gamma \backslash G$ and $R \gg_{\epsilon} t_0$, we have
\begin{align*}
\left|\frac{1}{\mu_{xN^+}^{\mathrm{PS}}\bigl(xN_{\mathrm{max},R}^+\bigr)} \int_{xN_{\mathrm{max},R}^+} \phi \, d\mu_{xN^+}^{\mathrm{PS}} - m^{\mathrm{BMS}}(\phi)\right| \leq C\|\phi\|_{\mathcal{S}^k} R^{-\eta};
\end{align*}
\item for all $(\epsilon, t_0)$-Diophantine $x \in \Gamma \backslash G$ and $R \gg_{\epsilon, \supp(\phi)} t_0$, we have
\begin{align*}
\left|\frac{1}{\mu_{xN^+}^{\mathrm{PS}}\bigl(xN_{\mathrm{max},R}^+\bigr)} \int_{xN_{\mathrm{max},R}^+} \phi \, d\Leb_{xN^+} - m^{\mathrm{BR}}(\phi)\right| \leq C\|\phi\|_{\mathcal{S}^k} R^{-\eta}.
\end{align*}
\end{enumerate}
\end{theorem}

We now mention some further directions which is outside the scope of this paper. Recently, Chow--Sarkar \cite{CS22} extended the main theorem of \cite{SW21} to convex cocompact rank one locally symmetric spaces. The natural further work is to generalize and combine the techniques in \cite{LP22}, \cite{CS22}, and this paper, to simultaneously extend the main theorems to geometrically finite rank one locally symmetric spaces. Another direction is the question of uniform exponential mixing which was addressed in the convex cocompact setting initially for hyperbolic surfaces by Oh--Winter \cite{OW16}, and later for hyperbolic manifolds by Sarkar \cite{Sar22a}. The natural further work is to treat the geometrically finite setting. In the work of Avila--Gou\"{e}zel--Yoccoz \cite{AGY06}, they proved exponential mixing of the Teichm\"{u}ller geodesic flow which was then shown to imply that the $\SL_2(\mathbb{R})$-action on the moduli space of Abelian differentials has a spectral gap (see also \cite{AG13}). Inspired by this, it would be interesting to study whether one can deduce a spectral gap for the $G$-action on $L^2(\Gamma\backslash G)$ from \cref{thm:ExponentialMixing}. Last but not least, inspired by results regarding resonance-free half plane for the resolvent of the Laplacian obtained from exponential mixing of the geodesic flow (see \cite{OW16,Sar22a,LP22}), one could explore analogues for exponential mixing of the frame flow.

\subsection{Outline of the proof}
As mentioned before, we prove \cref{thm:ExponentialMixing} for Zariski dense torsion-free geometrically finite subgroups $\Gamma < G$ with \emph{parabolic elements} using dynamical techniques, namely, using Dolgopyat's method \cite{Dol98,Dol02}. We recall that Dolgopyat's method was originally developed for the finite symbolic setting. It was later adapted to the countable symbolic setting in the works of many authors such as Baladi--Vall\'{e}e \cite{BV05}, Avila--Gou\"{e}zel--Yoccoz \cite{AGY06}, and Ara\'{u}jo--Melbourne \cite{AM16}. These techniques were recently utilized by Li--Pan \cite{LP22} for geodesic flows for geometrically finite hyperbolic manifolds with cusps. Similarly, our proof begins by utilizing the countably infinite coding constructed in \cite{LP22}. The countably infinite coding is given by $\Delta_{\sqcup} := \bigsqcup_{j \in \mathcal{J}} \Delta_j$ which is a full Patterson--Sullivan measure subset of the fundamental domain $\Delta_0$ of one of the parabolic fixed points. It also comes with an associated (piecewise) expanding map $T: \Delta_{\sqcup} \to \Delta_0$ which satisfies some important properties. Using this coding, we wish to proceed as in \cite{SW21} by performing a frame flow version of Dolgopyat's method using transfer operators twisted by holonomy. 

As in \cite{SW21}, the cancellations in the summands of the transfer operators are derived using the \emph{local non-integrability condition (LNIC)} and the \emph{non-concentration property (NCP)}. Non-integrability type conditions have been in extensive use since the works of Chernov \cite{Che98} and Dolgopyat \cite{Dol98} and LNIC is the appropriate generalization for frame bundles. We first describe NCP \emph{as stated in \cite[Proposition 6.6]{SW21}} for $\Gamma$ without parabolic elements and why it is required, and then describe its generalization for $\Gamma$ with parabolic elements. When dealing with frame flows, accessibility properties of the dynamical system play an important role in the analysis, sometimes without obstructions. In our setting, the non-wandering set of the frame flow is typically fractal in nature and consequently, there are obstructions. Roughly speaking, the non-integrability can be measured using certain ideal frames which are accessible starting from a reference frame via the strong stable and strong unstable foliations of the frame flow and we would like them to be contained in the non-wandering set. The latter condition is typically not satisfied by the ideal frames but at worst, we may use approximate frames which are close to the ideal ones where the closeness is measured precisely by an angular bound in the Lie algebra of the horospherical subgroup. NCP is a property for the limit set $\Lambda_\Gamma$ which very roughly says that such an angular bound holds for vectors in $\Lambda_\Gamma$, or in other words, $\Lambda_\Gamma$ does not concentrate along any proper affine subspace. This ensures that the approximate frames are contained in the non-wandering set. Note that if $\Gamma$ is a lattice, there are no obstructions in accessibility---this can be seen either from the fact that the non-wandering set is all of $\Gamma \backslash G$ or the fact that NCP is trivial since the limit set is all of $\partial \HS$, and so we may use the ideal frames directly.

Throughout the literature, there has been technical difficulties in various contexts due to the parabolic elements of non-maximal rank (see \cite{FHP91,Gui06,GM12,DFSU21,TW22}). Similarly, in our setting, the main difficulty for $\Gamma$ with parabolic elements lies in the fact that if there are parabolic elements of non-maximal rank, then NCP \emph{as stated in \cite[Proposition 6.6]{SW21}} does \emph{not} hold---$\Lambda_\Gamma$ gets concentrated along a proper affine subspace near the parabolic fixed points of non-maximal rank (see \cref{exa:non concentration})---which is insufficient to obtain the aforementioned cancellations (see \cref{sec:Dolgopyat'sMethod}). Nevertheless, an appropriate generalization of NCP still holds where we simply restrict to a large proper subset $\Omega_t \cap \Lambda_\Gamma \subset \Delta_0 \cap \Lambda_\Gamma$ which is bounded away from the parabolic fixed points, associated to large frequencies $t \gg 1$. Consequently, we can only obtain cancellations on $\Omega_t$. In order to tackle the small bad set $\Delta_0 - \Omega_t$, we turn to the recent techniques of Tsujii--Zhang \cite{TZ23}. They developed a less stringent version of Dolgopyat's method where they use stochastic dominance to show that it is sufficient to have cancellations inside a set with the \emph{large deviation property (LDP)}. Here LDP is with respect to the expanding map $T$ and so it is convenient to work with $\Lambda_+ = \{x\in \Lambda_{\Gamma}: T^n \text{ is defined at } x \text{ for all } n \in \mathbb{N}\}$ which is a Patterson--Sullivan full measure subset of $\Delta_0$. Using this version of Dolgopyat's method, we need an extra step in our setting which is to show that $\Omega_t \cap \Lambda_+$ satisfies LDP \emph{uniformly} in $t \gg 1$. There are existing techniques to derive LDP for $\Omega_t \cap \Lambda_+$ for a fixed $t$ but the challenge comes from the uniformity in $t \gg 1$. The larger the parameter $t \gg 1$, the more irregular the set $\Omega_t \cap \Lambda_+$, making it unclear how to obtain LDP for the expanding map at first glance. The proof is based on the interplay of two different dynamics: for $x\in \Lambda_+$ and $n \in \N$, we see whether the points $T^n(x)$ lying in $\Omega_t$ is roughly equivalent to the corresponding unit tangent vectors $u_{T^n(x)}$ (on a fixed piece of an unstable horosphere) returning to a fixed compact subset of $\T(X)$ under the geodesic flow $a_t$. From here, we study the combinatorics of cusp excursions. One of the key propositions is to estimate the probability of a ``random walk'' on $\R$ with a large residual waiting time. It is proved using an effective renewal theorem as in the work of Li \cite{Li22} which uses the exponential decay of the transfer operators for the geodesic flow of Li--Pan \cite{LP22}. We think of this in the form of the following interesting observation: for $\Gamma$ with parabolic elements of non-maximal rank, proving exponential mixing of the geodesic flow is \emph{required first} as an intermediate step before proving exponential mixing of the frame flow.

\subsection{Organization of the paper}
First we provide the necessary background in \cref{sec:Preliminaries,sec:SymbolicModel}. We then introduce the transfer operators with holonomy for the countably infinite symbolic coding in \cref{sec:TransferOperators}. \Cref{sec:LNIC,sec:NCP,sec:LDP} are independent of each other and provide the three key ingredients, LNIC, NCP, and LDP, required for Dolgopyat's method in \cref{sec:Dolgopyat'sMethod}. We finish briefly with \cref{sec:ExponentialMixing} to convert the spectral bounds to exponential mixing. We refer to \cref{fig:Sections} for a comprehensive diagram of the structure of the paper.

\begin{figure}[h]
\centering
\begin{tikzpicture}[>=stealth]
\node[] at (1.5, 0) {\small \ref{sec:Introduction}};
\node[] at (3, 0) {\small \ref{sec:Preliminaries}};
\node[] at (4.5, 0) {\small \ref{sec:SymbolicModel}};
\node[] at (6, 0.3) {\small \ref{sec:TransferOperators}};
\node[] at (6, -0.3) {\small \ref{sec:LNIC}};
\node[] at (6, -0.9) {\small \ref{sec:NCP}};
\node[] at (7.5, 0.3) {\small \ref{sec:LDP}};
\node[] at (9, 0) {\small \ref{sec:Dolgopyat'sMethod}};
\node[] at (10.5, 0) {\small \ref{sec:ExponentialMixing}};
\node[] at (6, 0.9) {\small \ref{sec:appendix}};

\begin{scope}[decoration={markings, mark=at position 0.5 with {\arrow{>}}}] 
\draw[black,postaction={decorate}] (3.15, 0) to (4.35, 0);
\draw[black,postaction={decorate}] (4.65, 0) to[out=0, in=180] (5.85, 0.3);
\draw[black,postaction={decorate}] (4.65, 0) to[out=0, in=180] (5.85, -0.3);
\draw[black,postaction={decorate}] (4.65, 0) to[out=0, in=180] (5.85, -0.9);
\draw[black,dashed,postaction={decorate}] (6.15, 0.3) to (7.35, 0.3);
\draw[black,postaction={decorate}] (7.65, 0.3) to[out=0, in=180] (8.85, 0);
\draw[black,postaction={decorate}] (6.15, -0.3) to[out=0, in=180] (8.85, 0);
\draw[black,postaction={decorate}] (6.15, -0.9) to[out=0, in=180] (8.85, 0);
\draw[black,dashed,postaction={decorate}] (9.15, 0) to (10.35, 0);
\draw[black,dashed,postaction={decorate}] (6.15, 0.9) to[out=0, in=180] (7.35, 0.3);
%\draw[black,postaction={decorate}] (6.15, 0.3) to[out=30, in=150] (10.35, 0);
\draw[black,postaction={decorate}] (6.15, 0.3) to[out=0, in=180] (8.25, 0.7) to[out=0, in=180] (10.35, 0);
\end{scope}
\end{tikzpicture}
\caption{This diagram shows the dependence between sections. The solid (resp. dashed) lines are strong (resp. weak) dependence.}
\label{fig:Sections}
\end{figure}
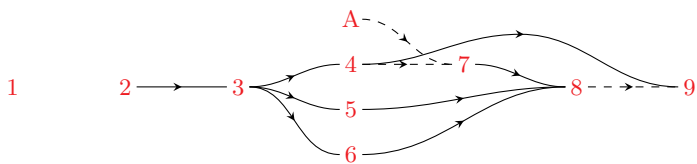

\subsection*{Acknowledgements} We would like to thank Zhiyuan Zhang for explaining to us the paper \cite{TZ23}. Part of this work was done while the authors attended the conference ``Ergodic geometry, number theory and Margulis legacy: the next generation'' in June 2022. We would like to thank the organizers for providing the opportunity to collaborate.

\section{Preliminaries}
\label{sec:Preliminaries}
\subsection{Hyperbolic spaces}
\label{subsec:hyperbolic spaces}
Let $\HS$ be the $(d + 1)$-dimensional hyperbolic space for any $d \in \N$ endowed with the Riemannian hyperbolic metric. We will often use the upper half space model:
\begin{equation*}
\HS=\{x=(x_1,\ldots,x_{d+1})\in \mathbb{R}^{d+1}: x_{d+1}>0\}.
\end{equation*}
Any complete connected $(d + 1)$-dimensional hyperbolic manifold is then of the form $\Gamma \backslash \HS$ for some torsion-free discrete subgroup $\Gamma < \SO(d+1,1)^{\circ}$ where the latter can be identified with the group of orientation-preserving isometries of $\HS$. We denote its unit tangent bundle by $\T(\Gamma \backslash \HS)$ and its (oriented orthonormal) frame bundle by $\F(\Gamma \backslash \HS)$.

\subsubsection{\texorpdfstring{Identifying $\HS$, $\T(\HS)$, and $\F(\HS)$ with homogeneous spaces}{Identifying ℍᵈ⁺¹, T¹(ℍᵈ⁺¹), and F(ℍᵈ⁺¹) with homogeneous spaces}}
Let $(e_1, \dotsc, e_{d+1})$ be the standard basis of $\R^{d+1}$. Fix the reference point $o := e_{d+1} \in \HS$, the reference vector $v_o := (e_{d+1}, -e_{d+1}) \in \operatorname{T}_o^1(\HS)$, and the reference frame $F_o := ((e_{d+1}, e_1), \dotsc, (e_{d+1}, e_{d-1}), (e_{d+1}, e_d), \allowbreak (e_{d+1}, -e_{d+1})) \in \F_o(\HS)$. Let $G = \SO(d+1,1)^{\circ}$ which we endow with a left $G$-invariant and right $K$-invariant Riemannian metric. We denote by $\langle \cdot, \cdot\rangle$ the inner product and by $\|\cdot\|$ the norm on any tangent space of $G$ or its quotient spaces. Similarly, we denote by $d$ the induced distance function on $G$ or its quotient spaces. Let $K:=\Stab_G(o)$ and $M:=\Stab_G(v_o) < K$. Note that $\Stab_G(F_o)$ is trivial. We can assume that the Riemannian metric on $G$ was chosen so that we have identifications via the isomorphisms $\HS \cong G/K$, $\T(\HS) \cong G/M$, and $\F(\HS) \cong G$. For any torsion-free discrete subgroup $\Gamma < G$, we can extend these identifications to the quotient spaces: $\Gamma\backslash \HS \cong \Gamma\backslash G/K$, $\T(\Gamma\backslash \HS) \cong \Gamma\backslash G/M$, and $\F(\Gamma\backslash \HS) \cong \Gamma\backslash G$. In particular, the Riemannian metric on $G$ descends to the Riemannian hyperbolic metric on $\HS$ and $\Gamma\backslash \HS$, and $d$ denotes the induced hyperbolic distance function on $\HS$ and $\Gamma\backslash \HS$.

Let $A=\{a_t:t\in \mathbb{R}\} < C_G(M)$ be the one-parameter subgroup of semisimple elements such that for any torsion-free discrete subgroup $\Gamma < G$, the geodesic flow $\{\mathcal{G}_t\}_{t \in \R}$ on $\T(\Gamma\backslash \HS)$ and the frame flow $\{\mathcal{F}_t\}_{t \in \R}$ on $\F(\Gamma\backslash \HS)$ are represented by the right translation action of $A$ on $\Gamma\backslash G/M$ and $\Gamma\backslash G$, respectively. Exploiting the above identifications, for any $u \in \T(\Gamma\backslash \HS)$ and $F \in \F(\Gamma\backslash \HS)$, we will often write the application of the geodesic flow and the frame flow for time $t \in \R$ as $ua_t$ and $Fa_t$, respectively. We denote by $N^+$ and $N^-$ the associated unstable and stable horospherical subgroups, respectively:
\begin{equation*}
N^{\pm}=\{g\in G: a_t ga_{-t}\to e \text{ as } t\to \pm\infty\}.
\end{equation*}
They are abelian groups isomorphic to $\mathbb{R}^d$.

\subsubsection{Visual boundary}
Let $\partial \HS$ be the visual boundary of $\HS$. The compactification of $\HS$ is $\overline{\HS} = \HS \cup \partial \HS$. On $\partial \HS=\mathbb{R}^d\cup \{\infty\} \cong \mathbb{S}^d$, we have the spherical metric $d_{\mathbb{S}^d}$. We also have the (extended) Euclidean metric $d_{\operatorname{E}}$ defined by $d_{\operatorname{E}}(x,x') = \|x-x'\|$ for any $x,x'\in \partial\HS$. 

The $G$-action on $\HS$ induces a $G$-action on $\partial \HS$ by conformal maps. For $x\in \partial \HS$, let 
$\|(dg)_x\|$ be the operator norm of the differential $(dg)_x:\operatorname{T}_x(\partial \HS)\to \operatorname{T}_x (\partial \HS)$ with respect to the Euclidean metric. Let $\|(dg)_x\|_{\mathbb{S}^d}$ be the operator norm with respect to the spherical metric. We have the relations
\begin{equation*}
\|(dg)_x\|_{\mathbb{S}^d}=\frac{1+\|x\|^2}{1+\|gx\|^2}\|(dg)_x\| = e^{-\beta_x(g^{-1}o,o)},
\end{equation*}
where $\beta: \partial \HS \times \HS \times \HS \to \R$ is the Busemann function given by 
$\beta_x(z,z')=\lim_{t\to \infty}(d(z,x_t)-d(z',x_t))$
with $\{x_t \in \HS: t>0\}$ an arbitrary geodesic ray tending to $x \in \partial \HS$.

\subsubsection{\texorpdfstring{Hopf parametrization of $\T(\HS)$}{Hopf parametrization of T¹(ℍᵈ⁺¹)}}
For all $v \in \T(\HS) \cong G/M$ and $F \in \F(\HS) \cong G$, denote by $v^+, F^+ \in \partial \HS$ (resp. $v^-, F^- \in \partial \HS$) their forward endpoints (resp. backward endpoints), and denote by $v_*, F_* \in \HS$ their basepoints. Set $\partial^2(\HS)=\partial \HS\times \partial \HS -\operatorname{Diagonal}(\partial \HS)$. The Hopf parametrization is the identification via the diffeomorphism
\begin{align}
\nonumber
\T(\HS) &\to \partial^2(\HS)\times \mathbb{R} \\
\label{eqn:Hopf parametrization}
v &\mapsto (v^+,v^-,s=\beta_{v^+}(o,v_*)).
\end{align}
The geodesic flow in the Hopf parametrization is simply the translation action on the $\mathbb{R}$-coordinate.

Given a torsion-free discrete subgroup $\Gamma < G$, it acts on $\T(\HS)$, and hence it acts on $\partial^2(\HS)\times \mathbb{R}$, which is given by the formula:
\begin{equation}
\label{action on the boundary}
\gamma(v^+,v^-,s)=(\gamma v^+,\gamma v^-, s-\beta_{v^+}(o,\gamma^{-1}o)).
\end{equation}

\subsection{Geometrically finite subgroups}
We now cover some fundamentals of geometrically finite subgroups. Let $\Gamma < G$ be a discrete subgroup.

The limit set of $\Gamma$ is the set $\Lambda_{\Gamma} \subset \overline{\HS}$ of all limit points of the orbit $\Gamma o \subset \overline{\HS}$. It is independent of $o \in \HS$ and discreteness of $\Gamma$ implies $\Lambda_{\Gamma} \subset \partial \HS$.

A limit point $x\in \Lambda_{\Gamma}$ is {\it{conical}} if there exists a geodesic ray tending to $x$ and a sequence of elements $\{\gamma_n\}_{n \in \N} \in \Gamma$ such that $\lim_{n \to \infty} \gamma_n o = x$, and the distance between $\gamma_n o$ and the geodesic ray is bounded for all $n \in \N$. A subgroup $\Gamma' < \Gamma$ is {\it{parabolic}} if it fixes only one point in $\partial \HS$. A point $x\in \Lambda_{\Gamma}$ is called a {\it{parabolic fixed point}} if $\Stab_{\Gamma}(x)$ is parabolic. It is said to be {\it{bounded}} if the quotient $\Stab_{\Gamma}(x)\backslash (\Lambda_{\Gamma}-\{x\})$ is compact.

The subgroup $\Gamma < G$ is {\it{non-elementary}} if $\#\Lambda_{\Gamma} \geq 3$ and hence $\#\Lambda_{\Gamma} = \infty$. Moreover, such a subgroup is {\it{geometrically finite}} if $\Lambda_\Gamma$ consists only of conical limit points and parabolic fixed points, or equivalently, $\Lambda_\Gamma$ consists only of conical limit points and bounded parabolic fixed points \cite{Bow93,KL19}.

\subsection{Structure of cusps}
\label{subsec:StructureOfCusps}
Let $\Gamma < G$ be a geometrically finite subgroup with parabolic elements and suppose $\infty \in \partial\HS$ is a parabolic fixed point of $\Gamma$. Let $\Gamma'_{\infty}=\Stab_{\Gamma}(\infty) <\Gamma$ be the parabolic subgroup fixing $\infty$. Then $\Gamma'_{\infty}$ acts on $\mathbb{R}^d \subset \partial \HS$ by Euclidean isometries. The following result of Bieberbach gives a more precise characterization of the $\Gamma'_{\infty}$-action on $\mathbb{R}^d$ (see \cite[Section 7.5]{Rat19}).

\begin{lemma}[Bieberbach]
\label{lem:Bieberbach}
There exist a maximal normal abelian subgroup $\Gamma_{\infty} < \Gamma'_{\infty}$ of finite index $m \in \N$ and a $\Gamma'_{\infty}$-invariant affine subspace $Z\subset \mathbb{R}^d$ of dimension $k \in \N$ such that $\Gamma_{\infty}$ acts on $Z$ as a group of translations of rank $k$. Consequently, decomposing $\mathbb{R}^d = Y \oplus Z$ into orthogonal affine subspaces and viewing the later as vector spaces in their own right, we can write the $\Gamma'_{\infty}$-action on $\mathbb{R}^d$ in the following form: for all $\gamma \in \Gamma'_{\infty}$, there exist
\begin{align*}
A_{\gamma} &\in \operatorname{O}(Y), & R_{\gamma} &\in \operatorname{O}(Z), & b_{\gamma} &\in Z,
\end{align*}
where $R_{\gamma}^m = \Id$ and moreover $R_{\gamma} = \Id$ if $\gamma \in \Gamma_{\infty}$, such that
\begin{equation*}
\gamma(y,z)=(A_{\gamma} y, R_{\gamma}z+b_{\gamma}) \qquad \text{for all $(y, z) \in Y \oplus Z$}.
\end{equation*}
\end{lemma}

The dimension $k$ in \cref{lem:Bieberbach} is also called the {\it{rank}} of the parabolic fixed point $\infty$.

Consider the orthogonal decomposition $\mathbb{R}^d=Y \oplus Z$ from \cref{lem:Bieberbach}. Since $\Gamma_{\infty}$ acts on $Z$ as a group of translations, it admits a fundamental domain $\Delta'_{\infty} \subset Z$ which is an open $k$-dimensional parallelotope. Since $\Gamma$ is geometrically finite, $\infty$ is a bounded parabolic fixed point. By definition, the quotient $\Gamma'_{\infty}\backslash (\Lambda_{\Gamma}-\{\infty\})$ is compact and so the quotient $\Gamma_{\infty}\backslash (\Lambda_{\Gamma}-\{\infty\})$ is also compact as $\Gamma_{\infty} < \Gamma'_{\infty}$ is of finite index. Therefore, there exists a constant $C_\infty>0$ such that the set $B_{Y}(C_\infty)=\{y\in Y: \|y\|<C_\infty\}$ has the property that 
\begin{equation}
\label{limit set inclusion}
\Lambda_{\Gamma} - \{\infty\} \subset \bigcup_{\gamma\in \Gamma_{\infty}} \gamma \bigl(\overline{B_{Y}(C_\infty)\times \Delta_{\infty}'}\bigr).
\end{equation}
We call the open set $\Delta_{\infty}:=B_Y(C_\infty)\times \Delta'_{\infty}$ a {\it{fundamental domain}} for the parabolic fixed point $\infty$.

\begin{figure}[h]
\begin{center}
\def\svgwidth{10cm}
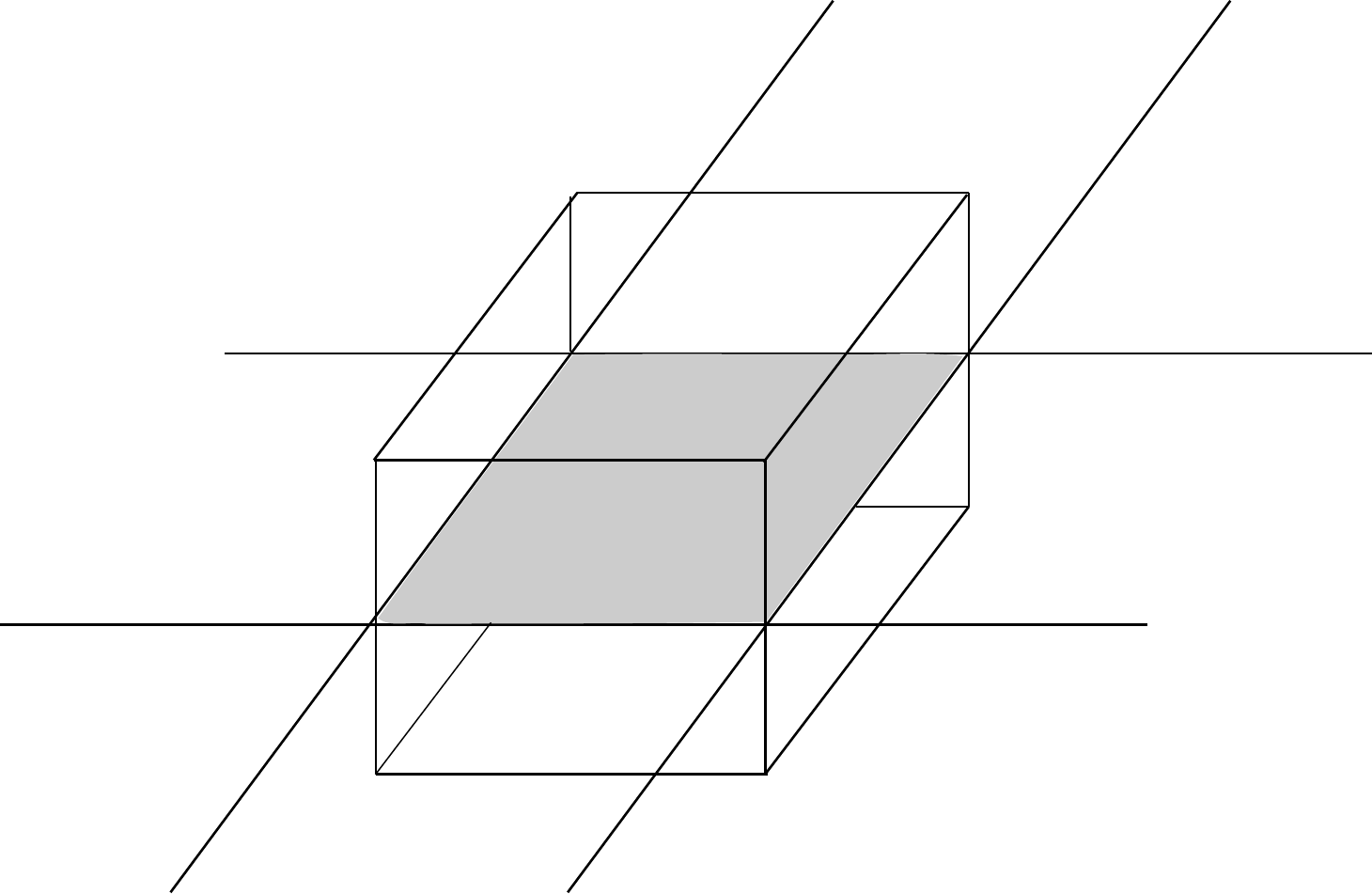
\end{center}
\caption{The intersection $\Lambda_{\Gamma}\cap \mathbb{R}^d$ has bounded distance to $\mathbb{R}^k$.}
\end{figure}

\subsection{PS measure and BMS measure}
\label{subsec:PS measure and BMS measure}
Let $\Gamma < G$ be a geometrically finite subgroup. Denote by $\delta \in (0, d]$ the critical exponent of $\Gamma$, i.e., the abscissa of convergence of the Poincar\'{e} series $s \mapsto \sum_{\gamma\in \Gamma}e^{-sd(o,\gamma o)}$. 

\subsubsection{Patterson--Sullivan measure}
Patterson and later Sullivan constructed a $\Gamma$-invariant conformal density $\bigl\{\mu_y^{\mathrm{PS}}\bigr\}_{y\in \HS}$ of dimension $\delta$, i.e., a set of finite Borel measures on $\partial \HS$ such that for any $y,z\in \HS$, $x\in \partial \HS$, and $\gamma\in \Gamma$, we have
\begin{equation}
\label{PS measure}
\frac{d\mu_{y}^{\mathrm{PS}}}{d\mu_{z}^{\mathrm{PS}}}(x)=e^{-\delta \beta_x(y,z)} \qquad \text{and} \qquad \gamma_{*} \mu_y^{\mathrm{PS}}=\mu_{\gamma y}^{\mathrm{PS}},
\end{equation}
where the pushforward is defined by $\gamma_{*}\mu_y^{\mathrm{PS}}(E)=\mu_y^{\mathrm{PS}}(\gamma^{-1}E)$ for any Borel subset $E \subset \partial \HS$. Such a set of measures is unique up to homothety and the $\Gamma$-action on $\partial \HS$ is ergodic relative to the corresponding measure class; in particular, the measures are supported on $\Lambda_{\Gamma} \subset \partial \HS$. These measures are called Patterson--Sullivan measures (or PS measures).

Since the PS measures are absolutely continuous with respect to each other, it is convenient to fix $\mu := \mu_o^{\mathrm{PS}}$ for the rest of the paper. It enjoys the {\it{quasi-invariance}} property: for any Borel subset $E \subset \partial \HS$ and any $\gamma\in \Gamma$,
\begin{equation}
\label{eqn:quasi invariance}
\mu(\gamma E)=\int_{E} \|(d\gamma)_x\|_{\mathbb{S}^d}^{\delta} \, d\mu(x).
\end{equation}

\subsubsection{Bowen--Margulis--Sullivan measure}
Using the Hopf parametrization, the Bowen--Margulis--Sullivan measure (or BMS measure) on $\T(\HS)$ is defined by
\begin{equation*}
d\tilde{m}^{\operatorname{BMS}}(v^+,v^-,s)=e^{\delta \beta_{v^+}(o,v_*)} e^{\delta \beta_{v^-}(o,v_*)} \, d\mu(v^+) \, d\mu(v^-) \, ds.
\end{equation*}
It induces the following measures which are also called BMS measures. Note that $\tilde{m}^{\operatorname{BMS}}$ is left $\Gamma$-invariant because of the conformality of the PS measure (\cref{PS measure}) and the formula for $\Gamma$-action on $\partial^2(\HS)\times \mathbb{R}$ (\cref{action on the boundary}). Hence $\tilde{m}^{\operatorname{BMS}}$ induces a measure $m^{\operatorname{BMS}}$ on $\T(\Gamma\backslash \HS)$. Since $\Gamma$ is geometrically finite, $m^{\operatorname{BMS}}(\T(\Gamma\backslash \HS))$ is finite and so we may normalize $m^{\operatorname{BMS}}$ to a probability measure. Using the identifications $\T(\HS)\cong G/M$ and $\F(\HS)\cong G$, and the Haar probability measure on $M$, we lift the measure $\tilde{m}^{\operatorname{BMS}}$ to a right $M$-invariant measure on $\F(\HS)$, which we also denote by $\tilde{m}^{\operatorname{BMS}}$ by abuse of notation. Similarly, for the quotient spaces, we lift the measure $m^{\operatorname{BMS}}$ on $\T(\Gamma\backslash \HS)\cong \Gamma\backslash G/M$ to a right $M$-invariant measure $m^{\operatorname{BMS}}$ on $\F(\Gamma\backslash \HS)\cong \Gamma\backslash G$. It follows from definitions that the BMS measures are invariant under the geodesic flow $\{\mathcal{G}_t\}_{t \in \R}$ on the respective unit tangent bundles and invariant under the frame flow $\{\mathcal{F}_t\}_{t \in \R}$ on the respective frame bundles.

\section{Symbolic model for the frame flow}
\label{sec:SymbolicModel}
For the rest of the paper, let $\Gamma<G$ be a Zariski dense torsion-free geometrically finite subgroup with parabolic elements. Furthermore, conjugating $\Gamma$ by an element in $G$ if necessary, we can assume that $\infty \in \partial \HS$ is a parabolic fixed point of $\Gamma$. Denote $X := \Gamma \backslash \HS$.

\subsection{Expanding map on the boundary}
\label{subsec:expanding map}
Let us first recall the symbolic model for the geodesic flow constructed by Li--Pan \cite[Proposition 4.1]{LP22}.

One approach in symbolic dynamics to tackle the geodesic flow is to construct a Poincar\'{e} section $S \subset \T(X)$. It is a $2d$-dimensional submanifold with boundary which is transverse to the geodesic flow and has good recurrence properties. The challenge lies in finding a return time map $\Ret$ defined on a full measure subset $S_0 \subset S$, such that the map $v \mapsto \mathcal{G}_{\Ret(v)}(v)$ on $S_0$ is hyperbolic and can be modeled by a full shift of a countably infinite alphabet. To overcome this difficulty, Li--Pan constructed a countably infinite symbolic coding on $\partial \HS$ equipped with an expanding map and connected it to the return time map.

Fix $\Delta_0 := \Delta_\infty$ to be the fundamental domain for the parabolic fixed point $\infty$.

\begin{proposition}[{\cite[Proposition 4.1]{LP22}}]
\label{prop:Coding}
There exist $C_1 > 0$, $\lambda \in (0, 1)$, $\epsilon_0 \in (0, 1)$, a countably infinite set $\{\Delta_j\}_{j \in \mathcal{J}}$ of mutually disjoint open subsets of $\Delta_0$, and an expanding map $T: \Delta_{\sqcup} \to \Delta_0$ where $\Delta_{\sqcup} := \bigsqcup_{j \in \mathcal{J}} \Delta_j$, such that:
\begin{enumerate}
\item $\mu(\Delta_0) = \sum_{j \in \mathcal{J}} \mu(\Delta_j)$;
\item for all $j \in \mathcal{J}$, there exists $\gamma_j \in \Gamma$ such that $\Delta_j = \gamma_j \Delta_0$ and $T|_{\Delta_j} = \gamma_j^{-1}$;
\item for all $j \in \mathcal{J}$, $\gamma_j$ is a uniform contraction: we have $\|(d\gamma_j)_x\|\leq \lambda$ for all $x \in \Delta_0$;
\item for all $j \in \mathcal{J}$, we have $\|\left(d(\log\|d\gamma_j\|)\right)_x\|\leq C_1$ for all $x \in \Delta_0$, where $\left(d(\log\|d\gamma_j\|)\right)_x$ is the derivative of the map $z\mapsto \log \|(d\gamma_j)_z\|$ at $x$;
\item the return time map $\Ret: \Delta_{\sqcup} \to \R$ defined by $\Ret(x) = \log\|(dT)_x\|$ for all $x \in \Delta_{\sqcup}$ satisfies the exponential tail property: $e^{\epsilon_0 \Ret} \in L^1(\Delta_{\sqcup}, \mu)$.
\end{enumerate}
\end{proposition}

Denote by $\calH=\{\gamma_j\}_{j\in \mathcal{J}}$ the set of inverse branches of $T$. Let
\begin{equation*}
\calH^n=\{\gamma_{j_1}\cdots \gamma_{j_n}: j_k \in \mathcal{J} \text{ for all } 1\leq k\leq n\}
\end{equation*}
be the set of inverse branches of \emph{length} $n \in \N$. For any $\gamma\in \bigcup_{n \in \N}\calH^n$, set
\[\|d\gamma\|:=\sup\{\|(d\gamma)_x\|:x\in \Delta_0\}.\]
Define
\begin{equation*}
\Lambda_+ = \Lambda_{\Gamma} \cap \bigcap_{n \in \N} T^{-n}(\Delta_0) = \{x\in \Lambda_{\Gamma}: T^n \text{ is defined at } x \text{ for all } n \in \mathbb{N}\}.
\end{equation*}
This is the space that admits a countably infinite coding.

The expanding map $T$ gives a contraction action in a neighborhood of $\infty$.
\begin{lemma}[{\cite[Lemma 4.8]{LP22}}]
\label{lambda-}
There exist $0<\lambda<1$ and a neighborhood $\Lambda_-$ of $\infty$ in $\Lambda_{\Gamma}$ such that $\Lambda_{-}$ is disjoint from $\overline{\Delta_0}$ and for any $\gamma\in \calH$ and any $y,y'\in \Lambda_-$,
\begin{equation*}
\gamma^{-1}(\Lambda_-)\subset \Lambda_-,\,\,\,d_{\mathbb{S}^d}(\gamma^{-1}y, \gamma^{-1}y')\leq \lambda d_{\mathbb{S}^d}(y,y').
\end{equation*}
\end{lemma}

Using \cref{prop:Coding}, it can be shown as in \cite[Lemma 2]{You98} that there exists a $T$-invariant ergodic probability measure $\nu$ on $\Delta_0$ such that
\begin{equation}
\label{eqn:nu_and_mu}
d\nu = f_0 \, d\mu
\end{equation}
where $f_0: \Delta_0 \to \R$ is a positive Lipschitz density function which is bounded away from $0$ and $+\infty$. Note that as a result, $\Lambda_+ \subset \Delta_0$ is a full measure subset with respect to the equivalent measures $\mu$ and $\nu$.

\begin{definition}[Cylinder]
We call subsets of the form $\gamma\Delta_0 \subset \Delta_0$ for some $\gamma \in \calH^n$ and $n \in \N$ \emph{cylinders} of \emph{length} $n$. We regard $\Delta_0$ as a cylinder of length $0$. These will typically be denoted by $\mathtt{C}$ or other typewriter style letters.
\end{definition}
In \cref{good partitions}, we will use cylinders to construct $\mu$-measurable partitions of $\Delta_0$. 

Following \cite[Section 5]{Sto11}, we introduce another distance function $D$ on $\Delta_0$ by
\begin{align*}
D(x, y) = \min \{\diam(\mathtt{C}): \text{cylinders } \mathtt{C} \subset \Delta_0 \text{ such that } x,y\in\overline{\mathtt{C}}\}
\end{align*}
for $x, y \in \Delta_0$ with $x\neq y$, and $D(x,y)=0$ otherwise. The proof of the following lemma is straightforward.

\begin{lemma}
The following holds:
\begin{enumerate}
\item $D$ is a distance function on $\Delta_0$;
\item for all $x,y\in \Delta_0$, we have $d_{\mathrm{E}}(x,y)\leq D(x,y)$;
\item for all $n\in \mathbb{N}$, $\gamma\in \calH^n$, and $x,y\in \Delta_0$, we have $D(\gamma x,\gamma y)\leq \lambda^n D(x,y)$.
\end{enumerate}
\end{lemma}

The following lemma records basic estimates for cylinders which is used throughout the paper. The estimate for measures holds more generally for Borel subsets.
\begin{lemma}
\label{lem:CylinderEstimate}
For all cylinders $\mathtt{C} \subset \Delta_0$, Borel subsets $E \subset \Delta_0$, and $\gamma \in \bigcup_{n\in \mathbb{N}}\calH^n$, we have
\begin{align*}
\diam(\gamma\mathtt{C}) &\asymp \|d\gamma\|\diam(\mathtt{C}), & \mu(\gamma E) &\asymp \|d\gamma\|^{\delta} \mu(E), & \nu(\gamma E) &\asymp \|d\gamma\|^{\delta} \nu(E),
\end{align*}
with some implicit constant $C_{\mathrm{cyl}} > 1$.
\end{lemma}

\begin{proof}
%The lemma follows from Properties (3) and (4) in \cref{prop:Coding}.
Let $\mathtt{C}$ and $E$ be as in the lemma and $\gamma=\gamma_{j_1}\cdots \gamma_{j_n} \in \calH^n$. We first estimate $\mu(\gamma E)$ and $\nu(\gamma E)$. Using Properties~(3) and (4) in \cref{prop:Coding} and the chain rule, for any $x,y \in \Delta_0$, we have
\begin{align*}
\left|\log \|(d\gamma)_x \|-\log \|(d\gamma)_y \|\right| &\leq \sum_{l = 1}^{n - 1} C_1 d_{\mathrm{E}}(\gamma_{j_{l+1}}\cdots \gamma_{j_n} x,\gamma_{j_{l+1}}\cdots \gamma_{j_n} y) \\
&\leq \frac{C_1}{1-\lambda}\diam(\Delta_0).
\end{align*}
Then for any $x\in\Delta_0$, we have
\begin{equation}\label{equ:d gamma x}
\|(d \gamma)_x\|\asymp\|d\gamma\|. 
\end{equation}
Due to \cref{eqn:quasi invariance,eqn:nu_and_mu}, for any Borel subset $E \subset \Delta_0$, we have
\[\nu(\gamma E)\asymp \mu(\gamma E)\asymp \mu(E)\|d\gamma\|^{\delta}. \]

Next, we estimate $\diam(\gamma \mathtt{C})$. We first show the latter part of the inequality. For any $y\in \gamma\Delta_0$, by writing $y=\gamma x$, we have
\begin{equation}\label{eq:gammainverse}
\|(d\gamma^{-1})_y\|=\|(d\gamma)_x\|^{-1}\asymp \|\gamma\|^{-1}.
\end{equation}
Note that for any points $x,y \in \mathtt{C}$, the straight line $\overline{xy}$ is contained in $\Delta_0$ by convexity. Now, $d_{\mathrm{E}}(\gamma x,\gamma y)$ is bounded by the length of the curve $\gamma \cdot \overline{xy}$. Hence, using the inequality  $\|(d \gamma)_z\|\leq \|d\gamma\|$ for all $z \in \Delta_0$ gives
\[ d_{\mathrm{E}}(\gamma x,\gamma y)\leq \|d\gamma\|d_{\mathrm{E}}(x,y). \]
By taking the supremum of $d(\gamma x,\gamma y)$ over all points $x,y \in \mathtt{C}$, we obtain
\[\diam(\gamma\mathtt{C})\leq \|d\gamma\|\diam(\mathtt{C}). \]

Now, we show the former part of the inequality. Fix any ball $B \subset \Delta_0$. We write $\mathtt{C}=\beta\Delta_0$ for some $\beta\in\calH^m$, where $m$ is the length of the cylinder. The images $\gamma\beta B \subset \beta B \subset \Delta_0$ are still balls due to the conformality of the $G$-action on $\partial \HS$. Thus, using convexity of the balls and \cref{eq:gammainverse} for $\gamma^{-1}$ and $\beta^{-1}$, we can apply a similar argument as in the previous paragraph to obtain
\[\diam(B)\ll \frac{\diam(\gamma\beta B)}{\|d\gamma\|\cdot\|d\beta\|} \leq \frac{\diam(\gamma\mathtt{C})}{\|d\gamma\|\cdot\|d\beta\|}  .  \]
From the previous paragraph, we also have
\[\diam(\mathtt{C})=\diam(\beta\Delta_0)\leq \|d\beta\|\diam(\Delta_0) \ll \|d\beta\|\diam(B). \]
Combining these two inequalities, we obtain
\[\|d\gamma\| \diam(\mathtt{C})\ll\diam(\gamma\mathtt{C}). \]
\end{proof}

Fix a constant
\begin{align}
\label{eqn:ConstantcDelta0}
C_{\Delta_0} \geq C_{\mathrm{cyl}} \cdot \max\{1, \diam(\Delta_0), \mu(\Delta_0)\} > 1.
\end{align}

\subsection{Symbolic model for the geodesic flow}
Recall the subsets $\Lambda_+,\Lambda_- \subset \Lambda_{\Gamma}$ from \cref{subsec:expanding map}.
Define $\hat{T}: \Lambda_+ \times \Lambda_- \to \Lambda_+ \times \Lambda_-$ by
\begin{equation*}
\hat{T}(x,y) = \bigl(\gamma_j^{-1}x,\gamma_j^{-1}y\bigr) \qquad \text{if $x\in \Delta_j$, for all $y \in \Lambda_+$ and $j \in \mathcal{J}$}.
\end{equation*}

Let $\Ret:\Delta_{\sqcup}\to \mathbb{R}$ be the function given in \cref{prop:Coding}. By abusing notation, define $\Ret: \Delta_{\sqcup} \times \Lambda_- \to \R$ by setting $\Ret(x,y)=\Ret(x)$ for all $(x,y) \in \Delta_{\sqcup} \times \Lambda_-$. We define the space
\begin{equation*}
\Lambda^\Ret=\{(x,y,s)\in \Lambda_+\times \Lambda_-\times \R:0\leq s<\Ret(x, y)\}.
\end{equation*}
We use the notation $\Ret_0 = 0$ and
\begin{align}
\label{eqn:BirkhoffSum}
\Ret_n=\sum_{j=0}^{n-1}\Ret\circ \hat{T}^j \qquad \text{for all $n \in \N$}.
\end{align}
The hyperbolic skew product semiflow $\{\hat{T}_t\}_{t\geq 0}$ is defined by
\begin{equation*}
\hat{T}_t(x,y,s)=(\hat{T}^n(x,y),s+t-\Ret_n(x, y))
\end{equation*}
for all $(x,y,s) \in \Lambda_+\times \Lambda_- \times \R$, where $n \in \Z_{\geq 0}$ such that $0\leq s+t-\Ret_n(x,y)<\Ret(\hat{T}^n(x,y))$. 

\subsubsection{\texorpdfstring{The map from $\Lambda^\Ret$ to $\T(X)$}{The map from Λᴿ to T¹(X)}}
We introduce the following embedding:
\begin{align}
\label{eqn:embedding}
\tilde{\Phi}: \Delta_0\times \Lambda_-\times \{0\}&\to \partial^2(\HS)\times \R\\
\nonumber
(x,y,0)&\mapsto (x,y,\log(1+\|x\|^2)).
\end{align}
We can see from the formula of the Hopf parametrization (\cref{eqn:Hopf parametrization}) that $\tilde{\Phi}$ maps $\Delta_0\times \{\infty\}\times \{0\}$ to the unstable horosphere based at $\infty$ which contains $o=e_{d + 1} \in \HS \subset \R^{d + 1}$. 
Abusing notation, we define the following time change map:
\begin{align*}
\tilde{\Phi}: \Lambda^\Ret&\to \partial^2(\HS)\times \R\\
(x,y,s)&\mapsto (x,y,s+\log(1+\|x\|^2)).
\end{align*}
For all $x \in \Lambda_+$ and $y,y'\in \Lambda_-$, the points $\tilde{\Phi}(x,y,0)$ and $\tilde{\Phi}(x,y',0)$ lie on the same stable horosphere based at $x$.

Using the identification $\T(X) \cong \Gamma\backslash (\partial^2(\HS)\times \R)$,  the map $\tilde{\Phi}$ induces a map $\Phi:\Lambda^\Ret\to \T(X)$. We have that $\Phi$ defines a semiconjugacy between two semiflows:
\begin{equation*}
\Phi\circ \hat{T}_t=\mathcal{G}_t\circ \Phi \qquad \text{for all $t \geq 0$}.
\end{equation*}

\subsection{Symbolic model for the frame flow}
\label{subsec:SymbolicModelForFrameFlows}
Fix a reference point $x_0\in \Lambda_+ \subset \Delta_0$. We define a section
\begin{equation*}
\tilde{F}: \Delta_0 \times \Lambda_- \to \F(\HS)
\end{equation*}
which is \emph{smooth} in the first argument in the following fashion: 
\begin{itemize}
\item Fix a frame $\tilde{F}(x_0,\infty) \in \F(\HS)$ based at the tangent vector $\tilde{\Phi}(x_0, \infty, 0)$.
\item Extend the section $\tilde{F}$ such that for any $x,x' \in \Delta_0$, the frames $\tilde{F}(x,\infty)$ and $\tilde{F}(x',\infty)$ are backward asymptotic, i.e.,
\begin{equation*}
\lim_{t\to -\infty} d(\tilde{F}(x,\infty)a_t,\tilde{F}(x',\infty)a_t)=0.
\end{equation*}
Then, we must have $\tilde{F}(x',\infty)=\tilde{F}(x,\infty)n^+$ for some unique $n^+\in N^+$.
\item Extend the section $\tilde{F}$ such that for any $x \in \Delta_0$ and $y,y'\in \Lambda_-$, the frames $\tilde{F}(x,y)$ and $\tilde{F}(x,y')$ are forward asymptotic, i.e.,
\begin{equation*}
\lim_{t\to \infty}d(\tilde{F}(x,y)a_t,\tilde{F}(x,y')a_t)=0.
\end{equation*}
Then, we must have $\tilde{F}(x,y')=\tilde{F}(x,y)n^-$ for some unique $n^-\in N^-$.
\end{itemize}
Then, $\tilde{F}$ induces a map $F: \Delta_0 \times \Lambda_- \to \F(X)$.
\begin{definition}[Holonomy]
\label{holonomy}
The \emph{holonomy} is a map $\Hol: \Delta_{\sqcup} \times \Lambda_- \to M$ such that for all $(x,y) \in \Delta_{\sqcup} \times \Lambda_-$, we have
\begin{equation*}
F(x,y)a_{\Ret(x,y)}=F\bigl(\hat{T}(x,y)\bigr)\Hol(x,y)^{-1}.
\end{equation*}
\end{definition}

\begin{definition}[Generalized holonomy]
We call the combined map
\begin{align*}
\GHol: \Delta_{\sqcup} \times \Lambda_- &\to AM \\
(x, y) &\mapsto a_{\Ret(x, y)} \Hol(x, y)
\end{align*}
the \emph{generalized holonomy}.
\end{definition}

The following lemma can be deduced using the construction of $F$ (see \cite[Lemma 4.2]{SW21} for its proof).

\begin{lemma}
\label{lem:stable direction}
For all $x \in \Delta_{\sqcup}$, the maps $\Ret|_{\{x\} \times \Lambda_-}$, $\Hol|_{\{x\} \times \Lambda_-}$, and $\GHol|_{\{x\} \times \Lambda_-}$ are constant.
\end{lemma}

We define the space
\begin{equation*}
\Lambda^\Ret\times M=\{(x,y,m,s)\in \Lambda_+\times \Lambda_-\times M\times \R:0\leq s<\Ret(x,y)\}.
\end{equation*}
We use the notations $\Hol_0 = \GHol_0 = e$ and
\begin{align*}
\begin{aligned}
\Hol_n &= \Hol \cdot \bigl(\Hol\circ \hat{T}\bigr) \cdots \bigl(\Hol\circ\hat{T}^{n-1}\bigr), \\
\GHol_n &= \GHol \cdot \bigl(\GHol\circ \hat{T}\bigr) \cdots \bigl(\GHol\circ\hat{T}^{n-1}\bigr),
\end{aligned}
\qquad \text{for all $n \in \N$}.
\end{align*}
The symbolic frame semiflow $\{\hat{T}_t\}_{t\geq 0}$ is defined by
\begin{equation*}
\hat{T}_t(x,y,m,s)=(\hat{T}^n(x,y), \Hol_n(x,y)^{-1}m, s+t-\Ret_n(x,y))
\end{equation*}
for all $(x,y,m,s) \in \Lambda_+\times \Lambda_- \times M \times \R$, where $n \in \Z_{\geq 0}$ such that $0\leq s+t-\Ret_n(x,y)<\Ret(\hat{T}^n(x,y))$.

\subsubsection{\texorpdfstring{The map from $\Lambda^\Ret\times M$ to $\F(X)$}{The map from Λᴿ✕M to F(X)}}
Abusing notation, we define the map
\begin{align*}
\tilde{\Phi}:\Lambda^\Ret\times M&\to \F(\HS)\\
(x,y,m,s)&\mapsto \tilde{F}(x,y)a_sm.
\end{align*}
This map is well-defined, and it induces a map $\Phi:\Lambda^\Ret\times M\to \F(X)$. In fact, we have the relation
\begin{equation*}
\Phi(x,y,m,a_s)=F(x,y)a_sm.
\end{equation*}
We claim that $\Phi$ defines a semiconjugacy between two semiflows:
\begin{equation}
\label{flow conj}
\Phi\circ \hat{T}_t = \mathcal{F}_t\circ \Phi \qquad \text{for all } t\geq 0.
\end{equation}
To see this, note that for any $(x,y,m,s)\in \Lambda^\Ret\times M$, we have the expression
\begin{align*}
\hat{T}_t(x,y,m,s) = \bigl(\hat{T}^n(x,y),\Hol_n(x,y)^{-1}m,s+t-\Ret_n(x,y)\bigr).
\end{align*}
Then on the one hand, we have
\begin{align*}
\Phi\circ \hat{T}_t &= \Phi\bigl(\hat{T}^n(x,y),\Hol_n(x,y)^{-1}m,s+t-\Ret_n(x,y)\bigr)\\
&= F\bigl(\hat{T}^n(x,y)\bigr)\Hol_n(x,y)^{-1}ma_{s+t-\Ret_n(x,y)}.
\end{align*}
On the other hand, we have
\begin{align*}
\mathcal{F}_t\circ \Phi &= \mathcal{F}_t\circ \Phi(x,y,m,s)\\
&= F(x,y)ma_{s+t}\\
&= F\bigl(\hat{T}^n(x,y)\bigr)\Hol_n(x,y)^{-1}ma_{s+t-\Ret_n(x,y)} \qquad \text{(by \cref{holonomy})}.
\end{align*}
This finishes the justification of \cref{flow conj}.

\subsubsection{\texorpdfstring{Relating $\hat{\nu}^\Ret\otimes m^{\mathrm{Haar}}$ with $m^{\operatorname{BMS}}$}{Relating ν̂ᴿ⊗mᴴᵃᵃʳ with mᴮᴹˢ}}
It is proved in \cite[Proposition 4.11]{LP22} that there exists a unique $\hat{T}$-invariant ergodic probability measure $\hat{\nu}$ on $\Lambda_+\times\Lambda_-$ which projects to the measure $\nu$ on $\Lambda_+$. We equip $\Lambda^{\Ret}$ with the $\{\hat{T}_t\}_{t \geq 0}$-invariant measure $d\hat{\nu}^{\Ret}:=d\hat{\nu}\, d\Leb/\bar{\Ret}$, where $\Leb$ is the Lebesgue measure and $\bar{\Ret} = \hat{\nu}(\Ret)$. Fix the Haar probability measure $m^{\mathrm{Haar}}$ on $M$, which is implicitly used within integrals. On $\Lambda^\Ret\times M$, we consider the product measure $\hat{\nu}^\Ret\otimes m^{\mathrm{Haar}}$: for any bounded continuous function $f: \Lambda^\Ret\times M \to \R$,
\begin{equation*}
\bigl(\hat{\nu}^\Ret\otimes m^{\mathrm{Haar}}\bigr)(f)=\int_{\Lambda^\Ret}\int_{M}f(x,y,m,s) \, dm \, d\hat{\nu}^\Ret(x,y,s).
\end{equation*}

Recall the following result proved in \cite{LP22}.

\begin{proposition}[{\cite[Proposition 4.15]{LP22}}]
\label{prop:FactorMap}
The map
\begin{equation*}
\Phi:\bigl(\Lambda^\Ret,\{\hat{T}_t\}_{t \geq 0}, \hat{\nu}^\Ret\bigr)\to \linebreak \bigl(\T(X),\{\mathcal{G}_t\}_{t \in \R},\m\bigr)
\end{equation*}
is a factor map, i.e.,
\begin{equation*}
\Phi_{*}\hat{\nu}^\Ret=\m \qquad \text{and} \qquad \Phi\circ \hat{T}_t=\mathcal{G}_t\circ \Phi \qquad \text{for all $t\geq0$}.
\end{equation*}
\end{proposition}

\begin{corollary}\label{cor:factormap}
The map
\begin{equation*}
\Phi:\bigl(\Lambda^\Ret\times M,\{\hat{T}_t\}_{t \geq 0}, \hat{\nu}^\Ret\otimes m^{\mathrm{Haar}}\bigr)\to \bigl(\F(X),\{\mathcal{F}_t\}_{t \in \R},\m\bigr)
\end{equation*}
is a factor map, i.e.,
\begin{equation*}
\Phi_{*}\bigl(\hat{\nu}^\Ret\otimes m^{\mathrm{Haar}}\bigr)=\m \qquad \text{and} \qquad \Phi\circ \hat{T}_t=\mathcal{F}_t\circ \Phi \qquad \text{for all $t\geq0$}.
\end{equation*}
\end{corollary}

\begin{proof}
Given a bounded function $f\in C(\F(X), \R)$, we define the function $\bar{f}:\T(X) \to \R$ by 
\begin{equation*}
\bar{f}(\Gamma gM)=\int_{M}f(\Gamma gm) \, dm.
\end{equation*}
We have $\bar{f}\in C(\T(X), \R)$ which is also bounded. Note that for any $(x,y,s)\in \Lambda^\Ret$, we have
\begin{equation*}
(\bar{f}\circ \Phi)(x,y,s)=\int_{M}f(F(x,y)a_sm) \, dm.
\end{equation*}
Hence, using \cref{prop:FactorMap}, we have
\begin{align*}
\m(f) &= \int_{\T(X)}\int_{M}f(\Gamma gm) \, dm \, d\m(\Gamma g M)=\int_{\T(X)} \bar{f} \, d\m\\
&= \int_{\Lambda^\Ret}(\bar{f}\circ \Phi)(x,y,s) \, d\hat{\nu}^\Ret(x,y,s)\\
&= \int_{\Lambda^\Ret}\int_{M}f(F(x,y)a_sm) \, dm \, d\hat{\nu}^\Ret(x,y,s)=\bigl(\hat{\nu}^\Ret\otimes m^{\mathrm{Haar}}\bigr)(f).
\end{align*}
\end{proof}

\section{Transfer operators with holonomy}
\label{sec:TransferOperators}
In this section, we introduce the transfer operator with holonomy associated to the countably infinite coding. The main technical objective in this paper is to obtain spectral bounds for these operators in \cref{sec:Dolgopyat'sMethod}.

Recall that we have fixed the Haar probability measure on $M$. Also recall that $L^2(M, \C)$ is a Hilbert space equipped with the standard inner product $\langle \cdot, \cdot \rangle$ defined by $\langle f, g\rangle = \int_M fg$. As usual, we denote by $\|\cdot\|_2$ the corresponding $L^2$ norm on $L^2(M, \C)$ and any of its subspaces.

In the proof of exponential mixing, we need to deal with the function space
\begin{align*}
C(\Lambda^\Ret\times M, \C) \cong C(\Lambda^\Ret, C(M, \C)) \subset C(\Lambda^\Ret, L^2(M, \C)).
\end{align*}
Let $\widehat{M}$ be the unitary dual of $M$. Denote the trivial irreducible representation by $1 \in \widehat{M}$. Define $\widehat{M}_0 = \widehat{M} - \{1\}$. Due to the above, it is natural to use the Peter--Weyl theorem and obtain the Hilbert space decomposition
\begin{align*}
L^2(M, \C) = \widehat{\bigoplus}_{\rho \in \widehat{M}} V_\rho^{\oplus \dim(\rho)}
\end{align*}
corresponding to the decomposition $\varrho = \widehat{\bigoplus}_{\rho \in \widehat{M}} \rho^{\oplus \dim(\rho)}$ of the left regular representation $\varrho: M \to \U(L^2(M, \C))$.

For all $b \in \mathbb R$ and $\rho \in \widehat{M}$, we define the tensored unitary representation $\rho_b: AM \to \U(V_\rho)$ by
\begin{align*}
\rho_b(a_tm)(z) = e^{ibt}\rho(m)(z) \qquad \text{for all $z \in V_\rho$, $t \in \mathbb R$, and $m \in M$}.
\end{align*}

We introduce some notations related to Lie algebras. We denote Lie algebras corresponding to Lie groups by the corresponding Fraktur letters: $\mathfrak{g} = \operatorname{T}_e(G)$, $\mathfrak{a} = \operatorname{T}_e(A)$, $\mathfrak{m} = \operatorname{T}_e(M)$, and $\mathfrak{n}^\pm = \operatorname{T}_e(N^\pm)$. For any unitary representation $\rho: M \to \U(V)$ for some Hilbert space $V$, we denote the differential at $e \in M$ by $d\rho = (d\rho)_e: \mathfrak{m} \to \mathfrak{u}(V)$, and define the norm
\begin{align*}
\|\rho\| = \sup_{z \in \mathfrak{m}, \|z\| = 1} \|d\rho(z)\|_{\mathrm{op}}
\end{align*}
and similarly for any unitary representation $\rho: AM \to \U(V)$.

The following are useful facts regarding the Lie theoretic norms (see \cite{SW21} for their proofs).

\begin{lemma}[{\cite[Lemma 4.3]{SW21}}]
\label{lem:LieTheoreticNormBounds}
For all $b \in \mathbb R$ and $\rho \in \widehat{M}$, we have
\begin{align*}
\sup_{a \in A, m \in M} \sup_{\substack{z \in \operatorname{T}_{am}(AM),\\ \|z\| = 1}} \|(d\rho_b)_{am}(z)\|_{\mathrm{op}} = \|\rho_b\|
\end{align*}
and $\max\{|b|, \|\rho\|\} \leq \|\rho_b\| \leq |b| + \|\rho\|$.
\end{lemma}

\begin{lemma}[{\cite[Lemmas 4.4]{SW21}}]
\label{lem:maActionLowerBound}
There exists $\varepsilon_1 \in (0, 1)$ such that for all $b \in \mathbb R$, $\rho \in \widehat{M}$, and $\omega \in V_\rho^{\oplus \dim(\rho)}$ with $\|\omega\|_2 = 1$, there exists $z \in \mathfrak{a} \oplus \mathfrak{m}$ with $\|z\| = 1$ such that $\|d\rho_b(z)(\omega)\|_2 \geq \varepsilon_1 \|\rho_b\|$.
\end{lemma}

The source of the oscillations needed in Dolgopyat's method is provided by the \emph{local non-integrability condition (LNIC)} which will be introduced in \cref{sec:LNIC} and the oscillations themselves are propagated when $\|\rho_b\|$ is sufficiently large. But this occurs precisely when $|b|$ is sufficiently large or $\rho \in \widehat{M}$ is nontrival. Let $b_0 > 0$ which we fix later. This motivates us to define
\begin{align*}
\widehat{M}_0(b_0) = \{(b, \rho) \in \mathbb R \times \widehat{M}: |b| > b_0 \text{ or } \rho \neq 1\}.
\end{align*}
We fix the related constant $\delta_{\varrho} = \inf_{b \in \mathbb R, \rho \in \widehat{M}_0} \|\rho_b\| = \inf_{\rho \in \widehat{M}_0} \|\rho\|$ which is positive because $M$ is a compact connected Lie group (recall from \cite[Example 3.1.4]{Lub10} that compact Lie groups have property (T)). Then, $\inf_{(b, \rho) \in \widehat{M}_0(b_0)} \|\rho_b\| \geq \min\{b_0, \delta_{\varrho}\}$. Thus, we also fix
\begin{align}
\label{eqn:Constantdelta1varrho}
\delta_{1, \varrho} = \min\{1, \delta_{\varrho}\}.
\end{align}

\begin{definition}[Transfer operator with holonomy]
For all $\xi = a + ib \in \mathbb C$ with $a > -\epsilon_0$ and $\rho \in \widehat{M}$, the \emph{transfer operator with holonomy} $\mathcal{M}_{\xi\Ret, \rho}: C\bigl(\Lambda_+, V_\rho^{\oplus \dim(\rho)}\bigr) \to C\bigl(\Lambda_+, V_\rho^{\oplus \dim(\rho)}\bigr)$ is defined by
\begin{align*}
\mathcal{M}_{\xi\Ret, \rho}(H)(x) &= \sum_{\gamma \in \calH} e^{-\xi\Ret(\gamma x)} \rho(\Hol(\gamma x)^{-1}) H(\gamma x) \\
&= \sum_{\gamma \in \calH} e^{-a\Ret(\gamma x)} \rho_b(\GHol(\gamma x)^{-1}) H(\gamma x)
\end{align*}
for all $x \in \Lambda_+$ and $H \in C\bigl(\Lambda_+, V_\rho^{\oplus \dim(\rho)}\bigr)$ with $\|H\|_\infty < \infty$.
\end{definition}

The above is well-defined due to the exponential tail property (see Property~(5) in \cref{prop:Coding}). Denote $\mathcal{L}_{\xi\Ret} := \mathcal{M}_{\xi\Ret, 1}$.

Let $(V, \|\cdot\|)$ be any normed vector space over $\R$ or $\C$. Let $d$ be any distance function on $\Delta_0$; in particular, $d = d_{\mathrm{E}}$ or $d = D$. For any function $H: \Lambda_+\to V$, denote
\begin{align*}
\|H\|_{\infty} &= \sup\{\|H(x)\|: x\in \Lambda_+\},\\
\Lip_d(H)&=\sup\left\{\frac{\|H(x) - H(x')\|}{d(x, x')}: x, x'\in \Lambda_+, x\neq x'\right\},\\
\|H\|_{\Lip(d)}&=\|H\|_{\infty}+\Lip_d(H).
\end{align*}
Denote by $\Lip_d(\Lambda_+, V)$ the space of functions $H:\Lambda_+\to V$ with $\|H\|_{\Lip(d)} < \infty$. We omit $d$ from the above notations if $d = d_{\mathrm{E}}$. In particular, we will work with the function spaces $\Lip\bigl(\Lambda_+, V_\rho^{\oplus \dim(\rho)}\bigr)$ and $\Lip_D(\Lambda_+, \R)$ corresponding to the normed vector spaces $\bigl(V_\rho^{\oplus \dim(\rho)}, \|\cdot\|_2\bigr)$ for some $\rho \in \widehat{M}$ and $(\R, |\cdot|)$.

For any function $H:\Delta_0\to V$, denote
\begin{equation*}
\Lip_d^{\mathrm{e}}(H)=\sup\left\{\frac{\|H(x) - H(x')\|}{d(x, x')}: x, x'\in \Delta_j, x\neq x', j \in \mathcal{J}\right\}.
\end{equation*}

Define the PS measure $\mu_{\mathrm{E}}$ on $\Delta_0$ with respect to the Euclidean metric by
\begin{equation}
\label{equ:mu E}
d\mu_{\mathrm{E}}(x)=(1+\|x\|^2)^{\delta} \, d\mu(x).
\end{equation}
Note that $\Lambda_+ \subset \Delta_0$ is a full measure subset with respect to $\mu_{\mathrm{E}}$. Using the quasi-invariance of the PS measure $\mu$, a straightforward computation gives $\mathcal{L}_{\delta \Ret}^*(\mu_{\mathrm{E}}) = \mu_{\mathrm{E}}$.

By \cref{lem:analyticity}, the family $\xi\mapsto \mathcal{L}_{(\delta+\xi)\Ret}$ of operators on $\Lip(\Lambda_+,\mathbb{C})$ is analytic on $\{\xi=a+ib\in \mathbb{C}: a>-\frac{\epsilon_0}{2}\}$. It can be shown as in \cite[Proposition A]{You98} that $\mathcal{L}_{\delta\Ret}|_{\Lip(\Lambda_+,\mathbb{C})}$ has a spectral gap with a maximal simple eigenvalue $\lambda_0 = 1$ with a corresponding eigenfunction $h_0:\Lambda_+\to \mathbb{R}$ given by $h_0(x)=(1+\|x\|^2)^{-\delta}f_0(x)$, where $f_0$ is the density function defined in \cref{eqn:nu_and_mu}. It satisfies $\int_{\Lambda_+} h_0 \, d\mu_{\mathrm{E}} = 1$. Moreover, by perturbation theory of operators (see \cite[Chapter 7]{Kat95}), there exists $a_0' \in (0, \frac{\epsilon_0}{2})$ and analytic maps
\begin{itemize}
\item $[-a_0', a_0'] \to \R$ denoted by $a \mapsto \lambda_a$,
\item $[-a_0', a_0'] \to \Lip(\Lambda_+, \R)$ denoted by $a \mapsto h_a$,
\end{itemize}
such that $\mathcal{L}_{(\delta + a)\Ret}h_a = \lambda_a h_a$, and $h_a$ is bounded away from $0$ and $+\infty$ and normalized such that $\int_{\Lambda_+} h_a \, d\mu_{\mathrm{E}}=1$.

Recall the measures $\mu$ and $\nu$ from \cref{subsec:PS measure and BMS measure,subsec:expanding map}, respectively, and also \cref{eqn:nu_and_mu}. Combining with \cref{equ:mu E} and the definition of $h_0$, we have
\begin{equation}
d\nu = h_0 \, d\mu_{\mathrm{E}}.
\end{equation}

Define the function
\begin{align*}
\FRet^{(a)} = -(\delta + a)\Ret - \log(\lambda_a) + \log \circ h_a - \log \circ h_a \circ T
\end{align*}
which is cohomologous to $-(\delta + a)\Ret - \log(\lambda_a)$. Due to Property~(4) in \cref{prop:Coding}, we can fix some $a_0' > 0$ and
\begin{align}
\label{eqn:ConstantC1'C2}
C_1' &> \max\left(\Lip^{\mathrm{e}}(\Ret), \Lip^{\mathrm{e}}(\GHol), \sup_{|a| \leq a_0'} \Lip^{\mathrm{e}}\bigl(\FRet^{(a)}\bigr)\right), & C_2 &= \frac{C_1'}{1 - \lambda},
\end{align}
where $\Lip^{\mathrm{e}}(\GHol)$ is defined similarly using the Riemannian metric on $AM$.

For all $\xi = a + ib \in \mathbb C$ with $a > -\epsilon_0$ and $\rho \in \widehat{M}$, we \emph{normalize} the transfer operator with holonomy as
\begin{align*}
\mathcal{M}_{\xi, \rho} = m_{\lambda_a h_a}^{-1} \circ \mathcal{M}_{(\delta + \xi)\Ret, \rho} \circ m_{h_a}
\end{align*}
where $m_h: C\bigl(\Lambda_+, V_\rho^{\oplus \dim(\rho)}\bigr) \to C\bigl(\Lambda_+, V_\rho^{\oplus \dim(\rho)}\bigr)$ denotes the multiplication operator by $h \in C(\Lambda_+, \R)$. For all $k \in \N$, its $k$-th iteration is simply given by
\begin{align*}
\mathcal{M}_{\xi, \rho}^k(H)(x) = \sum_{\gamma \in \calH^k} e^{\FRet_k^{(a)}(\gamma x)} \rho_b(\GHol^k(\gamma x)^{-1}) H(\gamma x)
\end{align*}
for all $x \in \Lambda_+$ and $H \in C\bigl(\Lambda_+, V_\rho^{\oplus \dim(\rho)}\bigr)$. Denote $\calL_\xi:=\calM_{\xi,1}$. Due to the above normalization, it satisfies $\cal{L}_0^*(\nu) = \nu$.

\section{Local non-integrability condition}
\label{sec:LNIC}
In this section we establish the first key property called the \emph{local non-integrability condition (LNIC)}.

Let $\Gamma_{\mathrm{b}} < \Gamma$ be the subsemigroup generated by $\{\gamma^{-1} \in \Gamma: \gamma\in \calH\}$. Let $\Lambda_{\mathrm{b}} \subset \Lambda_\Gamma \subset \partial \HS$ be the limit set of $\Gamma_{\mathrm{b}}$. Recall the following results established in \cite{LP22}.

\begin{lemma}[{\cite[Lemma 6.62]{LP22}}]
\label{lem:LambdabZDense}
The limit set $\Lambda_{\mathrm{b}}$ is not contained in any generalized sphere in $\partial \HS = \R^d\cup \{\infty\}$.
\end{lemma}

\begin{lemma}[{\cite[Lemma 6.64]{LP22}}]
\label{lem:BackwardTopologicalMixing}
Let $\xi \in \Lambda_{\mathrm{b}}$. For any $\epsilon>0$, there exists $n_{\xi}\in \N$ such that for all $n \geq n_{\xi}$, there exists $\gamma \in \calH^n$ satisfying
\begin{equation*}
d_{\mathbb{S}^d}(\gamma^{-1}\infty,\xi)\leq \epsilon.
\end{equation*}
\end{lemma}

\begin{remark}
It follows from \cref{lambda-} that $\Lambda_{\mathrm{b}}\subset\Lambda_-$. Moreover, the construction of $\gamma\in \calH$ yields that $\infty \in \Lambda_{\mathrm{b}}$ (see \cite[Section 6]{LP22}).
\end{remark}

Recall the reference point $x_0 \in \Lambda_+$ from \cref{subsec:SymbolicModelForFrameFlows}. We start with a definition similar to \cite[Definition 6.1]{SW21} which was inspired by Brin--Pesin moves \cite{BP74,Bri82}.

\begin{definition}[Associated sequence of frames]
For any sequence
\begin{align}
\label{eqn:SequenceOnBoundary^2}
((x_0, \infty), (x_0, y), (x, y), (x, \infty), (x_0, \infty)) \in (\Delta_0 \times \Lambda_-)^5
\end{align}
we define a unique \emph{associated sequence of frames} $(g_1, g_2, \dotsc, g_5) \in \F(\mathbb H^n)^5 \cong G^5$ by
\begin{align*}
g_1 &= \tilde{F}(x_0, \infty), \\
g_2 &= \tilde{F}(x_0, y) \in g_1N^- \text{ such that } g_2M = \tilde{\Phi}(x_0, y, 0) \in \T(\mathbb H^n) \cong G/M, \\
g_3 &\in g_2N^+ \text{ such that } g_3a_tM = \tilde{\Phi}(x, y, 0) \in \T(\mathbb H^n) \cong G/M \text{ for some } t \in \R, \\
g_4 &\in g_3N^- \text{ such that } g_4a_tM = \tilde{\Phi}(x, \infty, 0) \in \T(\mathbb H^n) \cong G/M \text{ for some } t \in \R, \\
g_5 &\in g_4N^+ \text{ such that } g_5a_tM = \tilde{\Phi}(x_0, \infty, 0) \in \T(\mathbb H^n) \cong G/M \text{ for some } t \in \R.
\end{align*}
\end{definition}

We continue using the notation in the above definition. Define the subsets
\begin{align*}
N_1^+ &= \{n^+ \in N^+: F(x_0, \infty)n^+ \in F(\Delta_0 \times \{\infty\})\} \subset N^+, \\
N_1^- &= \{n^- \in N^-: F(x_0, \infty)n^- \in F(\{x_0\} \times \Lambda_{\mathrm{b}})\} \subset N^-,
\end{align*}
where the first is open and the second is compact. Define $\tilde{N}_1^- \subset N^-$ to be an open neighborhood of $N_1^-$. Now, if the sequence in \cref{eqn:SequenceOnBoundary^2} corresponds to some $n^+ \in N_1^+$ and $n^- \in N_1^-$ such that $F(x, \infty) = F(x_0, \infty)n^+$ and $F(x_0, y) = F(x_0, \infty)n^-$ respectively, then we can define the map
\begin{align*}
\Xi: N_1^+ \times N_1^- &\to AM \\
(n^+, n^-) &\mapsto g_5^{-1}g_1.
\end{align*}
To view it as a function of only the first argument for a fixed $n^- \in N_1^-$, we write $\Xi_{n^-}: N_1^+ \to AM$.

Now we relate this to the generalized holonomy $\GHol$ and LNIC. Following definitions and using \cref{lambda-}, for all $\gamma \in \bigcup_{n\in \mathbb{N}}\calH^n$, there exists $n_\gamma \in N_1^-$ such that
\begin{align*}
F(x_0, \gamma^{-1}(\infty)) = F(x_0, \infty)n_\gamma.
\end{align*}

We want to show that the generalized holonomy $\GHol$ is in some suitable sense rapidly oscillating. We begin by relating it to the map $\Xi$ and the associated Lie theory as follows.

\begin{lemma}
\label{lem:BrinPesinInTermsOfHolonomy}
Let $\gamma \in \calH^n$ for some $n \in \N$ and $n^- = n_\gamma \in N_1^-$. Let $x \in \Delta_0$ and $n^+ \in N_1^+$ such that $F(x, \infty) = F(x_0, \infty)n^+$. Then, we have
\begin{align*}
\Xi(n^+, n^-) = \GHol^n(\gamma(x_0), \infty)^{-1}\GHol^n(\gamma(x), \infty).
\end{align*}
\end{lemma}
This lemma is analogous to \cite[Lemma 6.2]{SW21} and can be proved in a similar fashion.

Let $\pi: \LieG \to \LieA \oplus \LieM$ be the projection map with respect to the decomposition $\LieG = \LieA \oplus \LieM \oplus \LieN^+ \oplus \LieN^-$. The following lemma can be proven exactly as in \cite[Lemma 6.3]{SW21}.

\begin{lemma}
\label{lem:BrinPesinDerivativeImageIsAdjointProjection}
For all $n^- \in N_1^-$, we have
\begin{align*}
(d\Xi_{n^-})_e = \pi \circ \Ad_{n^-}|_{\LieN^+} \circ (dh_{n^-})_e
\end{align*}
where $h_{n^-}: N_1^+ \to N^+$ is a diffeomorphism onto its image which is also smooth in $n^- \in \tilde{N}_1^-$ and satisfies $h_e = \Id_{N_1^+}$. Consequently, its image is $(d\Xi_{n^-})_e(\LieN^+) = \pi(\Ad_{n^-}(\LieN^+))\subset \LieA \oplus \LieM$.
\end{lemma}

The following lemma can also be proven as in \cite[Lemma 6.4]{SW21} using \cref{lem:LambdabZDense} which replaces the fact that the whole limit set $\Lambda_{\Gamma}$ is not contained in any generalized sphere in $\partial \HS = \mathbb{R}^d\cup \{\infty\}$.

\begin{lemma}
\label{lem:am_ProjectionOfAdjointImage}
There exist $n_1^{-},n_2^{-},\ldots,n_{\jj}^-\in N_1^-$ for some $\jj \in \N$ and $\epsilon > 0$ such that if $\eta_1^{-},\eta_2^-,\ldots,\eta_{\jj}^-\in N_1^-$ with $d_{N^-}(\eta_j^-,n_j^-)\leq \epsilon$ for all $1\leq j\leq \jj$, then
\begin{equation*}
\sum_{j=1}^{\jj} \pi\bigl(\Ad_{\eta_j^-}(\LieN^+)\bigr)=\LieA\oplus \LieM.
\end{equation*}
\end{lemma}

We can now state the LNIC appropriate to our setting. Given any pair of inverse branches $\alpha, \beta \in \calH^m$ for some $m\in \mathbb{N}$, define the map $\BP_{\alpha, \beta}:\Delta_0\times \Delta_0\to AM$ by
\begin{align*}
\BP_{\alpha, \beta}(x,y) = \GHol^m(\alpha x, \infty)^{-1}\GHol^m(\alpha y, \infty)\GHol^m(\beta y, \infty)^{-1}\GHol^m(\beta x, \infty)
\end{align*}
for all $x, y \in \Delta_0$.

\begin{proposition}[LNIC]
\label{prop:LNIC}
There exist $\epsilon\in (0,1)$, $m_0\in \N$, $\jj\in \N$, such that for all $m\geq m_0$, there exist $\{\alpha_j\}_{j = 0}^{\jj} \subset \calH^m$ such that for all $x \in \Delta_0$ and $\omega \in \LieA \oplus \LieM$ with $\|\omega\| = 1$, there exist $1\leq j\leq \jj$ and $Z\in \operatorname{T}_x(\Delta_0)$ with $\|Z\| = 1$ such that
\begin{equation*}
|\langle  (d\BP_{j, x})_x(Z),\omega\rangle|\geq \epsilon
\end{equation*}
where we denote $\BP_{j}:=\BP_{\alpha_0,\alpha_j}$ and $\BP_{j, x} := \BP_j(x, \cdot)$ for all $x \in \Delta_0$ and $1\leq j\leq \jj$. 
\end{proposition}

\Cref{prop:LNIC} will be derived from \cref{prop:PreLNIC}. 

\begin{proposition}
\label{prop:PreLNIC}
There exist $\epsilon\in (0,1)$, $m_0\in \N$, $\jj\in \N$, and an open neighborhood $U \subset \Delta_0$ of $x_0$ such that for all $m\geq m_0$, there exist $\{\alpha_j\}_{j = 0}^{\jj} \subset \calH^m$ such that for all $x \in U$ and $\omega \in \LieA \oplus \LieM$ with $\|\omega\| = 1$, there exist $1\leq j\leq \jj$ and $Z\in \operatorname{T}_x(U)$ with $\|Z\|=1$ such that
\begin{equation*}
|\langle  (d\BP_{j, x})_x(Z),\omega\rangle|\geq \epsilon
\end{equation*}
using the same notation as in \cref{prop:PreLNIC}.
\end{proposition}
\cref{prop:PreLNIC} can be proven as in \cite[Proposition 6.5]{SW21} using \cref{lem:BrinPesinDerivativeImageIsAdjointProjection,lem:am_ProjectionOfAdjointImage} with only notational changes. \Cref{lem:BackwardTopologicalMixing} is also required as a replacement for the topological mixing property of the Markov section. 

We need the following lemma.
\begin{lemma}
\label{lem:CylinderInOpenSet}
Let $U \subset \Delta_0$ be an open subset with $U \cap \Lambda_\Gamma \neq \varnothing$. There exists $m_0 \in \N$ such that for all integers $m \geq m_0$, there exists $\gamma \in \calH^m$ such that $\gamma\Delta_0 \subset U$.
\end{lemma}

\begin{proof}
Let $U \subset \Delta_0$ be an open subset with $U \cap \Lambda_\Gamma \neq \varnothing$. Let $x \in U \cap \Lambda_\Gamma$ and $\epsilon > 0$ such that $B(x,{2\epsilon}) \subset U$. Note that $\mu(B(x,\epsilon)) > 0$. Fix $m_0 \in \N$ such that $\lambda^{m_0}\diam(\Delta_0) < \epsilon$. Let $m \geq m_0$ be an integer. Since $\mu(\Delta_0) = \sum_{\gamma \in \calH^n} \mu(\gamma\Delta_0)$ for any $n \in \N$, there exists $\gamma \in \calH^m$ such that $B(x,\epsilon) \cap \gamma\Delta_0 \neq \varnothing$. By Property~(3) in \cref{prop:Coding}, we also have $\diam(\gamma\Delta_0) \leq \lambda^m \diam(\Delta_0) \leq \lambda^{m_0} \diam(\Delta_0) < \epsilon$ and hence $\gamma\Delta_0 \subset B(x,{2\epsilon}) \subset U$.
\end{proof}

\begin{proof}[Proof that \cref{prop:PreLNIC} implies \cref{prop:LNIC}]
Let $\tilde{\epsilon}$, $\tilde{m}_0$, $\jj$, and $U$ be the $\epsilon$, $m_0$, $\jj$, and $U$ from \cref{prop:PreLNIC}. By \cref{lem:CylinderInOpenSet}, we can fix some $\gamma \in \calH^n$ for some $n \in \N$ such that $\gamma\Delta_0 \subset U$. Fix $\epsilon \in (0, \tilde{\epsilon} \cdot \inf\{\|(d\gamma)_x\|: x \in \Delta_0\})$ which is possible due to Property~(4) in \cref{prop:Coding}. Fix $m_0 = \tilde{m}_0 + n$. Let $m \geq m_0$ be an integer. Let $\tilde{m} = m - n \geq \tilde{m}_0$. Let $\{\tilde{\alpha}_j\}_{j = 0}^{\jj}$ be the inverse branches provided by \cref{prop:PreLNIC} and $\{\alpha_j\}_{j = 0}^{\jj} = \{\tilde{\alpha}_j \gamma\}_{j = 0}^{\jj}$. Denote $\widetilde{\BP}_j := \BP_{\tilde{\alpha}_0, \tilde{\alpha}_j}$ and $\BP_j := \BP_{\alpha_0, \alpha_j}$. Denote by $C_g: G \to G$ the conjugation map by $g \in G$. For all $1 \leq j \leq \jj$, using definitions and the fact that the group $A$ commutes with the group $AM$, we calculate that
\begin{align*}
\BP_j(x, x') = C_{\Hol^n(\gamma x, \infty)^{-1}}(\widetilde{\BP}_j(\gamma x, \gamma x')) \qquad \text{for all $x, x' \in \Delta_0$},
\end{align*}
so taking the differential gives
\begin{align*}
(d\BP_{j, x})_x = \Ad_{\Hol^n(\gamma x, \infty)^{-1}} \circ (d\widetilde{\BP}_{j, \gamma x})_{\gamma x} \circ (d\gamma)_x \qquad \text{for all $x \in \Delta_0$}.
\end{align*}
Let $x \in \Delta_0$ and $\omega \in \LieA \oplus \LieM$ with $\|\omega\| = 1$, and take $\tilde{\omega} = \Ad_{\Hol^n(\gamma x, \infty)^{-1}}^*(\omega)$. Note that we still have $\|\tilde{\omega}\| = 1$ since the inner product on $\LieG$ is left $\Ad_K$-invariant. Using $\tilde{x} := \gamma x \in \gamma\Delta_0 \subset U$ and \cref{prop:PreLNIC}, there exist $1 \leq \tilde{j} \leq \jj$ and $\tilde{Z} \in \operatorname{T}_{\tilde{x}}(U)$ with $\|\tilde{Z}\|=1$ such that
\begin{align*}
|\langle (d\widetilde{\BP}_{\tilde{j}, \tilde{x}})_{\tilde{x}}(\tilde{Z}),\tilde{\omega}\rangle| \geq \tilde{\epsilon}.
\end{align*}
Thus, taking $j = \tilde{j}$ and $Z = \frac{(dT^n)_{\tilde{x}}(\tilde{Z})}{\|(dT^n)_{\tilde{x}}(\tilde{Z})\|}$, we calculate that
\begin{align*}
|\langle  (d\BP_{j, x})_x(Z),\omega\rangle| &= |\langle (d\widetilde{\BP}_{j, \gamma x})_{\gamma x}((d\gamma)_x(Z)),  \Ad_{\Hol^n(\gamma x, \infty)^{-1}}^*(\omega)\rangle| \\
&= \|(d\gamma)_x(Z)\| \cdot |\langle (d\widetilde{\BP}_{\tilde{j}, \tilde{x}})_{\tilde{x}}(\tilde{Z}), \tilde{\omega}\rangle| \\
&\geq \tilde{\epsilon} \cdot \inf\{\|(d\gamma)_x\|: x \in \Delta_0\} > \epsilon.
\end{align*}
\end{proof}

Fix $\varepsilon_2 \in (0, 1)$, $m_0 \in \N$, and $\jj \in \N$ to be the $\epsilon$, $m_0$, and $\jj$ provided by \cref{prop:LNIC} for the rest of the paper.

We finish this section with an approximation lemma which will be used in \cref{sec:Dolgopyat'sMethod}. Fix $\delta_{AM} > 0$ such that any pair of points in $B_{AM}(e,\delta_{AM}) \subset AM$ has a unique geodesic through them. Fix a constant $C_{\BP} > 0$ such that
\begin{equation*}
C_{\BP} \geq \sup\{\|(d\BP_{j, x})_y\|_{\mathrm{op}}:x, y \in \Delta_0, j \in \{1, 2, \dotsc, \jj\}\}
\end{equation*}
for any $m\geq m_0$ and corresponding maps $\{\BP_j\}_{j=1}^{\jj}$ provided by \cref{prop:LNIC}. It can be checked from the proofs of \cref{prop:PreLNIC,prop:LNIC} (see \cite[Proposition 6.5]{SW21}, in particular, \cite[Eq. (6)]{SW21}) that $C_{\BP}$ can be chosen independently of the inverse branches and their length provided by \cref{prop:LNIC}. \Cref{lem:ComparingExpWithBP} can be proved as in \cite[Lemma 7.1]{SW21}. 

\begin{lemma}
\label{lem:ComparingExpWithBP}
Let $m\geq m_0$ and $\{\BP_j\}_{j=1}^{\jj}$ be the corresponding maps provided by \cref{prop:LNIC}. Then there exists $C_{\exp, \BP} > 0$ such that for all $1 \leq j \leq \jj$ and $x, y \in \Delta_0$ with $\|x - y\| < \frac{\delta_{AM}}{C_{\BP}}$, we have
\begin{align*}
d_{AM}(\exp(Z), \BP_j(x, y)) \leq C_{\exp, \BP}\|x - y\|^2
\end{align*}
where $Z = (d\BP_{j, x})_x(y-x)$.
\end{lemma}

\section{Non-concentration property}
\label{sec:NCP}
In this section we establish the second key property called the \emph{non-concentration property (NCP)}. The main result, \cref{prop:NCP}, is the appropriate generalization of NCP in \cite[Proposition 6.6]{SW21} for geometrically finite hyperbolic manifolds with cusps.

We start with some notations. Recall the choice of reference vector $v_o\in \T(\HS)$ from \cref{subsec:hyperbolic spaces}. With this choice, $a_t$ acts on $\HS \subset \R^{d+1}$ simply by scaling by a factor of $e^{-t}$ for all $t \in \R$. We parametrize the unstable horospherical subgroup
\begin{equation*}
N^+=\{n_x^+:x\in \mathbb{R}^d\}
\end{equation*}
so that for all $x \in \R^d$, we have the forward endpoint $(n_x^+)^+ = x$. Define the map $u_\bullet: \Delta_0 \to \T(X)$ by $u_x = \pi\circ\tilde\Phi(x, \infty, 0)$ for all $x \in \Delta_0$, where $\tilde{\Phi}$ is the embedding defined as in \cref{eqn:embedding} and $\pi$ is the projection from $\partial^2(\HS)\times \mathbb{R}\cong \mathrm{T}^1(\HS)$ to $\mathrm{T}^1(X)$. Note that $u_{\Delta_0}$ is then the immersion of a subset of the horosphere corresponding to $\Delta_0 \times \{\infty\} \times \{0\}$. Recalling \cref{eqn:ConstantcDelta0}, we define the compact region
\begin{align}
\label{compact set}
\Omega_R \subset \T(X)
\end{align}
to be the closed $(R + \log(C_{\Delta_0}))$-neighborhood of $u_{\Delta_0}$.

\begin{proposition}[NCP]
\label{prop:NCP}
Let $R > 0$. There exists $\eta \in (0, 1)$ such that for all $\epsilon \in (0, 1)$, $x \in \Lambda_\Gamma \cap \Delta_0 - \overline{B(\partial \Delta_0,\epsilon)}$ with $u_x a_{-\log(\epsilon)} \in \Omega_R$, and $w \in \R^d$ with $\|w\| = 1$, there exists $y \in \Lambda_\Gamma \cap B(x,\epsilon) \subset \Delta_0$ such that $|\langle y - x, w \rangle| \geq \epsilon \eta$.
\end{proposition}

\begin{proof}
To obtain a contradiction, suppose the proposition is false. Then there exists $R > 0$ such that for all $j \in \mathbb N$, taking $\eta_j = \frac{1}{j}$, there exist $\epsilon_j \in (0, 1)$, $x_j \in \Lambda_\Gamma \cap \Delta_0 - \overline{B(\partial \Delta_0,\epsilon_j)}$ with $u_{x_j} a_{-\log(\epsilon_j)}  \in \Omega_R$, and $w_j \in \mathbb R^d$ with $\|w_j\| = 1$, such that $|\langle y - x_j, w_j \rangle| \leq \epsilon_j \eta_j = \frac{\epsilon_j}{j}$ for all $y \in \Lambda_\Gamma \cap B(x_j, \epsilon_j)$. Hence, we can rewrite this as
\begin{align}
\label{eqn:IfLemmaIsFalse}
\Lambda_\Gamma \cap B(x_j,\epsilon_j) \subset \left\{y \in \mathbb R^d: |\langle y - x_j, w_j \rangle| \leq \frac{\epsilon_j}{j}\right\} \qquad \text{for all $j \in \mathbb N$}.
\end{align}

We want to use the self-similarity property of the fractal set $\Lambda_\Gamma$. We have $\Gamma n_{x_j}^+M = u_{x_j} \in u_{\Delta_0}$. For all $j \in \mathbb N$, setting $t_j = -\log(\epsilon_j)$ we have $\Gamma n_{x_j}^+ a_{t_j}M = u_{x_j} a_{t_j} \in \Omega_R$ by hypothesis and hence $n_{x_j}^+ a_{t_j} \in \Gamma \tilde{\Omega}_R$ where $\tilde{\Omega}_R$ is some $M$-invariant lift of $\Omega_R$ which we note is compact. Thus, for all $j \in \mathbb N$, there exist $\beta_j \in \Gamma$ and $g_j \in \tilde{\Omega}_R$ such that $n_{x_j}^+ a_{t_j} = \beta_j g_j$. Now for all $j \in \mathbb N$, we have $g_j a_{-t_j} n_{-x_j}^+ = \beta_j^{-1}$. Hence, the action of $g_j a_{-t_j} n_{-x_j}^+ $ on $\partial_\infty(\mathbb H^{d+1})$ preserves $\Lambda_\Gamma$.

Now, applying $g_j a_{-t_j} n_{-x_j}^+$ in \cref{eqn:IfLemmaIsFalse} gives
\begin{align*}
\Lambda_\Gamma \cap g_j B(0,1) \subset g_j\left\{y \in \mathbb R^{d}: |\langle y, w_j \rangle| \leq \frac{1}{j}\right\} \qquad \text{for all $j \in \mathbb N$}.
\end{align*}
By compactness, we can pass to subsequences so that $\lim_{j \to \infty} w_j = w \in \mathbb R^{d}$ with $\|w\| = 1$ and $\lim_{j \to \infty} g_j = g \in \tilde{\Omega}_R$. Then in the limit $j \to \infty$, we have $\Lambda_\Gamma \cap gB(0,1) \subset g\left\{y \in \mathbb R^{d}: \langle y, w \rangle = 0\right\}$. This contradicts \cite[Proposition 3.12]{Win15} since $\Gamma < G$ is Zariski dense.
\end{proof}

It can also be deduced from the structure of cusps (see \cref{subsec:StructureOfCusps}) and the geometry of $\partial \HS$ that NCP is true without the condition involving the compact subset $\Omega_R \subset \T(X)$ if and only if all the cusps are of maximal rank (cf. \cite[Theorem 3.15]{DFSU21} and its proof). That is, the condition involving the compact subset $\Omega_R \subset \T(X)$ is necessary in the presence of cusps of non-maximal rank. We end this section with an example below which demonstrates the latter by elaborating on the comments after \cite[Theorem 3.15]{DFSU21}. As a result, the difficulty is that in Dolgopyat's method, we cannot obtain cancellations on \emph{every} set in the partition $\calP_{(b, \rho)}$ of $\Delta_0$ (see \cref{lem:PartnerPointInZariskiDenseLimitSetForBPBound,prop:CancellationCylinder}).

\begin{example}
\label{exa:non concentration}
Suppose $\Gamma < G$ is a geometrically finite subgroup with parabolic elements and $p \in \partial \HS-\{\infty\}$ is a rank $1$ parabolic fixed point for the $\Gamma$-action on $\partial \HS$. Let $g\in G$ be an element such that $gp=\infty$. Define $\tilde{\Gamma} = g\Gamma g^{-1} < G$ which is isomorphic to $\Gamma$. Then, $\infty \in \partial \HS$ is a rank $1$ parabolic fixed point for the $\tilde{\Gamma}$-action on $\partial \HS$. Recalling \cref{subsec:StructureOfCusps} and the notations there, we have a corresponding fundamental domain $\Delta_{\infty}:=B_Y(C_\infty)\times \Delta_{\infty}'$.

Let $\Delta_{Y} \subset Y \subset \R^d$ be an open $(d-1)$-dimensional parallelotope containing $B_Y(C_\infty)$. By \cref{limit set inclusion}, we have
\begin{equation*}
\Lambda_{\tilde{\Gamma}} - \{\infty\} \subset \bigcup_{\gamma\in \tilde{\Gamma}_\infty} \gamma \left(\overline{\Delta_Y\times \Delta_{\infty}'}\right),
\end{equation*}
where $\tilde{\Gamma}_\infty < \Stab_{\tilde{\Gamma}}(\infty)$ is the finite index subgroup provided by \cref{lem:Bieberbach}. Let $\{F_{j, k}: j \in \{1, 2, \dotsc, d\}, k \in \{1, 2\}\}$ be the set of $(d-1)$-dimensional generalized affine subspaces in $\R^d \cup \{\infty\}$ which contain the corresponding $(d-1)$-dimensional faces of $\Delta_Y\times \Delta_{\infty}'$. The set is ordered such that $F_{j, 1}$ and $F_{j, 2}$ are parallel for all $j \in \{1, 2, \dotsc, d\}$, and $F_{d, 1}$ and $F_{d, 2}$ are orthogonal to the $1$-dimensional subspace $Z \subset \R^d$. This set of generalized affine subspaces determine a corresponding set $\{H_{j, k}: j \in \{1, 2, \dotsc, d - 1\}, k \in \{1, 2\}\}$ of open half spaces not containing $\Lambda_{\tilde{\Gamma}}$.

Using $g^{-1}$, we conclude that $\{g^{-1}F_{j, k}: j \in \{1, 2, \dotsc, d\}, k \in \{1, 2\}\}$ consists of $(d-1)$-dimensional spheres such that $g^{-1}F_{j, 1}$ and $g^{-1}F_{j, 2}$ are mutually tangent to each other at $p$ for all $j \in \{1, 2, \dotsc, d\}$. Moreover, $g^{-1}F_{j, 1}$ is the boundary of the open ball $g^{-1}H_{j, k}$ for all $j \in \{1, 2, \dotsc, d - 1\}$ and $k \in \{1, 2\}$. In fact, we have
\begin{equation*}
\Lambda_{\Gamma}\subset \R^d \cup \{\infty\} - \bigcup_{\substack{j \in \{1, 2, \dotsc, d - 1\},\\k \in \{1, 2\}}} g^{-1}H_{j, k}.
\end{equation*}
Choose $w\in \mathbb{R}^d$ to be any unit vector based at $p$ and orthogonal to both $g^{-1}F_{j, 1}$ and $g^{-1}F_{j, 2}$ for any choice of $j \in \{1, 2, \dotsc, d - 1\}$. Note that $w$ is then automatically tangent to both $g^{-1}F_{d, 1}$ and $g^{-1}F_{d, 2}$. Then NCP as stated in \cite[Proposition 6.6]{SW21} fails at $p$ for the direction $w$, i.e., the following is \emph{false}: there exists $\eta \in (0, 1)$ such that for all $\epsilon \in (0, 1)$, there exists $y \in \Lambda_\Gamma \cap B(p, \epsilon)$ such that $|\langle y - p, w\rangle| \geq \epsilon \eta$.

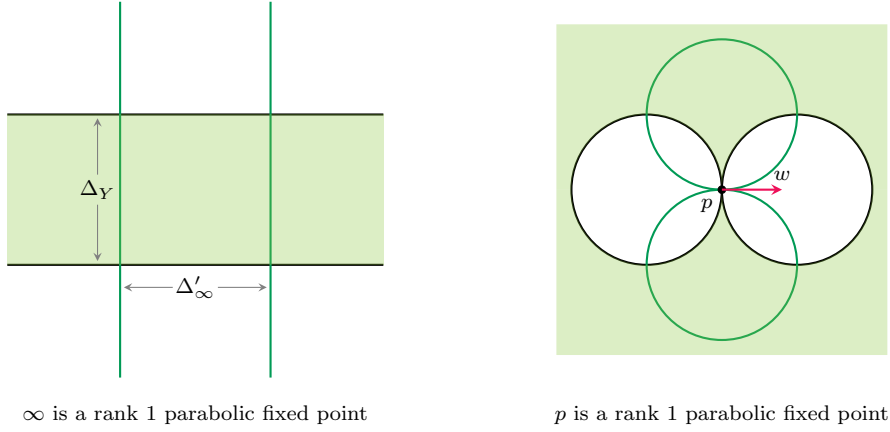
\begin{figure}[h]
\label{fig:rank 1}
\centering
\begin{tikzpicture}[>=stealth]
%Checkerboard
\draw[black, thick] (-6.5, 1) to (-1.5, 1);
\draw[black, thick] (-6.5, -1) to (-1.5, -1);
\draw[ForestGreen, thick] (-5, 2.5) to (-5, -2.5);
\draw[ForestGreen, thick] (-3, 2.5) to (-3, -2.5);
%Shading for Checkerboard
\fill[LimeGreen, fill opacity=0.3] (-6.5, 1) to (-6.5, -1) to (-1.5, -1) to (-1.5, 1)  -- cycle;
%Text
\node[] at (-5.3, 0) {\footnotesize $\Delta_Y$};
\draw[gray,->] (-5.3, 0.2) to (-5.3, 0.95);
\draw[gray,->] (-5.3, -0.2) to (-5.3, -0.95);
\node[] at (-4, -1.3) {\footnotesize $\Delta_\infty'$};
\draw[gray,->] (-4.3, -1.3) to (-4.95, -1.3);
\draw[gray,->] (-3.7, -1.3) to (-3.05, -1.3);
%\node[] at (-4, 0) {\footnotesize $\Delta_Y \times \Delta_\infty'$};
\node[] at (-4, -3) {\footnotesize $\infty$ is a rank $1$ parabolic fixed point};

%Flower
\draw[black, thick] (2, 0) circle  [radius = 1, fill = white];
\draw[black, thick] (4, 0) circle  [radius = 1, fill = white];
\draw[ForestGreen, thick] (3, 1) circle  [radius = 1, fill = white];
\draw[ForestGreen, thick] (3, -1) circle  [radius = 1, fill = white];
%Shading for Flower
\fill[LimeGreen, fill opacity=0.3] (3, 0) arc (0:180:1) to (0.8, 0) to (0.8, 2.2) to (5.2, 2.2) to (5.2, 0) to (5, 0) arc (0:180:1);
\fill[LimeGreen, fill opacity=0.3] (3, 0) arc (0:-180:1) to (0.8, 0) to (0.8, -2.2) to (5.2, -2.2) to (5.2, 0) to (5, 0) arc (0:-180:1);
%Intersection
\draw[fill = black] (3, 0) circle  [radius = 0.05];
%Text
\node[below left] at (3, 0) {\footnotesize $p$};
\node[] at (3, -3) {\footnotesize $p$ is a rank $1$ parabolic fixed point};

%Vector
\draw[OrangeRed,thick,->] (3, 0) to (3.8, 0) ;
%Text
\node[above] at (3.8, 0) {\footnotesize $w$};
\end{tikzpicture}
\caption{An illustration of a rank $1$ parabolic fixed point in $\partial\mathbb{H}^3$. The limit sets are contained in the shaded regions.}
\end{figure}
\end{example}

\section{Large deviation property}
\label{sec:LDP}\label{subsec:CombinatoricWords}
In this section we establish the third key property called the \emph{large deviation property (LDP)}.

Recall the measures $\mu$ and $\nu$ from \cref{subsec:PS measure and BMS measure,subsec:expanding map}, respectively, and also \cref{eqn:nu_and_mu}.

For all $t > 0$ and $R > 0$, we define $\Omega^\dagger(t, R)$ to be the set consisting of maximal cylinders $\mathtt{C} \subset \Delta_0$ satisfying
\begin{align*}
e^{-R - t} \leq \diam(\mathtt{C}) \leq e^{R - t}.
\end{align*}
Define 
\[\Omega(t, R) = \bigcup_{\mathtt{C}\in \Omega^\dagger(t, R)}\mathtt{C} \subset \Delta_0.\]

The following is the relation between $\Omega(t,R) \subset \Delta_0$ and the compact subset $\Omega_R \subset \T(X)$ defined in \cref{compact set}.

\begin{lemma}
\label{lem:goodpartitioncusp}
Let $t > 0$ and $R > 0$. For all $x \in \Omega(t, R)$, we have $u_xa_t\in \Omega_R$.
\end{lemma}

To prove \cref{lem:goodpartitioncusp}, we need a quick estimate which follows from \cref{lem:CylinderEstimate}.
\begin{lemma}
\label{lem:DerivativeEstimate}
Let $t > 0$ and $R > 0$. For any cylinder $\mathtt{C}\in \Omega^\dagger(t,R)$, write $\mathtt{C}=\gamma \Delta_0$ with $\gamma\in \bigcup_{n\in \mathbb{N}}\calH^n$. We have
\begin{equation*}
C_{\Delta_0}^{-1} e^{-R-t}\leq \lVert d\gamma\rVert\leq C_{\Delta_0} e^{R-t}.
\end{equation*}
\end{lemma}

\begin{proof}[Proof of \cref{lem:goodpartitioncusp}]
Let $t > 0$ and $R > 0$. Let $x \in \Omega(t, R)$. Take the cylinder $\mathtt{C} = \gamma \Delta_0 \in \Omega^\dagger(t, R)$ containing $x$ and let $n$ be its length. Recall that 
\begin{equation*}
\Ret_n(x)=\log\|(dT^n)_x\|=-\log \|(d\gamma)_{x'}\|
\end{equation*}
with $x=\gamma x'$.
Using \cref{lem:DerivativeEstimate} and \cref{equ:d gamma x}, we have
\begin{align*}
\Ret_n(x) \in [t - (R + \log(C_{\Delta_0})), t + (R + \log(C_{\Delta_0}))].
\end{align*}
Since $u_xa_{\Ret_n(x)}\in u_{\Delta_0}$, we obtain $u_xa_t \in \Omega_R$.
\end{proof}

The following proposition is the required LDP. We will show in \cref{sec:reduction_combinatoric} that it follows from \cref{prop:badword}.

\begin{proposition}[LDP]
\label{prop:LDP}
There exist $R_0>0$ and $\kappa \in (0, 1)$ such that the following holds.
For all $m,n \in \N$ and $t > 0$, we have
\begin{align*}
\nu\{x\in\Lambda_+: \#\{j \in \N: j \leq n, T^{jm}(x) \in \Omega(t,R_0) \} < \kappa n \} \leq  e^{-\kappa n}.
\end{align*}
\end{proposition}

Here the parameter $m$ is the step length. For the sake of simplicity, readers may take $m=1$ on first read of this section.

\Cref{prop:LDP} is a uniform version of a classical LDP in the sense that it is proved for $\Omega(t,R_0)$ with rate $\kappa$ which is uniform over all $t>0$. The idea of its proof is similar to the proof of LDP for an i.i.d. coin flipping process (see \cite[Section 3.1]{DZ10}). The main proposition is an estimate for the probability that $\ell$ fixed bad events occur in $n$ events, which we prove is less than $e^{-\epsilon\ell}$. Then we count the number of ways $\ell$ bad events occur in $n$ events and take the sum of the probabilities. We are interested in the $\ell$'s satisfying $\ell\geq n-\kappa n$, so by taking the parameter $\kappa$ sufficiently small, we obtain the desired LDP. \Cref{prop:LDP} indicates that we need to study whether $T^{jm}(x)$ is in $\Omega(t,R_0)$, and by \cref{lem:goodpartitioncusp}, it is roughly equivalent to study whether the geodesic ray $\{u_{T^{jm}(x)}a_t \in \T(X): t > 0\}$ is in some cusp, which corresponds to a bad event for our dynamical system. The difficulty is that we need to further partition the $\ell$ bad events into a certain union of consecutive bad events, whose geometric picture is that the geodesic trajectory $u_xa_t$ remains in the same cusp without coming back to the compact part. \Cref{prop:badword} is the estimate for such consecutive bad events. We need some preparation before stating \cref{prop:badword}.

\subsection{From cusp to bounds on residual waiting time}
%Let $\kappa>0$ be a constant to be determined later.
We first recall some constants and Lipschitz bounds which will be used often in the rest of the section. By \cref{prop:Coding}, for any $z,z'\in\Delta_0$ and $\gamma\in\calH$, we have
\begin{equation}\label{eq:gammaz}
\|\gamma z-\gamma z'\|\leq \lambda \|z-z'\|
\end{equation}
and also (see \cref{eqn:FaLipEstimate}), for all $z,z'\in\Delta_0$, $l\in \mathbb{N}$, and $\gamma\in\calH^l$, we have
\begin{equation}\label{eq:rl}
|\Ret_{l}(\gamma z)-\Ret_{l}(\gamma z')|\leq C_2\|z-z'\|,
\end{equation}
recalling the constant from \cref{eqn:ConstantC1'C2}. Fix 
\begin{align*}
\lambda_0&{}=\inf \{\Ret(x):x\in \Delta_{\sqcup}\},\\
C_3&{}=\max\{C_2,C_2\operatorname{diam}(\Delta_0)+|\log(\operatorname{diam}(\Delta_0))|\}.
\end{align*}

\begin{definition}[Stopping time, Residual waiting  time]
For all $(y,t)\in \Lambda_+\times \mathbb{R}^+$, its \emph{stopping time} is the unique $l=l(y,t)\in\N$ such that
$\Ret_{l}(y)\leq t< \Ret_{l+1}(y)$, and its \emph{residual waiting  time} is the difference $\Ret_{l+1}(y)-t$.
\end{definition}
When there is no confusion about the pair $(y,t)$, we will often write $l$ to simplify the notation. 

We give the geometric intuition of this notion. \cref{prop:badword}, one of the main estimates, is about studying points $y$ of the form $T^{im}x$. We want to investigate whether the point $y$ is in $\Omega(t,R)$, and by \cref{lem:goodpartitioncusp}, it is roughly equivalent to see whether $u_ya_t$ is in some cusp. The observation is that we can roughly identify $u_ya_{t}$ with $u_{T^ny}a_{t-\Ret_n(y)}$: it follows from the construction of the coding that $u_ya_{\Ret_n(y)}$ is bounded away from $u_{T^ny}$ and they are in the same stable leaf; hence the distance between $u_ya_{t}$ and $u_{T^ny}a_{t-\Ret_n(y)}$ are uniformly bounded for $t\geq \Ret_n(y)$. The stopping time $l = l(y,t)$ is the maximal possible value of $n$ satisfying $t\geq \Ret_n(y)$. The remaining time $t-\Ret_l(y)$ and the next return time $\Ret(T^ly)$ together tell us whether the point $u_ya_t$ is in some cusp or not. The following lemma is a quantitative version of this observation. See \cref{fig:geotrajectory} for illustration.

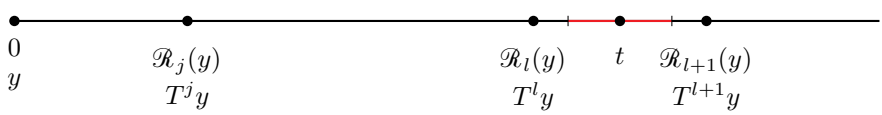
\begin{figure}[h]
\begin{center}
\begin{tikzpicture}[scale=2.3]

\draw [thick, -] (0,0.5) -- (5,0.5);
\draw [Red, thick, -] (3.2,0.5) -- (3.8,0.5);
\draw (3.2,0.53) -- (3.2,0.47);
\draw (3.8,0.53) -- (3.8,0.47);

\path[fill=black,inner sep=0pt] (0,0.5) circle [radius=0.03];
\node[below,inner sep=0pt] at (0, 0.4) {$0$};
\node[below,inner sep=0pt] at (0, 0.2) {$y$};

\draw [Blue, thick, -] (0,0.65) -- (3.5,0.65);
\path[fill=black] (3.5,0.5) circle [radius=0.03];
\node[below] at (3.5, 0.4) {$t$};

\path[fill=black] (1,0.5) circle [radius=0.03];
\node[below] at (1, 0.4) {$\Ret_j(y)$};
\node[below] at (1, 0.2) {$T^jy$};

\path[fill=black] (3,0.5) circle [radius=0.03];
\node[below] at (3, 0.4) {$\Ret_l(y)$};
\node[below] at (3, 0.2) {$T^ly$};

\path[fill=black] (4,0.5) circle [radius=0.03];
\node[below] at (4, 0.4) {$\Ret_{l+1}(y)$};
\node[below] at (4, 0.2) {$T^{l+1}y$};

\end{tikzpicture}
\end{center}
\caption{This is a timeline. For a time of the form $\Ret_j(y)$, we add a point $T^jy$ below, which can be regarded as a point in the geodesic trajectory $u_ya_{\Ret_j(y)}\approx u_{T^jy}$. The blue segment is of length $t$. The red segment is the interval of time $s$ whose corresponding point $u_ya_s$ is inside the cusp region. }\label{fig:geotrajectory}
\end{figure}

\begin{lemma}\label{lem:cuspreturn}
For all $y\notin \Omega(t,R )$, $t>0$, and $R>2C_3$, we have a stronger inequality
\[\Ret_{l}(y)+(R -C_3)< t< \Ret_{l+1}(y)-(R -C_3). \]
\end{lemma}

\begin{proof}
Suppose $y=\gamma_1\cdots \gamma_lz \notin\Omega(t,R )$ with $z\in\Delta_0$ and $\gamma_j\in \calH$. Let $t>0$, and $R>2C_3$. Then by \cref{eq:rl}, for $x\in\Delta_0$ we have
\[ \|(d\gamma_1\cdots \gamma_l)_x\|=\exp(-\Ret_{l}(\gamma_1\cdots\gamma_lx))\in \exp(-\Ret_{l}(y))[e^{-C_2\|x-z\|},e^{C_2\|x-z\|}]. \]
Therefore, by \cref{lem:CylinderEstimate}
\[\diam(\gamma_1\cdots\gamma_l\Delta_0)\in \exp(-\Ret_{l}(y))[e^{-C_3},e^{C_3}]. \]
Since $y\notin\Omega(t,R )$, we have
\[ \diam(\gamma_1\cdots\gamma_l\Delta_0)\notin \exp(-t)[e^{-R },e^{R }]. \] 
Hence, it is impossible that the second interval is contained in the first interval. By the definition of stopping time and the assumption of the lemma, we have $\Ret_{l}(y)+C_3\leq t+R $. Therefore,
\[-\Ret_{l}(y)+C_3> -t+R , \]
which implies $t>\Ret_{l}(y)+R -C_3$. For the other side, the proof is similar.
\end{proof}

Given two points $T^{i}y, T^{r}y\notin \Omega(t,R)$, we can regard $u_{T^{i}y}a_t$ and $u_{T^{r}y}a_t$ as points on  the geodesic trajectory $\{u_ya_s \in \T(X): s > 0\}$. We introduce the following definition to detect whether they are in the same cusp.

\begin{definition}[Equivalent]
\label{def.sametype}
Let $y \in \Lambda_+$, $t > 0$, and $R > 0$. For two points $(T^{i}y, t)$ and $(T^{r}y,t)$ with $T^{i}y, T^{r}y\notin \Omega(t,R)$ for some $i, r \in \Z_{\geq 0}$, we say that they are \emph{equivalent} if 
\[ l(T^{i}y,t)+i=l(T^{r}y,t)+r ,\]
and we use $[(T^iy,t)]$ to denote its equivalence class. Another characterization is that $\Ret_{i+l_i}(y)< t+\Ret_{r}(y)<\Ret_{i+l_i+1}(y) $ with $l_i=l(T^{i}y,t)$ (see \cref{fig:sametype}). From this characterization, we deduce that the set of $k\in\N$ such that $(T^ky,t)$ is equivalent to $(T^ix,t)$ is an interval containing $i$ and $r$. 
\end{definition}

\begin{lemma}\label{lem:cuspdeep}
Let $y \in \Lambda_+$, $t > 0$, and $R > 2C_3$. If $(T^{i}y,t)$ and $(T^ry,t)$ are equivalent with $0 \leq i\leq r$,
then for the stopping time $l=l(T^{i}y,t)$, the residual waiting time satisfies
\[\Ret_{l+1}(T^{i}y)-t>\frac{R }{2}+f\lambda_0,\]
with $f=r-i$.
\end{lemma}

\begin{proof}
Assume the hypotheses of the lemma. By applying \cref{lem:cuspreturn} to $T^ry$, due to the fact that $l+i=l(T^ry,t)+r$, we have
\[\Ret_{l+1-(r-i)}(T^{r}y)-R > t-C_3. \]
This implies 
\[ \Ret_{l+1}(T^{i}y)=\Ret_{l+1-(r-i)}(T^{r}y)+\Ret_{(r-i)}(T^{i}y) > t+(R -C_3)+(r-i)\lambda_0. \]
The proof is complete.
\end{proof}

\begin{figure}[h]
\begin{center}
\begin{tikzpicture}[scale=2.3]

\draw [thick, -] (0,0.5) -- (5,0.5);
\draw [Red, thick, -] (3.2,0.5) -- (4.2,0.5);
\draw (3.2,0.53) -- (3.2,0.47);
\draw (4.2,0.53) -- (4.2,0.47);

\path[fill=black,inner sep=0pt] (0,0.5) circle [radius=0.03];
\node[below,inner sep=0pt] at (0, 0.4) {$0$};
\node[below,inner sep=0pt] at (0, 0.2) {$y$};

\path[fill=black] (1,0.5) circle [radius=0.03];
\node[below] at (1, 0.4) {$\Ret_i(y)$};
\node[below] at (1, 0.2) {$T^iy$};

\draw [Blue, thick, -] (1,0.65) -- (3.5,0.65);
\path[fill=black] (3.5,0.65) circle [radius=0.03];
\node[above] at (3.5, 0.65) {$t+\Ret_i(y)$};

\path[fill=black] (1.5,0.5) circle [radius=0.03];
\node[below] at (1.5, 0.4) {$\Ret_r(y)$};
\node[below] at (1.5, 0.2) {$T^ry$};

\draw [Blue, thick, -] (1.5,0.9) -- (4,0.9);
\path[fill=black] (4,0.9) circle [radius=0.03];
\node[above] at (4, 0.9) {$t+\Ret_r(y)$};

\path[fill=black] (3,0.5) circle [radius=0.03];
\node[below] at (3, 0.4) {$\Ret_{i+l_i}(y)$};
\node[below] at (3, 0.2) {$T^{i+l_i}y$};

\path[fill=black] (4.4,0.5) circle [radius=0.03];
\node[below] at (4.4, 0.4) {$\Ret_{i+l_i+1}(y)$};
\node[below] at (4.4, 0.2) {$T^{i+l_i+1}y$};

\end{tikzpicture}
\end{center}
\caption{The points $(T^iy,t)$ and $(T^ry,t)$ are equivalent. Here $l_i=l(T^iy,t)$, $l_r=l(T^ry,t)$, and $i+l_i=r+l_r$. The blue segments are of length $t$. The red segment is the interval of time $s$ whose corresponding point $u_ya_s$ is inside the cusp region. }\label{fig:sametype}
\end{figure}
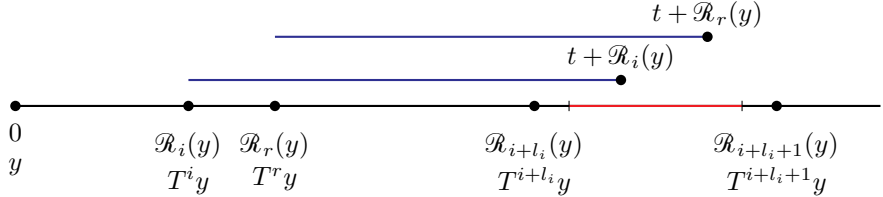

\subsection{Reduction to combinatorial words}\label{sec:reduction_combinatoric}
Let $\bf n:=\{1,\dotsc, n\}$. A word $w$ is a subset of $\bf n$. Let $\ell(w)$ be the number of elements in $w$. For a word $w$, we will partition it into (discrete) subintervals of the form
\[w=\bigcup_{j = 1}^{\jI+1} I_j=\{i_1,\dotsc, r_1 \}\cup\cdots\cup \{i_{\jI+1},\dotsc, r_{\jI+1}\}, \]
with $\jI+1\leq n$, and $i_j\leq r_j<i_{j+1}$ for all $1 \leq j \leq \jI$, and $I_j=\{i_j,i_j+1,\dotsc, r_j\}$ for all $1 \leq j \leq \jI+1$. Here $\jI+1$ is called the number of subintervals in $w$. We denote $I := \{I_j: 1\leq j\leq \jI+1\}$. We use the pair $(w,I)$ to denote the word $w$ and its partition.
For example, if we take $n=6$, then $w = I_1 \cup I_2 =\{1,2\}\cup \{4,5,6\}$ is a word of $\bf n = \{1, \dotsc, 6\}$. Both this partition and the ones like $w=\{1,2\}\cup\{4,5\}\cup\{6 \}$ are the kind of partitions of $w$ we will study.

\begin{definition}[Type]
Let $R > 0$. Given a triple $(w,t,m)$ with $w\subset \bf n$, $t>0$, and $m\in\N$, a point $x\in \Lambda_+$ is of \emph{type} $(w,t,m)$ if
\begin{equation}\label{equ:k in n}
w = \{k\in\bf n: T^{km}x\notin \Omega(t,R)\}. 
\end{equation}

Given a quadruple $(w,I,t,m)$ with $w\subset \bf n$, a partition $I$ of $w$, $t>0$, and $m\in\N$, a point $x\in\Lambda_+$ is of \emph{type} $(w,I,t,m)$ if $x$ is of type $(w,t,m)$ and the partition $I$ of $w$ is determined by the equivalence classes of the points in $\{(T^{km}x,t): k \in w\}$ according to \cref{def.sametype}, i.e., for all $1 \leq j \leq \jI + 1$, there exists $k_j \in w$ such that
\begin{align*}
I_j = \{k \in w: (T^{km}x, t) \in [(T^{k_jm}x, t)]\}.
\end{align*}
\end{definition}

For each quadruple $(w,I,t,m)$, we estimate the measure of the set of $x\in \Lambda_+$ of type $(w,I,t,m)$ using cusp excursion. The main proposition is as follows.

\begin{proposition}
\label{prop:badword}
There exists $\epsilon>0$ such that for all $R > 0$ sufficiently large, $t>0$, and $n,m\in\N$, the following holds. For all quadruple $(w,I,t,m)$ with $w$ a word in $\mathbf{n}$ and a partition $I$ of $w$, we have
\[\nu\{x\in\Lambda_+: x\text{ is of type }(w,I,t,m) \} \leq e^{-\epsilon((\jI+1)R+\ell(w)-\jI-1 )}.\]
\end{proposition}

\begin{proof}[Proof that \cref{prop:badword} implies \cref{prop:LDP}]
Fix $\epsilon > 0$ provided by \cref{prop:badword}. Let $R >0$, $\kappa \in (0, 1)$, $t>0$, and $n,m\in\N$. Due to \cref{def.sametype}, for each point $x \in \Lambda_+$ of type $(w,t,m)$, there exists a partition $I$ of $w$ such that $x$ is of type $(w,I,t,m)$. For each word $w$, the number ways of diving $w$ into $k$ subintervals is less than $\binom{\ell(w)}{k-1}$ for $k\leq \ell(w)$. Therefore by \cref{prop:badword},
\begin{align*}
\nu\{x\in\Lambda_+: x\text{ is of type }(w,t,m) \} &\leq \sum_{k = 1}^{\ell(w)} \binom{\ell(w)}{k-1}e^{-\epsilon(kR+\ell(w) -k)}\\
&\leq e^{-\epsilon\ell(w)}(1+e^{-\epsilon (R-1)})^{\ell(w)}.
\end{align*}
Taking $R$ sufficiently large depending on $\epsilon$, we obtain
\begin{equation}\label{eq:badword}
\nu\{x\in\Lambda_+: x\text{ is of type }(w,t,m) \} \leq e^{-\epsilon\ell(w)/2}.
\end{equation}

A word $w$ is called \emph{bad} if its length satisfies $\ell(w)\geq \lfloor(1-\kappa)n\rfloor$. 
For each bad word $w$ we have the estimate of \cref{eq:badword}. To get the large deviation estimate in \cref{prop:LDP}, it remains to sum over bad words.
The number of bad words is bounded by short sum of binomial coefficients: for any $\ell\leq n$, we have
\[\sum_{k=0}^\ell\binom{n}{k}\leq \sum_{k=0}^\ell \frac{n^k}{k!}=\sum_{k=0}^\ell\frac{\ell^k}{k!}\left( \frac{n}{\ell}\right)^k\leq e^\ell \left( \frac{n}{\ell}\right)^\ell. \]
Therefore by taking $\ell=n-\lfloor(1-\kappa)n\rfloor$, we obtain
\[\sum_{k=\lfloor(1-\kappa)n\rfloor}^n\binom{n}{k}=\sum_{k=0}^{n-\lfloor(1-\kappa)n\rfloor}\binom{n}{k}\leq e^{(\kappa+|\log\kappa|\kappa)n}.   \]
Combined with \cref{eq:badword}, we obtain
\begin{align*}
&\nu\{x\in\Lambda_+: \#\{1\leq j\leq n,\ T^{jm}x\notin\Omega(t,R) \}\geq(1-\kappa)n \}\\
\leq{}&e^{-\epsilon(1-\kappa)n/2}e^{(\kappa+|\log\kappa|\kappa)n}. 
\end{align*}
Taking $\kappa>0$ sufficiently small finishes the proof.
\end{proof}

The aim of the rest of the section is to prove \cref{prop:badword}.

\subsection{Measure estimate of large residual waiting time}
Recall that for any $(x,t)\in\Lambda_+\times \R^+$, we denote by $l=l(x,t)$ the stopping time for $(x,t)$ such that $\Ret_{l}(x)\leq t< \Ret_{l+1}(x)$. The difference $\Ret_{l+1}(x)-t$ is called the residual waiting time.

\begin{proposition}\label{prop:effect}
There exists $\epsilon>0$ such that for all $R>0$ sufficiently large and $t>R$, we have
\[\nu\{x\in\Lambda_+:  \Ret_{l+1}(x)-t>R \}\leq e^{-\epsilon R}. \]
\end{proposition}

The proof of the above proposition will be given in \cref{sec:renewal}. Let us give a direct consequence of the exponential tail property (see Property~(5) in \cref{prop:Coding}).

\begin{proposition}\label{prop:exp_tail} There exists $\epsilon>0$ such that for all $R > 0$ sufficiently large, we have
\[\nu\{x\in\Lambda_+: \Ret(x) > R\}\leq e^{-\epsilon R}. \]
\end{proposition}

\begin{proof}
By Chebyshev's inequality
\begin{align*}
\nu\{x\in\Lambda_+: \Ret(x) > R\} \leq e^{-\epsilon_0 R}\int_{\Lambda_+} e^{\epsilon_0 \Ret(x)} \, \dd\nu(x) \qquad \text{for all $R > 0$}.
\end{align*}
Taking $R$ sufficiently large finishes the proof.
\end{proof}

\begin{proposition}\label{prop:onecusp}
There exists $\epsilon>0$ such that for all $R > 0$ sufficiently large, $i\in\N$, and $t> 0$, 
we have
\[\nu\{x\in\Lambda_+: \Ret_{l+1}(T^ix)-t>R\}\leq e^{-\epsilon R}, \]
where $l = l(T^ix,t)$ is the stopping time for $(T^ix,t)$.
\end{proposition}

\begin{proof}
Since the measure $\nu$ is $T$-invariant, we only need to consider the case when $i=0$.

We divide into two cases: when the time $t$ is large, we use \cref{prop:effect}; when the time $t$ is small, we use \cref{prop:exp_tail}. Fix $\epsilon > 0$ to be the minimum of the $\epsilon$'s provided by the propositions and let $R > 0$ be sufficiently large so that both the proportions hold.

If $t> R$, we can apply \cref{prop:effect} to obtain the desired bound.

Otherwise $t\leq  R$. Note that if $t\geq \Ret_{l}(x)$ and $\Ret_{l+1}(x)-t>R$, then $\Ret(T^{l+1}x)>R$. For a fixed $l$, we can use \cref{prop:exp_tail} and $T$-invariance of $\nu$ to obtain that
\begin{equation}\label{equ:rtl}
\nu\{x\in\Lambda_+: \Ret(T^{l+1}x)> R  \}=\nu\{x\in\Lambda_+: \Ret(x)> R  \}\leq e^{-\epsilon R}.
\end{equation}
Due to $\inf\{\Ret(y):y\in \Lambda_+\}>\lambda_0$ and $R\geq t\geq \Ret_{l}(x)$, the number of possible choices of such $l$ is bounded by $R$. Combined with \cref{equ:rtl}, we obtain
\[\nu\{x\in\Lambda_+: \Ret_{l+1}(x)-t>R\} \leq Re^{-\epsilon R}\leq e^{-\epsilon R/2} \]
which completes the proof by adjusting $\epsilon$.
\end{proof}

\begin{corollary}[Single cusp-excursion]\label{cor:onecusp}
There exists $\epsilon>0$ such that for all $R > 0$ sufficiently large, $t>0$, and $n,m\in\N$, we have
\[\nu\{ x\in \Lambda_+: (x,t) \text{ and } (T^{nm}x,t) \text{ are equivalent}\}\leq  e^{-\epsilon(nm\lambda_0 +R)}.\]
\end{corollary}

\begin{proof}
Let $R$, $t$, $n$, and $m$ be as in the corollary. Let $x \in \Lambda_+$. By \cref{lem:cuspdeep}, if $(x,t)$ and $(T^{nm}x,t)$ are equivalent, we obtain
\[ \Ret_{l+1}(x)-t> \frac{R}{2}+(nm-0)\lambda_0. \]
We finish the proof by using \cref{prop:onecusp}.
\end{proof}

\subsection{Conditional probability}
The way we attain the estimate of multiple cusp-excursions, which is \cref{prop:badword}, from \cref{prop:onecusp} is inspired by the language of conditional probability. 

Given any $n,m\in\N$ and $t > 0$, we consider the quadruple $(w,I,t,m)$ with $w$ a word in $\mathbf{n}$ and $I=\{I_j:1\leq j\leq \jI+1 \}$ a partition of $w$: 
\[w=\bigcup_{j = 1}^{\jI+1} I_j=\{i_1,\dotsc, r_1 \}\cup\cdots\cup \{i_{\jI+1},\dotsc, r_{\jI+1}\}\]
with $i_j\leq r_j<i_{j+1}$ for all $1\leq j\leq \jI$. The proof of \cref{prop:badword} will use induction on the number of subintervals $\jI+1$.

Given a point $x\in \Lambda_+$ of type $(w,I,t,m)$, let $l_j:=l(T^{i_jm}x,t)$ be the stopping time for $(T^{i_jm}x, t)$. For each $1\leq j\leq \jI$, as $(T^{i_jm}x,t)$ and $(T^{i_{j+1}m}x,t)$ are not equivalent, we have 
\begin{equation}\label{eq:ljij}
l_j+i_jm<l_{j+1}+i_{j+1}m. 
\end{equation}
For each $1\leq j\leq \jI+1$ and all $i\in I_j$, the points $(T^{im}x,t)$ are equivalent. Hence, when $\jI+1=1$, \cref{prop:badword} is \cref{cor:onecusp}.

To explain the idea of the proof of \cref{prop:badword}, we describe the probability model of the proposition for comparison. The pair $(\Lambda_+,\nu)$ is a probability space. The $j$-th event happens at a point $x$ if $T^{im}x \notin \Omega(t,R)$ for all $i\in I_j$ and $(T^{im}x,t)$ are equivalent (\cref{def.sametype}) for all $i\in I_j$. The difference of our model is the definition of \emph{all} $\jI+1$ events happening at a point $x\in\Lambda_+$ which requires that $x$ satisfies all $j$-th event for $1\leq j\leq\jI+1$, and \emph{additionally} $(T^{im}x,t)$ are not equivalent for $i$ in the different intervals $I_j$'s.

In order to do induction, we introduce the following definition, which extracts all the information we need from $x$ of type $(w,I,t,m)$. This definition is weaker than \cref{def.sametype}, which helps to do induction.

\begin{definition}\label{def:quartuple}
Given $i,f,m\in\N$, $R >0$, and $t>0$, we say that $x\in \Lambda_+$ satisfies $(i,t,m,R,f)$ if the stopping time $l=l(T^{im}x,t)$ satisfies
\begin{equation}
\label{eq:itrm}
\Ret_{l}(T^{im}x)+R\leq t <\Ret_{l+1}(T^{im}x)-(R+mf\lambda_0).
\end{equation}
\end{definition}

Given any $x\in\Lambda_+$, we know that the expanding map $T^n$ is defined at $x$ for any $n\in\N$, and is of the form $T^nx=\gamma^{-1}x$ for some $\gamma\in\calH^n$. We call the element $\gamma$ the inverse branch of $x$ of length $n$.

\begin{lemma}\label{lem:stable}
Given $i,f,m\in\N$, $R\geq C_3$, and $t>0$, suppose $x\in \Lambda_+$ satisfies $(i,t,m,R,f)$. Let $l$ be the stopping time for $(T^{im}x,t)$. Take the inverse branch $\gamma$ of $x$ of length $k\geq im+l+1$. Then, for all $y\in \gamma\Delta_0$, the stopping time satisfies $l(T^{im}y,t)=l(T^{im}x,t)$.
Moreover, the point $y$ satisfies $(i,t,m,R-C_3\lambda^{k-(im+l+1)} ,f)$.
\end{lemma}

\begin{proof}
Assume the hypotheses of the lemma. Take $l=l(T^{im}x,t)$. As $\gamma$ is the inverse branch of $x$ of length $k$,  we can write $x=\gamma z,\ y=\gamma z'$. Due to \cref{eq:gammaz,eq:rl}, we have
\begin{align*}
|\Ret_{l}(T^{im}y)-\Ret_{l}(T^{im}x)| &= |\Ret_{l}(T^{im}(\gamma z))-\Ret_{l}(T^{im}(\gamma z'))| \\
&\leq C_2\|T^{im+l}(\gamma z)-T^{im+l}(\gamma z')\|\leq C_3\lambda^{k-(im+l)}.
\end{align*}
Similarly,
\[|\Ret_{l+1}(T^{im}y)-\Ret_{l+1}(T^{im}x) |\leq C_3\lambda^{k-(im+l+1)}. \]
By noticing \cref{eq:itrm} for the stopping time $l(T^{im}x,t)$, we have that $y$ satisfies \cref{eq:itrm} with $R$ replaced by $R-C_3\lambda^{k-(im+l+1)}$, which automatically implies that $l(T^{im}x,t)=l(T^{im}y,t)$.
\end{proof}

\begin{proof}[Proof of \cref{prop:badword}]
Let $R$, $t$, $n$, and $m$ be as in the proposition. Take any quadruple $(w,I,t,m)$. We do induction on the number $\jI+1$ of intervals of $I$.

For $\jI+1=1$, it is exactly \cref{cor:onecusp}. The constant $\lambda_0$ is absorbed into $\epsilon$ and notice that $\ell(w)=\sum_{j = 1}^{\jI+1}(f_j+1)=f_1+1$.

For $\jI+1>1$, let $R_1=R/2$ and $f_j=r_j-i_j$. 
Let
\[\cal S_1:=\{x\in \Lambda_+: x \text{ is of type }(w,I,t,m) \}.\]
For the rest of the proof, we will denote by $l_j(x)$ the stopping time $l(T^{i_jm}x,t)$. Let
\begin{equation}
\label{eq:xijlj}
\cal S_2:=\left\{\begin{array}{ll}
x\in \Lambda_+: i_{j}m+l_{j}(x)>i_{j-1}m+l_{j-1}(x) \text{ and}\\
x \text{ satisfies }(i_j,t,m,R_1,f_j) \text{ for } j=1,\dotsc, \jI+1 
\end{array}\right\}.
\end{equation}
By \cref{lem:cuspreturn,lem:cuspdeep}, \cref{eq:ljij} and the definition of type $(w,I,t,m)$, we have $\cal S_1\subset \cal S_2.$ So we only need to bound $\nu(S_2)$.

We want to compute the measure of the set given in \cref{eq:xijlj} using the language of conditional probability. We first fix the events for $j=1,\dotsc, \jI$ by fixing the inverse branch of $x$ and then compute the conditional probability of the last event $j=\jI+1$. 

We need the following before continuing the proof of the proposition. Let $k_{\jI  }(x):=i_{\jI  }m+l_{\jI  }(x)+1$. Let 
\begin{equation}
\calH(w,I,t,m):=\left\{
\begin{array}{ll}
\gamma \in \bigcup_{n \in \N}\calH^{n}:  \exists x\in\cal S_2 
\text{ such that }\\
\text{ the inverse branch of }
x\text{ of length }k_\jI(x)\text{ is } 
\gamma
\end{array}
\right\}.
\end{equation}

\begin{lemma}\label{lem:gammadisj}
The set $\{\gamma\Delta_0: \gamma\in\calH(w,I,t,m)\}$ consists of mutually disjoint cylinders, i.e., there are no $\gamma,\gamma' \in \calH(w,I,t,m)$ such that $\gamma=\gamma'\gamma''$ for some $\gamma'' \in \bigcup_{n \in \N}\calH^{n}$.
\end{lemma}

\begin{proof}
For the sake of contradiction, suppose the lemma is false. Then, we have $\gamma=\gamma'\gamma''$ for some $\gamma,\gamma'\in\calH(w,I,t,m)$ and $\gamma'' \in \bigcup_{n \in \N}\calH^{n}$. By definition, $\gamma$ (resp. $\gamma'$) is the inverse branch of some $x \in \cal S_2$ (resp. $x' \in \cal S_2$) of length 
\[k_\jI(x)=i_\jI m+l_\jI(x)+1 \qquad \text{(resp. $k_\jI(x')=i_\jI m+l_\jI(x') +1$)}.\]
We denote by $l_\jI$ (resp. $l_\jI'$) the stopping time $l_\jI(x)$ (resp. $l_\jI(x')$).
Also due to $x$ and $x'$ satisfying $(i_\jI,t,R_1,f_\jI)$, we have
\[ \Ret_{l_{\jI  }}(T^{i_{\jI  }m}x)\leq t-R_1 \qquad \text{and} \qquad \Ret_{l_{\jI  }'+1}(T^{i_{\jI  }m}x')-t>R_1+f_{\jI  }\lambda_0. \]
But by hypothesis $\gamma=\gamma'\gamma''$, and so we have $l_{\jI  }'+1\leq l_{\jI  }$. Note that $x=\gamma z=\gamma'\gamma''z$ and $x'=\gamma'z'$.
Therefore
\[\Ret_{l_{\jI  }'+1} (T^{i_{\jI  }m}\gamma'z')\leq \Ret_{l_{\jI  }'+1}(T^{i_{\jI  }m}\gamma'\gamma'' z)+C_3\leq \Ret_{l_{\jI  }}(T^{i_{\jI  }m}\gamma z)+C_3\leq t-R_1+C_3, \]
where the first inequality is due to \cref{eq:rl}. This is a contradiction.
\end{proof}
Then by definition and \cref{lem:gammadisj},
\begin{align*}
\nu(\cal S_2)&=\nu\left\{\begin{array}{ll}
x\in\Lambda_+: i_{j}m+l_{j}(x)>i_{j-1}m+l_{j-1}(x) \text{ and}\\ 
x \text{ satisfies }(i_j,t,m,R_1,f_j) \text{ for } j=1,\dotsc, \jI+1 
\end{array}\right\}\\
&\leq  \sum_{\gamma\in\calH(w,I,t,m)  }\nu\left\{
\begin{array}{ll}
\gamma z\in \gamma\Lambda_+: i_{\jI+1}m+l_{\jI+1}(\gamma z)>i_{\jI  }m+l_{\jI  }(\gamma z) \\
\text{and }\gamma z\text{ satisfies }(i_{\jI+1},t,m,R_1,f_{\jI+1}) 
\end{array}
\right\}.
\end{align*}
From here we can see that in order to estimate $\nu(\cal S_2)$, we need to estimate the following quantities:
\begin{enumerate}
\item[I.] the probability of the first $\jI$ events:
\[\sum_{\gamma\in\calH(w,I,t,m)  }\nu(\gamma\Delta_0);\]
\item[II.] the conditional probability when the first $\jI$ events happen:
\begin{equation}\label{equ:estimate 2}
\frac{1}{\nu(\gamma\Delta_0)} \cdot \nu\left\{
\begin{array}{ll}
\gamma z\in \gamma\Lambda_+: i_{\jI+1}m+l_{\jI+1}(\gamma z)>i_{\jI  }m+l_{\jI  }(\gamma z) \\
\text{and }\gamma z\text{ satisfies }(i_{\jI+1},t,m,R_1,f_{\jI+1}) 
\end{array}
\right\}.
\end{equation}
\end{enumerate}

\subsubsection{Estimate of I}
Take $\gamma\in \calH(w,I,t,m)$. There exists $x'=\gamma z'\in \cal S_2$. By \cref{lem:stable}, for all $x=\gamma z$ with $z\in\Lambda_+$, its stopping time is $l_j(x) = l_j(x')$ for all $1\leq j\leq\jI$. Such $x$ belongs to the set
\begin{equation}\label{eq:x'ij+1}
\left\{\begin{array}{ll}
x\in\gamma\Lambda_+: i_{j}m+l_{j}(x)>i_{j-1}m+l_{j-1}(x) \text{ and}\\
x\text{ satisfies }(i_j,t,m, R_1-C_3\lambda^{k_{\jI  }-(i_jm+l_j+1)},f_j) \text{ for } j=1,\dotsc, \jI\end{array}  
\right\}.
\end{equation}
Therefore, by \cref{eq:x'ij+1} and \cref{lem:gammadisj}, 
\begin{align}\label{eq:1jw}
&\sum_{\gamma\in\calH(w,I,t,m)  }\nu(\gamma\Delta_0)\nonumber\\
\leq{}& \nu \left\{\begin{array}{ll}
x\in\Lambda_+:
\ i_{j}m+l_{j}(x)>i_{j-1}m+l_{j-1}(x) \text{ and}\\
x\text{ satisfies }(i_j,t,m, R_1-C_3\lambda^{k_{\jI  }-(i_jm+l_j+1)},f_j) \text{ for } j=1,\dotsc, \jI\end{array}   \right\}.
\end{align}

\subsubsection{Estimate of II}
Now we estimate the ``conditional probability under $\gamma$". Let us be more precise. Let $\gamma\in \calH(w,I,t,m)$. 
For any $x=\gamma z$ with $z\in\Lambda_+$ satisfying the condition given in \cref{equ:estimate 2}, we rewrite the condition in terms of $z\in\Lambda_+$; then we can apply \cref{prop:onecusp}. Note that there exists $x'\in \gamma z'\in \mathcal{S}_2$. By \cref{lem:stable}, for all $x=\gamma z$ with $z\in\Lambda_+$, the stopping time $l_j(x)$ equals to $l_j(x')$ for $1\leq j\leq\jI$. We will denote the common number $l_j(x)$ by $l_j$.
Recall that $k_{\jI  }=i_{\jI}m+l_{\jI}+1$ is the length of $\gamma$.
\begin{enumerate}[label=(\roman*)]
\item If $k_{\jI  }\leq i_{\jI+1}m$, then $T^{i_{\jI+1}m}\gamma z=T^{i_{\jI+1}m-k_{\jI  }}z$. Hence, the point $z$ satisfies $\left(\frac{i_{\jI+1}m-k_{\jI  }}{m},t,m,R_1,f_{\jI+1} \right)$. Here we abuse notation. Although $\frac{i_{\jI+1}m-k_{\jI  }}{m}$ may not be an integer, $\left(\frac{i_{\jI+1}m-k_{\jI  }}{m},t,m,R_1,f_{\jI+1} \right)$ makes sense under \cref{def:quartuple}.
\item If $k_{\jI  }>i_{\jI+1}m$, let $l'=i_{\jI+1}m+l_{\jI+1}-k_{\jI  }$. Since $i_{\jI+1}m+l_{\jI+1}>i_{\jI}m+l_{\jI}$, we know $l'\geq 0$.
First fix $z_0 \in \Lambda_+$ and let $t_\gamma=t-\Ret_{k_{\jI  }-i_{\jI+1}m}(T^{i_{\jI+1}m}\gamma z_0)$.  \emph{We will verify that $l'$ is the new stopping time for $(z=T^{k_{\jI  }}x,t_\gamma)$ and $z$ satisfies $(0,t_\gamma,m,R_1-C_3,f_{\jI+1})$.}

Due to the fact that
\[\Ret_{l_{\jI+1}+1}(T^{i_{\jI+1}m}\gamma z)=\Ret_{l'+1}(z)+\Ret_{k_{\jI  }-i_{\jI+1}m}(T^{i_{\jI+1}m}\gamma z), \]
we obtain
\begin{align*}
&\Ret_{l'+1}(z)-t_\gamma\\
={}&\Ret_{l_{\jI+1}+1}(T^{i_{\jI+1}m} \gamma z)-t-(\Ret_{k_{\jI  }-i_{\jI+1}m}(T^{i_{\jI+1}m}\gamma z)-\Ret_{k_{\jI  }-i_{\jI+1}m}(T^{i_{\jI+1}m}\gamma z_0)) \\
\geq{}& R_1+f_{\jI+1}\lambda_0-C_3, 
\end{align*}
where the last inequality is due to \cref{eq:rl} and the fact that $x$ satisfies $(i_{\jI+1},t,m,R_1,f_{\jI+1})$.
Similarly,
\[  \Ret_{l'}(z)+R_1-C_3<t_\gamma. \]
\end{enumerate}

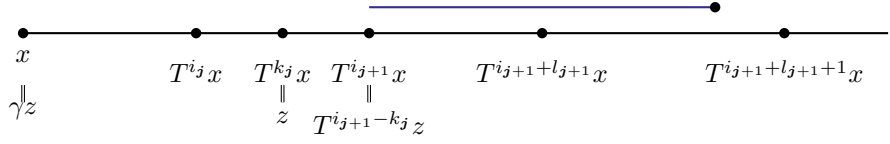
\begin{figure}[h]
\begin{center}
\begin{tikzpicture}[scale=2.3]

\draw [thick, -] (0,0.5) -- (5,0.5);
%	\draw [red, thick, -] (3.2,0.5) -- (4.2,0.5);

\path[fill=black,inner sep=0pt] (0,0.5) circle [radius=0.03];
%	\node[below] at (0, 0.4) {$0$};
\node[below,inner sep=0pt] at (0, 0.4) {$x$};
\draw [-,double] (0,0.2) -- (0,0.1);
%		\draw [-] (0.02,0.25) -- (0.02,0.2);
\node[below,inner sep=0pt] at (0, 0.1) {$\gamma z$};

\path[fill=black] (1,0.5) circle [radius=0.03];
%	\node[below] at (1, 0.4) {$\Ret_i(x)$};
\node[below] at (1, 0.4) {$T^{i_\jI}x$};

%	\draw [blue, -] (1,0.7) -- (3.5,0.7);
%		\path[fill=black] (3.5,0.7) circle [radius=0.03];
%	\node[above] at (3.5, 0.7) {$t+\Ret_i(x)$};

\path[fill=black] (1.5,0.5) circle [radius=0.03];
%	\node[below] at (1.5, 0.4) {$\Ret_r(x)$};
\node[below] at (1.5, 0.4) {$T^{k_\jI}x$};
\draw [-,double] (1.5,0.2) -- (1.5,0.1);
\node[below] at (1.5, 0.1) {$ z$};

\path[fill=black] (2,0.5) circle [radius=0.03];
%	\node[below] at (1.5, 0.4) {$\Ret_r(x)$};
\node[below] at (2, 0.4) {$T^{i_{\jI+1}}x$};
\draw [-,double] (2,0.2) -- (2,0.1);
\node[below] at (2, 0.1) {$ T^{i_{\jI+1}-k_\jI } z$};

\draw [Blue, thick, -] (2,0.65) -- (4,0.65);
\path[fill=black] (4,0.65) circle [radius=0.03];
%	\node[above] at (4, 0.9) {$t+\Ret_r(x)$};

\path[fill=black] (3,0.5) circle [radius=0.03];
%	\node[below] at (3, 0.4) {$\Ret_{i+l_i}(x)$};
\node[below] at (3, 0.4) {$T^{i_{\jI+1}+l_{\jI+1}}x$};

\path[fill=black] (4.4,0.5) circle [radius=0.03];
%	\node[below] at (4.4, 0.4) {$\Ret_{i+l_i+1}(x)$};
\node[below] at (4.4, 0.4) {$T^{i_{\jI+1}+l_{\jI+1}+1}x$};

\end{tikzpicture}
\end{center}
\caption{Case (i) of estimate of II. It is similar to  \cref{fig:geotrajectory}. To simplify the notation, we only write the points $T^ix$, which also represents $\Ret_i(x)$ in the timeline. The blue segment is of length $t$. }\label{fig:case1}
\end{figure}

\begin{figure}[h]
\begin{center}
\begin{tikzpicture}[scale=2.3]

\draw [thick, -] (0,0.5) -- (5,0.5);
%	\draw [red, thick, -] (3.2,0.5) -- (4.2,0.5);

\path[fill=black,inner sep=0pt] (0,0.5) circle [radius=0.03];
%	\node[below] at (0, 0.4) {$0$};
\node[below,inner sep=0pt] at (0, 0.4) {$x$};
\draw [-,double] (0,0.2) -- (0,0.1);
%		\draw [-] (0.02,0.25) -- (0.02,0.2);
\node[below,inner sep=0pt] at (0, 0.1) {$\gamma z$};

\path[fill=black] (1,0.5) circle [radius=0.03];
%	\node[below] at (1, 0.4) {$\Ret_i(x)$};
\node[below] at (1, 0.4) {$T^{i_\jI}x$};

%	\draw [blue, -] (1,0.7) -- (3.5,0.7);
%		\path[fill=black] (3.5,0.7) circle [radius=0.03];
%	\node[above] at (3.5, 0.7) {$t+\Ret_i(x)$};

\path[fill=black] (2,0.5) circle [radius=0.03];
%	\node[below] at (2, 0.4) {$\Ret_r(x)$};
\node[below] at (2, 0.4) {$T^{k_\jI}x$};
\draw [-,double] (2,0.2) -- (2,0.1);
\node[below] at (2, 0.1) {$ z$};

\path[fill=black] (1.5,0.5) circle [radius=0.03];
%	\node[below] at (1.5, 0.4) {$\Ret_r(x)$};
\node[below] at (1.5, 0.4) {$T^{i_{\jI+1}}x$};
\draw [-,double] (1.5,0.2) -- (1.5,0.1);
\node[below] at (1.5, 0.1) {$ T^{i_{\jI+1}} \gamma z$};

\draw [Blue, thick, -] (1.5,0.8) -- (4,0.8);
\path[fill=black] (4,0.8) circle [radius=0.03];
%	\node[above] at (4, 0.9) {$t+\Ret_r(x)$};

\draw [Orange, thick, -] (2,0.65) -- (4,0.65);
\path[fill=black] (4,0.65) circle [radius=0.03];
%	\node[above] at (4, 0.9) {$t+\Ret_r(x)$};

\draw [Brown, thick, -] (1.5,0.65) -- (2,0.65);
\path[fill=black] (2,0.65) circle [radius=0.03];
%	\node[above] at (4, 0.9) {$t+\Ret_r(x)$};

\path[fill=black] (3,0.5) circle [radius=0.03];
%	\node[below] at (3, 0.4) {$\Ret_{i+l_i}(x)$};
\node[below] at (3, 0.4) {$T^{i_{\jI+1}+l_{\jI+1}}x$};

\path[fill=black] (4.4,0.5) circle [radius=0.03];
%	\node[below] at (4.4, 0.4) {$\Ret_{i+l_i+1}(x)$};
\node[below] at (4.4, 0.4) {$T^{i_{\jI+1}+l_{\jI+1}+1}x$};

\end{tikzpicture}
\end{center}
\caption{Case (ii) of estimate of II. The brown line is of length approximately $\Ret_{k_{\jI  }-i_{\jI+1}}(T^{i_{\jI+1}}\gamma z_0)$ and the orange line is of length approximately $t_\gamma=t-\Ret_{k_{\jI  }-i_{\jI+1}}(T^{i_{\jI+1}}\gamma z_0)$.}\label{fig:case2}
\end{figure}
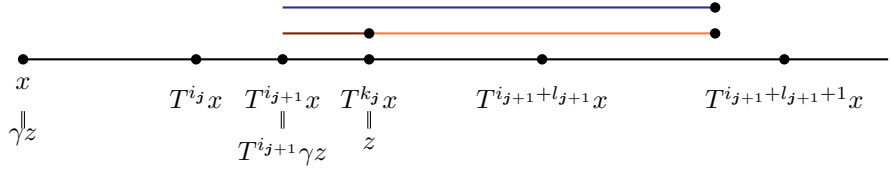

If the length of $\gamma$ satisfies case (ii), then we have
\begin{equation}\label{eq:mgammay}
\begin{split}
&\nu \left\{\begin{array}{ll}x=\gamma z\in \gamma \Lambda_+: x\text{ satisfies }(i_{\jI+1},t,R_1,f_{\jI+1})\\ \text{and }i_{\jI+1}m+l_{\jI+1}>i_{\jI  }m+l_{\jI  }\end{array}\right\}\\
\leq{}&\nu  \left(\{x=\gamma z\in \gamma \Lambda_+,\ z\text{ satisfies }(0,t_\gamma,R_1-C_3,f_{\jI+1})\}\right)\\
\leq{}&C_{\mathrm{cyl}}\|d\gamma\|^{\delta}\nu(\{ z\in\Lambda_+: z \text{ satisfies }(0,t_\gamma,R_1-C_3,f_{\jI+1})\}) \qquad \text{(by \cref{lem:CylinderEstimate})}\\
\leq{}& C_{\mathrm{cyl}}\|d\gamma\|^{\delta} e^{-\epsilon(R_1-C_3+f_{\jI+1}m\lambda_0)} \qquad \text{(by \cref{prop:onecusp})} \\
\leq{}& C_4\nu(\gamma\Lambda_+)e^{-\epsilon(R_1-C_3+f_{\jI+1}m\lambda_0)} \qquad \text{(by \cref{lem:CylinderEstimate})},
\end{split}
\end{equation}
with $C_4=C_{\mathrm{cyl}}^2$. 

If the length of $\gamma$ satisfies case (i), then the second equation above is for $z$ which $\text{satisfies }\left(\frac{i_{\jI+1}m-k_{\jI  }}{m},t,m,R_1 ,f_{\jI+1}\right) $. By the same estimate as above, we have
\begin{align}\label{eq:mgammay'}
\begin{split}
&\nu \left\{\begin{array}{ll}x=\gamma z\in \gamma \Lambda_+: x \text{ satisfies }(i_{\jI+1},t,R_1,f_{\jI+1})\\
\text{and } i_{\jI+1}m+l_{\jI+1}>i_{\jI  }m+l_{\jI  }\end{array}\right\}\\
\leq{}&  C_4\nu(\gamma\Lambda_+)e^{-\epsilon(R_1+f_{\jI+1}m\lambda_0)}.
\end{split}
\end{align}

\subsubsection{\texorpdfstring{Completing the estimate of $\nu(\cal S_2)$}{Completing the estimate of ν(\unichar{"1D4AE}₂)}}
Summing over all $\gamma\in\calH(w,I,t,m)$, we have
\begin{align*}
\nu(\cal S_2)
= {}& \sum_{\gamma\in\calH(w,I,t,m)  }\nu\left\{\begin{array}{ll}\gamma z\in \gamma \Lambda_+: \gamma z \text{ satisfies }(i_{\jI+1},t,m,R_1,f_{\jI+1})\\
\text{and } i_{\jI+1}m+l_{\jI+1}>i_{\jI  }m+l_{\jI  }\end{array} \right\}
\\ 
\leq{}&   \sum_{\gamma\in\calH(w,I,t,m)  } \nu(\gamma \Lambda_+)(C_4e^{-\epsilon(R_1-C_3+f_{\jI+1}m\lambda_0)}) \,\,\,\text{ (by \cref{eq:mgammay,eq:mgammay'})}
\\
\leq{} &\nu \left\{\begin{array}{ll}x\in \Lambda_+: x\text{ satisfies }(i_j,t, m,R_1-C_3\lambda^{k_{\jI  }-(i_jm+l_j+1)},f_j) \text{ and}\\
i_{j}m+l_{j}(x)>i_{j-1}m+l_{j-1}(x) \text{ for } j=1,\dotsc, \jI  \end{array} \right\} \\
{}&\times (C_4e^{-\epsilon(R_1-C_2+f_{\jI+1}m\lambda_0)}) \,\,\,\text{ (by \cref{eq:1jw})}.
\end{align*}

For the first term in the last equation, we use induction to estimate it. We repeat the process and reduce the number $\jI$. At each step, we will get a term similar to the second term of the last equation. The subtlety is that after the first step, the number $R_1$ for events $j=1,\dotsc, \jI$ changes, and the error is of the form $C_3\lambda^{k_{\jI  }-(i_jm+l_j+1)}$. Due to $\lambda<1$, for a fixed $j$, after $\jI+1-j$ steps, the number $R_1$ is replaced by 
\[R_1-C_3\sum_{r = j}^{\jI}  \lambda^{k_{r  }-(i_jm+l_j+1)},\]
which is greater than $ R_1-C_3/(1-\lambda)$,
where $k_r$'s are mutually distinct due to the hypothesis that $i_rm+l_r>i_{r-1}m+l_{r-1}$. 
Then, we use the argument of conditional probability to compute the measure of the $j$-th event. Altogether, we obtain
\begin{align*}
\nu(\cal S_2)\leq \prod_{j = 1}^{\jI+1} C_4e^{-\epsilon(f_{j}m\lambda_0+R_1-C_3(1+\frac{1}{1-\lambda}))}.
\end{align*}
Recall that $\ell(w)=\sum_{j = 1}^{\jI+1}(f_j+1)=\sum_{j = 1}^{\jI+1}f_j + (\jI+1)$.
Finally, we obtain \cref{prop:badword} by using $m\geq 1$, decreasing $\epsilon$ to $\epsilon \lambda_0/2$, and choosing $R_1=R/2$ sufficiently large to absorb the constants $C_3,C_4$.
\end{proof}

\subsection{Renewal theorem}\label{sec:renewal}

In this section, we will prove \cref{prop:effect}. For references of the renewal theorem, see for example \cite{Woo82} and \cite{Fel71}. We define a ``random walk" from $\Lambda_+\times \R$ to $\Lambda_+\times \R$: starting from a point $(x,t) \in \Lambda_+\times \R$, the probability to reach the next point $(y,t+\Ret(y))$ with $Ty=x$ is $e^{-\delta \Ret(y)}$. Then the integral form of the renewal theorem for this ``random walk" is given in \cref{thm:renewal}. This kind of renewal theorem can be obtained through the classical method of using transfer operators and Laplace transforms. In order to obtain an exponential error term in \cref{prop:effect}, we need a spectral gap result from \cite{LP22}.

\Cref{prop:effect} is for estimating the measure of the set of $x\in\Lambda_+$ whose residual waiting time satisfies $\Ret_{l+1}(x)-t>R$ where $l=l(x,t)$ is the stopping time. In \cref{lem:residual}, we use a classical computation to relate the residual waiting time to the integral of the renewal sum in \cref{eqn:renewal sum} which appears in \cref{thm:renewal}, the integral form of the renewal theorem. This enables us to derive \cref{prop:effect} from \cref{thm:renewal}.

Only for this subsection, denote $\L_s:=\calL_{(\delta+s)\Ret}$ for $s\in\C $ with $\Re s > -\epsilon_0$, which is the unnormalized transfer operator on $\Lip(\Lambda_+,\mathbb{C})$ defined in \cref{sec:TransferOperators}. We also use the measure $\mu_{\mathrm{E}}$ which satisfies $L^*_0(\mu_{\mathrm{E}}) = \mu_{\mathrm{E}}$. Recall $h_0 \in \Lip(\Lambda_+, \R)$ is the eigenfunction of the transfer operator $\L_{0}$ with eigenvalue $1$ and $\int_{\Lambda_+} h_0 \, \dd\mu_{\mathrm{E}}=1$. Let $\sigma_0:=\bar{\Ret}=\hat{\nu}(\Ret)=\int_{\Lambda_+} \Ret h_0 \, \dd\mu_{\mathrm{E}}$ be the Lyapunov exponent.

For any compactly supported function $f:\R \to \R$, $x\in\Lambda_+$, and $t>0$, we introduce the renewal sum
\begin{align}
\label{eqn:renewal sum}
\calR f(x,t)&:=\sum_{n=0}^\infty \sum_{\gamma \in \calH^n}e^{-\delta \Ret_{n}(\gamma x)}f(\Ret_{n}(\gamma x)-t)
\end{align}
which converges. We denote by $\|\cdot\|_1$ the $L^1$ norm.

\begin{theorem}[Renewal theorem]
\label{thm:renewal}
There exists $\epsilon>0$ such that for all $f \in C_{\mathrm{c}}^2(\R, \R)$, $x\in\Lambda_+$, and $t>0$, we have
\begin{equation*}
\calR f(x,t) =\frac{h_0(x)}{\sigma_0}\int_{-t}^\infty f \, \dd \Leb + O\bigl(e^{-\epsilon t}e^{\epsilon \Leb(\supp f)}(\|f''\|_1+\|f\|_1)\bigr)
\end{equation*}
as $t \to +\infty$.
\end{theorem}

By \cref{lem:analyticity}, we know that $s \mapsto \L_{s}$ is analytic on $\{s \in \C: \Re s>-\frac{\epsilon_0}{2}\}$, where $\epsilon_0$ comes from the exponential tail property in \cref{prop:Coding}. We need a lemma about the operator $(\Id-\L_{s})^{-1}=\sum_{n = 0}^\infty\L_{s}^n$. The proof is similar to the case of finite symbolic coding; see for example \cite[Propositions 7.2 and 7.3]{Lal89}. \Cref{lem:1-lz} is stronger than the results in \cite{Lal89} because we obtain a meromorphic extension to a half plane $\{s \in \C: \Re s>-\eta\}$ and the estimate of the norm of $U(s)$ is uniform up to a polynomial term in $s$.

\begin{lemma}\label{lem:1-lz}
We have a meromorphic extension of the map $s \mapsto (\Id-\L_{s})^{-1}$, which we denote by the same symbol, from the domain $\{s\in\C: \Re s>0\}$ to the domain $\{s\in\C: \Re s>-\eta\}$ for some $\eta \in (0, \frac{\epsilon_0}{2})$. Moreover, we have the decomposition
\begin{equation}\label{equ:decomposition}
(\Id-\L_{s})^{-1}=\frac{N(0)}{\sigma_0 s}+U(s) \qquad \text{for all $|\Re s| < \eta$},
\end{equation}
where $N(0)\phi = \mu_{\mathrm{E}}(\phi) h_0$ for all $\phi \in \Lip(\Lambda_+, \C)$ and $s \mapsto U(s)$ is analytic and satisfies
\begin{equation}\label{eqn:UsBound}
\|U(s)\|_{\mathrm{op}}\leq C(1+|\Im s|)^2 \qquad \text{for all $|\Re s| < \eta$}.
\end{equation}
\end{lemma}

\begin{proof}
First, we deal with $s \in \C$ with small $|s|$. It follows from the characterization of the spectrum of $\L_{0}$ (see \cref{sec:TransferOperators}) and its spectral decomposition that it is quasi-compact (see \cite[Chapter 3]{Kat95}). By further using perturbation theory of operators (see \cite[Chapter 7]{Kat95}), we obtain 
\[ \L_{s}=\lambda_sN(s)+Q(s) \qquad \text{for all $s \in \mathcal{O}$}, \]
for some open ball $\mathcal{O} \subset \C$ centered at $0$ and analytic maps $s \mapsto \lambda_s$, $s \mapsto N(s)$, and $s \mapsto Q(s)$. Here, $\lambda_s$ is the maximal simple eigenvalue of $\L_{s}$ and the operators satisfy $N(s)Q(s)=Q(s)N(s)=0$, $N(s)^2=N(s)$ (i.e., a projection), and $\|Q(s)\|_{\mathrm{op}}\leq \rho<1$.
Therefore,
\[(\Id-\L_{s})^{-1}=\sum_{n = 0}^\infty \L_{s}^n=\sum_{n = 0}^\infty(\lambda_s^nN(s)+Q(s)^n)= \frac{N(s)}{1-\lambda_s}+\sum_{n = 0}^\infty Q(s)^n.\]
In order to obtain \cref{equ:decomposition}, we need to compute $\partial_s\lambda_s|_{s=0}$ (cf. \cite[Proposition 4.10]{PP90} and \cite[Lemma 7.22]{AGY06}). Let $h_s \in \Lip(\Lambda_+, \C)$ be the eigenfunction of $L_s$ with eigenvalue $\lambda_s$ and $\int_{\Lambda_+} h_s \, d\mu_{\mathrm{E}}=1$ so that $s \mapsto h_s$ is analytic. In the following computation, we use $\dot{L}_s$ and $\dot{\lambda}_s$ to denote their derivatives with respect to the variable $s$.
By taking the derivative of $\lambda_sh_s = \L_{s}h_s$ and integrating, we have
\begin{align*}
\int_{\Lambda_+} (\dot{\lambda}_sh_s+\lambda_s\dot{h}_s) \, \dd\mu_{\mathrm{E}} = \int_{\Lambda_+} (\dot{\L}_{s}h_s+\L_{s}\dot{h}_s) \, \dd\mu_{\mathrm{E}}.
\end{align*}
Then set $s=0$ in the above formula. Since $L^*_0(\mu_{\mathrm{E}}) = \mu_{\mathrm{E}}$, we have $\int_{\Lambda_+} \dot{h}_0 \, \dd\mu_{\mathrm{E}}=\int_{\Lambda_+} \L_0\dot{h}_0 \, \dd\mu_{\mathrm{E}}$, and so we obtain
\[\partial_s\lambda_s|_{s=0}=\dot{\lambda}_0=\int_{\Lambda_+} \dot\L_0h_0 \, \dd\mu_{\mathrm{E}}= -\int_{\Lambda_+}\L_0(\Ret h_0) \, \dd\mu_{\mathrm{E}}
=-\sigma_0.\]
Hence, the map $s \mapsto N(s)/(1-\lambda_s)$ has a simple pole at $0$ and
\[s \mapsto \frac{N(s)}{1-\lambda_s}-\frac{N(0)}{\sigma_0 s}\]
is analytic on $\mathcal{O}$. Thus, $s \mapsto U(s):=(\Id-\L_{s})^{-1}-N(0)/(\sigma_0s)$ is analytic and \cref{equ:decomposition} holds on $\mathcal{O}$. \Cref{eqn:UsBound} is trivial.

Now, we deal with $s \in \C$ with bounded $|\Re s|$ and large $|\Im s|$. Using the spectral bounds provided by \cite[Proposition 7.3]{LP22} and an argument similar to the one after \cite[Proposition 5.3]{Nau05} or the proof of \cite[Proposition 7.16]{AGY06}, we obtain the following: there exist $C > 0$, $\eta \in (0, \diam(\mathcal{O})/2)$, $b_0>0$, and $\rho \in (0, 1)$ such that
\[ \|\L_{s}^n\|_{\mathrm{op}}\leq C(1+|\Im s|)^2\rho^n \qquad \text{on $\{s\in\C: |\Re s|<\eta,\ |\Im s|>b_0\}$}. \]
This implies that $\Id-\L_{s}$ is invertible and $s \mapsto (\Id-\L_{s})^{-1}$ is analytic on $\{s\in\C: |\Re s|<\eta,\ |\Im s|>b_0\}$, and its norm satisfies
\[ \|(\Id-\L_{s})^{-1} \|_{\mathrm{op}}\leq C\frac{(1+|\Im s|)^2}{1-\rho}. \]
Therefore, $s \mapsto U(s):=(\Id-\L_{s})^{-1}-N(0)/(\sigma_0s)$ is analytic and \cref{equ:decomposition,eqn:UsBound} hold on $\{s\in\C: |\Re s|<\eta,\ |\Im s|>b_0\}$.

For the rest of the region in $\{s \in \C: |\Re s|< \eta\}$ which is bounded away from $0$ and $\infty$, 
by the same argument as in \cite[Proposition 7.3]{Lal89} and \cite[Lemma 7.21]{AGY06}, the spectral radius of the operator $L_s$ is less than $1$ and $s \mapsto (\Id-\L_{s})^{-1}$ is analytic, which finishes the proof.
\end{proof}

\begin{proof}[Proof of \cref{thm:renewal}]
Let $f$, $x$, and $t$ be as in the lemma. Using \cref{eqn:renewal sum}, the Fourier transform, and the dominated convergence theorem, we obtain
\begin{align*}
\calR f(x,t)&=\sum_{n=0}^\infty \sum_{\gamma \in \calH^n}e^{-\delta \Ret_{n}(\gamma x)}\int_{-\infty}^{+\infty} e^{i\xi(\Ret_{n}(\gamma x)-t)}\hat f(\xi) \, \dd\xi\\
&=\sum_{n=0}^\infty\int_{-\infty}^{+\infty} e^{-it\xi} \hat f(\xi)\L_{-i\xi}^n(\chi_{\Lambda_+})(x) \, \dd\xi \\
&=\int_{-\infty}^{+\infty} e^{-it\xi} \hat f(\xi) (\Id-\L_{-i\xi})^{-1}(\chi_{\Lambda_+})(x) \, \dd\xi.
\end{align*}	
Using \cref{lem:1-lz}, we continue as in \cite[Proposition 4.27]{Li22}, which is a version of the classical Paley--Wiener theorem, which finishes the proof.
\end{proof}

Before we begin the proof of \cref{prop:effect}, we need a lemma which relates the residual waiting time to the integral of the renewal sum.

\begin{lemma}\label{lem:residual}
For all $R > 0$ and $t > 0$, we have
\begin{align*}
\mu_{\mathrm{E}}\{x\in\Lambda_+: \Ret_{l+1}(x)-t>R \}=\int_{\Lambda_+} \calR f_x(x,t) \, \dd\mu_{\mathrm{E}}(x),
\end{align*}
where $f_x := \chi_{\{s \in \R: R-\Ret(x)<s\leq 0\}}$ for all $x \in \Lambda_+$.
\end{lemma}
\begin{proof}
Let $R > 0$ and $t > 0$. Using $\L_{0}^*(\mu_{\mathrm{E}}) = \mu_{\mathrm{E}}$, we obtain
\begin{align*}
&\mu_{\mathrm{E}}\{x \in \Lambda_+: \Ret_{l+1}(x)-t>R \}\\
={}&\sum_{n = 0}^\infty\int_{\Lambda_+} \chi_{\{y \in \Lambda_+: \Ret_{n}(y) \leq t < \Ret_{n+1}(y)-R\}}(x) \, \dd\mu_{\mathrm{E}}(x)\\
={}&\sum_{n = 0}^\infty\int_{\Lambda_+} \L_{0}^n\bigl(\chi_{\{y \in \Lambda_+: \Ret_{n}(y) \leq t < \Ret_{n+1}(y)-R\}}\bigr)(x) \, \dd\mu_{\mathrm{E}}(x)\\
={}&\sum_{n = 0}^\infty\int_{\Lambda_+} \sum_{\gamma \in \calH^n}e^{-\delta \Ret_{n}(\gamma x)} \chi_{\{y \in \Lambda_+: R-\Ret(T^ny) < \Ret_{n}(y)-t \leq 0\}}(\gamma x) \, \dd\mu_{\mathrm{E}}(x)\\
={}&\int_{\Lambda_+} \sum_{n=0}^\infty \sum_{\gamma \in \calH^n}e^{-\delta \Ret_{n}(\gamma x)}f_x(\Ret_{n}(\gamma x)-t) \, \dd\mu_{\mathrm{E}} (x).
\end{align*}
The proof is complete by \cref{eqn:renewal sum}.
\end{proof}

\begin{proof}[Proof of \cref{prop:effect}]
Fix $\epsilon > 0$ to be the minimum of the $\epsilon$'s provided by \cref{prop:Coding,prop:exp_tail,thm:renewal}. Let $R > 1$ and $t > R$. By \cref{lem:residual}, we would like to apply \cref{thm:renewal} to  $\calR f_x$, where $f_x$ is as in the lemma. Since $f_x$ is not $C^2$, we take a smooth function which is greater than $f_x$.

Let $s_x=\max\{-t/2,R-\Ret(x)\}$.
If $s_x\leq 0$, let $g_x$ be a smooth bump function supported on $[s_x-1,1]$ and equal to $1$ on $[s_x,0]$. Otherwise, let $g_x=0$. Then, $\Leb(\supp g_x)\leq \min\{t/2+2, \Ret(x)-R + 2\}$ and $\|g_x''\|_1,\|g_x\|_1 \leq t$. By \cref{thm:renewal}, we have
\begin{align*}
\int_{\Lambda_+} \calR g_x(x,t) \, \dd\mu_{\mathrm{E}}(x)&=\int_{\Lambda_+} \left(\frac{h_0(x)}{\sigma_0}\int_{\R} g_x \, \dd \Leb+O\bigl(e^{-\epsilon t/2}t\bigr)\right) \, \dd\mu_{\mathrm{E}}(x)\\
&\ll \int_{\{y \in \Lambda_+: \Ret(y)>R\}} (\Ret(x)-R + 2) \, \dd\mu_{\mathrm{E}}(x)+ e^{-\epsilon t/4} \\
&\leq \frac{1}{\epsilon_0}\int_{\Lambda_+} e^{\epsilon_0(\Ret(x)-R + 2)} \, \dd\mu_{\mathrm{E}}(x)+ e^{-\epsilon t/4} \\
&\ll e^{-\epsilon R/4}
\end{align*}
due to the exponential tail property (see Property~(5) in \cref{prop:Coding}).

Now, we can bound $\int_{\Lambda_+} \calR f_x \, d\mu_{\mathrm{E}}$. If $\Ret(x)<R+t/2$, then we have $f_x\leq g_x$ due to the construction of $g_x$.
Therefore
\begin{align}\label{eq:R<R}
\int_{\{y \in \Lambda_+: \Ret(y)<R+t/2\}}\calR f_x(x,t) \, \dd\mu_{\mathrm{E}}(x)\leq \int_{\Lambda_+} \calR g_x(x,t) \, \dd\mu_{\mathrm{E}}(x)\ll e^{-\epsilon R/4}.
\end{align}
By the exponential tail property in the form of \cref{prop:exp_tail},
\[\mu_{\mathrm{E}}\{x \in \Lambda_+: \Ret(x)\geq R+t/2 \} \ll e^{-\epsilon(R+t/2)}. \]
For the renewal sum $\calR f_x$ defined in \cref{eqn:renewal sum}, if $n\geq t/\lambda_0$, then $\Ret_{n}(y)-t>0$, so $f_x(\Ret_{n}(y)-t)=0$. We only need to sum over integers $0 \leq n \leq t/\lambda_0$. Then due to the fact that $f_x\leq 1 \ll h_0$, we have
\begin{align*}
\calR f_x(x,t)&=\sum_{n = 0}^{\lfloor t/\lambda_0\rfloor}\sum_{\gamma \in \calH^n}e^{-\delta \Ret_{n}(\gamma x)}f_x(\Ret_{n}(\gamma x)-t)\ll \sum_{n = 0}^{\lfloor t/\lambda_0\rfloor}\sum_{\gamma \in \calH^n}e^{-\delta \Ret_{n}(\gamma x)}h_0(\gamma x) \\
&=\sum_{n = 0}^{\lfloor t/\lambda_0\rfloor}\L_0^n(h_0)(x)\leq \frac{t}{\lambda_0}h_0(x) \ll t.
\end{align*}
So we have
\begin{equation}\label{eq:R>R}
\int_{\{y \in \Lambda_+: \Ret(y)\geq R+t/2\}}\calR f_x(x,t) \, \dd\mu_{\mathrm{E}}(x)\ll e^{-\epsilon(R+t/2)}t \ll e^{-\epsilon R}.
\end{equation}
We obtain the desired inequality for $\mu_{\mathrm{E}}$ by combining \cref{eq:R<R,eq:R>R}. We finish the proof by first recalling that $\nu$ is absolutely continuous with respect to $\mu_{\mathrm{E}}$ with bounded density, and then by adjusting $\epsilon$ and taking $R > 0$ sufficiently large depending on $\epsilon$ to absorb the accumulated implicit constants.
\end{proof}

\begin{remark}
With a more precise estimate, we can obtain an asymptotic formula
\[ 	\mu_{\mathrm{E}}\{x: \Ret_{l+1}(x)-t>R \}= \int_{\{y \in \Lambda_+: \Ret(y)>R\}} h_0(x)(\Ret(x)-R) \, \dd\mu_{\mathrm{E}}(x)+O\bigl(e^{-\epsilon t}\bigr)\]
for all $R > 0$ sufficiently large and $t>R$. We only need an upper bound in \cref{prop:effect} and so we do not give the proof of this more precise formula.
\end{remark}

\begin{remark}
In the proof of \cref{prop:onecusp}, we see that we only need a weaker version of \cref{prop:effect}, that is, for all $t>e^{\epsilon R/2}$. As a consequence, a version of \cref{thm:renewal} with a polynomial error is sufficient for the argument. The current version of \cref{thm:renewal} may be of independent interest, so we present this stronger version here.
\end{remark}

\section{Dolgopyat's method}
\label{sec:Dolgopyat'sMethod}
As outlined in the introduction, in this section, we will perform a version of Dolgopyat's method which is a combination of the works of Stoyanov \cite{Sto11}, Sarkar--Winter \cite{SW21}, and Tsujii--Zhang \cite{TZ23}. Our goal is to prove \cref{thm:SpectralBound}.

We start with some notations. Let $(V, \|\cdot\|)$ be any normed vector space over $\R$ or $\C$. Let $d$ be any distance function on $\Delta_0$; in particular, $d = d_{\mathrm{E}}$ or $d = D$. Let $H: \Lambda_+\to V$ be any function. Following Dolgopyat \cite{Dol98}, define a family of equivalent norms
\begin{equation*}
\|H\|_{1, b}=\|H\|_{\infty}+\frac{1}{\max\{1,|b|\}} \Lip_d(H), \qquad b\in \mathbb{R}.
\end{equation*}
The Lipschitz norm is then simply $\|H\|_{\Lip(d)} = \|H\|_{1, 1}$. Recall the measure $\nu$ from \cref{subsec:expanding map}. If $H$ is measurable, denote
\begin{align*}
\|H\|_2 = \biggl(\int_{\Lambda_+} \|H(x)\|^2 \, d\nu(x)\biggr)^{\frac{1}{2}}.
\end{align*}
We also define the function $\|H\|: \Lambda_+ \to \R$ by
\begin{align*}
\|H\|(x) = \|H(x)\| \qquad \text{for all $x \in \Lambda_+$}.
\end{align*}

\begin{theorem}
\label{thm:SpectralBound}
There exist $\eta > 0$, $C > 0$, $a_0 > 0$, and $b_0 > 0$ such that for all $\xi=a+ib \in \C$ with $|a| < a_0$, if $(b, \rho) \in \widehat{M}_0(b_0)$, then for all $k \in \N$ and $H \in \Lip\bigl(\Lambda_+, V_\rho^{\oplus \dim(\rho)}\bigr)$, we have	
\begin{align*}
\bigl\|\mathcal{M}_{\xi, \rho}^k(H)\bigr\|_2 \leq Ce^{-\eta k} \|H\|_{1, \|\rho_b\|}.
\end{align*}
\end{theorem}

For all $(b, \rho) \in \widehat{M}_0(b_0)$, recall $\Omega(\log\|\rho_b\|,R_0) \subset \Delta_0$ from \cref{subsec:CombinatoricWords}.

\begin{theorem}
\label{thm:FrameFlowDolgopyat}
There exist $\eta \in (0, 1)$, $\kappa \in (0, 1)$, $a_0 > 0$, $b_0>0$, $m \in \N$, and a continuous function $\zeta: [-a_0, a_0] \to \R$ with $\zeta(0) = 1$ such that the following holds. For all $\xi=a+ib \in \C$ with $|a| < a_0$, if $(b, \rho) \in \widehat{M}_0(b_0)$, then for all $H \in \Lip\bigl(\Lambda_+, V_\rho^{\oplus \dim(\rho)}\bigr)$, there exist a sequence of positive functions $\{h_n\}_{n = 0}^\infty \subset \Lip_D(\Lambda_+, \R)$ with $h_0 = \|H\|_{1, \|\rho_b\|}$ and a sequence of closed subsets $\{\Omega_n\}_{n = 1}^\infty$ of $\Omega(\log\|\rho_b\|,R_0)$ such that:
\begin{enumerate}
\item for all $n \in \N$,  we have
\begin{align*}
\bigl\|\mathcal{M}_{\xi, \rho}^{nm}(H)(x)\bigr\|_2 \leq h_n(x) \qquad \text{for all $x \in \Lambda_+$};
\end{align*}
\item for all $n \in \N$, we have
\begin{align*}
h_n^2(x) \leq
\begin{cases}
\eta \mathcal{L}_0^m(h_{n - 1}^2)(x), & x \in \Omega_n, \\
\zeta(a)\mathcal{L}_0^m(h_{n - 1}^2)(x), & x \in \Lambda_+ - \Omega_n;
\end{cases}
\end{align*}
\item for all $n \in \N$, we have
\begin{align*}
\nu\{x \in \Lambda_+: \#\{j \in \N: 1\leq j \leq n, T^{jm}x \in \Omega_j\} < \kappa n\} < 2e^{-\kappa n}.
\end{align*}
\end{enumerate}
\end{theorem}

The following proof is inspired by \cite[Proposition 3.15]{TZ23} but we use the language of transfer operators instead.

\begin{proof}[Proof that \cref{thm:FrameFlowDolgopyat} implies \cref{thm:SpectralBound}]
Denote by $\tilde{\eta}$, $\kappa$, $\tilde{a}_0$, $b_0$, $m$, and $\zeta$ the constants and function provided by \cref{thm:FrameFlowDolgopyat}. Fix $\eta = \frac{1}{4}\min\{\kappa, -\kappa\log(\tilde{\eta})\}$ and $a_0 \in (0, \tilde{a}_0)$ such that $\sup\{|\log(\zeta(a))|: a \in [-a_0, a_0]\}\leq \eta$. Fix
\begin{align*}
C_0 &= \sup\bigl\{\bigl\|\mathcal{M}_{\xi, \rho}\bigr\|_{\mathrm{op}}^{2m}:|a| \leq a_0, \rho \in \widehat{M}\bigr\} \leq \sup\bigl\{\bigl\|{\mathcal{L}}_\xi\bigr\|_{\mathrm{op}}^{2m}:|a| \leq a_0\bigr\},\\
C &= \sqrt{3C_0},
\end{align*}
viewing the transfer operators as operators on $L^2\bigl(\Lambda_+, V_\rho^{\oplus \dim(\rho)}\bigr)$ and $L^2(\Lambda_+, \R)$ respectively. Let $\xi=a+ib \in \C$ with $|a| < a_0$. Suppose $(b, \rho) \in \widehat{M}_0(b_0)$. Let $k \in \N$ and write $k = nm + l$ for some integers $n \in \Z_{\geq 0}$ and $0 \leq l < m$. Let $H \in \Lip\bigl(\Lambda_+, V_\rho^{\oplus \dim(\rho)}\bigr)$. We then obtain corresponding sequences $\{h_n\}_{n = 0}^\infty$ and $\{\Omega_n\}_{n = 1}^\infty$ provided by \cref{thm:FrameFlowDolgopyat}.

We define the sequence of
functions $\{G_j: \Lambda_+ \to \R\}_{j = 0}^\infty$ recursively by
\begin{align*}
G_0 &= h_0^2, \\
G_j(x) &=
\begin{cases}
\tilde{\eta} G_{j - 1}(x), & x \in T^{-jm}(\Omega_j), \\
\zeta(a)G_{j - 1}(x), & x \in \Lambda_+ - T^{-jm}(\Omega_j),
\end{cases}
\qquad
\text{for all $j \in \N$}.
\end{align*}
We will first show by induction that they satisfy
\begin{align}
\label{eqn:TransferOperatorOnG_j}
\calL_0^{jm}(G_j) \geq h_j^2 \qquad \text{for all $j \in \Z_{\geq 0}$}.
\end{align}
The base case $j = 0$ is trivial. Now let $j \in \N$ and assume \cref{eqn:TransferOperatorOnG_j} holds for $j - 1$. It is immediate from definitions that for all $\gamma \in \calH^m$, we have
\begin{align}
\label{eqn:TransferOperatorOnG_j*}
\mathcal{L}_0^{(j - 1)m}(G_j)(\gamma x) =
\begin{cases}
\tilde{\eta} \mathcal{L}_0^{(j - 1)m}(G_{j - 1})(\gamma x), & \text{for all $x \in \Omega_j$} \\
\zeta(a)\mathcal{L}_0^{(j - 1)m}(G_{j - 1})(\gamma x), & \text{for all $x \in \Lambda_+ - \Omega_j$}.
\end{cases}
\end{align}
For all $x \in \Omega_j$, we use \cref{eqn:TransferOperatorOnG_j*}, the induction hypothesis, and Property~(2) in \cref{thm:FrameFlowDolgopyat}, to get
\begin{align*}
\mathcal{L}_0^{jm}(G_j)(x) &= \sum_{\gamma \in \calH^m} e^{\FRet_m^{(0)}(\gamma x)} \mathcal{L}_0^{(j - 1)m}(G_j)(\gamma x) \\
&= \tilde{\eta} \sum_{\gamma \in \calH^m}  e^{\FRet_m^{(0)}(\gamma x)} \mathcal{L}_0^{(j - 1)m}(G_{j - 1})(\gamma x) \\
&\geq \tilde{\eta} \sum_{\gamma \in \calH^m} e^{\FRet_m^{(0)}(\gamma x)} h_{j - 1}^2(\gamma x) = \tilde{\eta} \mathcal{L}_0^m(h_{j - 1}^2)(x)\geq h_j^2(x).
\end{align*}
For all $x \in \Lambda_+ - \Omega_j$, a similar calculation gives
\begin{align*}
\mathcal{L}_0^{jm}(G_j)(x) &= \sum_{\gamma \in \calH^m} e^{\FRet_m^{(0)}(\gamma x)} \mathcal{L}_0^{(j - 1)m}(G_j)(\gamma x) \\
&=\zeta(a) \sum_{\gamma \in \calH^m} e^{\FRet_m^{(0)}(\gamma x)} \mathcal{L}_0^{(j - 1)m}(G_{j - 1})(\gamma x) \\
&\geq \zeta(a)\sum_{\gamma \in \calH^m} e^{\FRet_m^{(0)}(\gamma x)} h_{j - 1}^2(\gamma x) =\zeta(a) \mathcal{L}_0^m(h_{j - 1}^2)(x)\geq h_j^2(x).
\end{align*}
This establishes \cref{eqn:TransferOperatorOnG_j}.

Now, using Property~(1) in \cref{thm:FrameFlowDolgopyat}, \cref{eqn:TransferOperatorOnG_j}, and $\mathcal{L}_0^*(\nu) = \nu$, we have
\begin{align*}
\bigl\|\mathcal{M}_{\xi, \rho}^k(H)\bigr\|_2^2 &\leq C_0\bigl\|\mathcal{M}_{\xi, \rho}^{nm}(H)\bigr\|_2^2 \leq C_0\|h_n\|_2^2 \\
&\leq C_0\int_{\Lambda_+}\mathcal{L}_0^{nm}(G_n) \, d\nu \\
&= C_0\int_{\Lambda_+} G_n \, d\nu.
\end{align*}
Using $\zeta(a) > \tilde{\eta}$, for all $x\in \{y \in \Lambda_+: \#\{j \in \N: j \leq n, T^{jm}y \in \Omega_j\} \geq \kappa n\}$, we have the bound
\begin{align*}
G_n(x)\leq \zeta(a)^{n - \lceil\kappa n\rceil}\tilde{\eta}^{\lceil\kappa n\rceil}G_0(x) \leq e^{\eta n}\tilde{\eta}^{\kappa n}h_0^2(x)
\end{align*}
while for all $x\in \{y \in \Lambda_+: \#\{j \in \N: j \leq n, T^{jm}y \in \Omega_j\} < \kappa n\}$, we have the trivial bound
\begin{align*}
G_n(x) \leq \zeta(a)^nG_0(x) \leq e^{\eta n} h_0^2(x).
\end{align*}
%Since $\Lambda_+ \subset \Delta_0$ is of full measure, 
Using Property~(3) in \cref{thm:FrameFlowDolgopyat}, we have
\begin{align*}
\bigl\|\mathcal{M}_{\xi, \rho}^k(H)\bigr\|_2^2 &\leq C_0\int_{\Lambda_+} G_n \, d\nu \\
&\leq C_0 e^{\eta n} ({\tilde{\eta}}^{\kappa n} \cdot (1 - 2e^{-\kappa n}) + 2e^{-\kappa n}) \|h_0\|_\infty^2 \\
&\leq C^2 e^{-2\eta n} \|H\|_{1, \|\rho_b\|}^2.
\end{align*}
\end{proof}

Our goal for the rest of this section is to prove \cref{thm:FrameFlowDolgopyat}. \cref{lem:LasotaYorke,lem:CylinderDiameterMeasureLowerBound,lem:PartnerPointInZariskiDenseLimitSetForBPBound} are preparatory lemmas.

The following is a Lasota--Yorke type lemma. Many similar lemmas can be found in the literature. But as we are dealing with a countably infinite coding and $V_\rho^{\oplus \dim(\rho)}$-valued functions, we give a proof analogous to that of \cite[Lemma 7.3]{SW21} and \cite[Appendix]{Sto11}.

\begin{lemma}
\label{lem:LasotaYorke}
Let $d$ be any distance function on $\Delta_0$ such that $d_{\mathrm{E}}(x,y)\leq d(x,y)$ for all $x,y\in \Lambda_+$ (e.g., $d = d_{\mathrm{E}}$ or $d = D$). There exists $A_0 > 1$ such that for all $\xi=a+ib \in \mathbb C$ with $|a| < a_0'$, if $(b, \rho) \in \widehat{M}_0(1)$, then for all $k \in \mathbb N$, we have:
\begin{enumerate}
\item if $h \in \Lip_{d}(\Lambda_+, \R)$ and $B>0$ satisfy
\begin{equation*}
|h(x) - h(x')| \leq Bh(x) d(x, x') 
\end{equation*}
for all $x,x'\in \Lambda_+$, then we have
\begin{align*}
|\mathcal{L}_a^k(h)(x) - \mathcal{L}_a^k(h)(x')| \leq A_0(B\lambda^k+1)\mathcal{L}_a^k(h)(x) d(x, x') 
\end{align*}
for all $x, x' \in \Lambda_+$.

\item if $H \in \Lip\bigl(\Lambda_+, V_\rho^{\oplus \dim(\rho)}\bigr)$, $h \in \Lip_d(\Lambda_+, \mathbb R)$ and $B>0$ satisfy 
\begin{equation*}
\|H(x) - H(x')\|_2 \leq Bh(x) d(x, x') 
\end{equation*}
for all $x, x' \in \Lambda_+$, then we have
\begin{align*}
\bigl\|\mathcal{M}_{\xi, \rho}^k(H)(x) - \mathcal{M}_{\xi, \rho}^k(H)(x')\bigr\|_2 \leq A_0\left(B\lambda^k\mathcal{L}_a^k(h)(x) + \|\rho_b\|\mathcal{L}_a^k\|H\|(x)\right) d(x, x')
\end{align*}
for all $x, x' \in \Lambda_+$.
\end{enumerate}
\end{lemma}

\begin{proof}
Recalling \cref{eqn:ConstantC1'C2,eqn:Constantdelta1varrho}, fix $A_0 > 4e^{C_2\diam(\Delta_0)}\frac{C_2}{\delta_{1, \varrho}}$. Let $\xi=a+ib \in \mathbb C$ with $|a| < a_0'$ and suppose $(b, \rho) \in \widehat{M}_0(1)$. Let $k \in \N$. Let $x, x' \in \Lambda_+$.

First we derive some bounds. For any $\gamma\in \calH^k$, by \cref{prop:Coding}, we have $d(T^l(\gamma x), T^l(\gamma x')) \leq \lambda^{k - l}d(x, x')$ for all integers $0 \leq l \leq k$. Using $\Lip^{\mathrm{e}}\bigl(\FRet^{(a)}\bigr) \leq C_1'$, we have 
\begin{align}
\label{eqn:FaLipEstimate}
\begin{aligned}
\bigl|\FRet_k^{(a)}(\gamma x) - \FRet_k^{(a)}(\gamma x')\bigr| &\leq \sum_{l = 0}^{k - 1} \bigl|\FRet^{(a)}(T^l(\gamma x)) - \FRet^{(a)}(T^l(\gamma x'))\bigr| \\
&\leq \sum_{l = 0}^{k - 1} C_1' \lambda^{k - l} d(x, x') \leq C_2d(x, x').
\end{aligned}
\end{align}
Along with the inequality that $|1 - e^z| \leq e^{\Re z}|z|$ for all $z \in \C$, we obtain
\begin{align}
\label{eqn:FaExponentialEstimate}
\left|1 - e^{\FRet_k^{(a)}(\gamma x') - \FRet_k^{(a)}(\gamma x)}\right|\leq e^{C_2d(x,x')}C_2d(x,x') \leq \frac{A_0}{2}\delta_{1, \varrho}d(x, x').
\end{align}

To prove Property~(1), suppose $h \in \Lip_{d}(\Lambda_+, \R)$ and $B > 0$ are as in the lemma. Using \cref{eqn:FaLipEstimate,eqn:FaExponentialEstimate}, we have
\begin{align*}
&\left|\mathcal{L}_a^k(h)(x) - \mathcal{L}_a^k(h)(x')\right| \\
\leq{}&\sum_{\gamma \in \calH^k} \left|e^{\FRet_k^{(a)}(\gamma x)}h(\gamma x) - e^{\FRet_k^{(a)}(\gamma x')} h(\gamma x')\right| \\
\leq{}&\sum_{\gamma \in \calH^k} \left(\left|1 - e^{\FRet_k^{(a)}(\gamma x') - \FRet_k^{(a)}(\gamma x)}\right|e^{\FRet_k^{(a)}(\gamma x)} h(\gamma x) + e^{\FRet_k^{(a)}(\gamma x')}|h(\gamma x) - h(\gamma x')|\right) \\
\leq{}&\sum_{\gamma \in \calH^k} \left(A_0d(x, x')e^{\FRet_k^{(a)}(\gamma x)}h(\gamma x) + A_0e^{\FRet_k^{(a)}(\gamma x)}Bh(\gamma x)d(\gamma x, \gamma x')\right) \\
\leq{}&(A_0 + A_0B \lambda^k) d(x, x')\sum_{\gamma \in \calH^k} e^{\FRet_k^{(a)}(\gamma x)}h(\gamma x) \\
={}&A_0(B\lambda^k + 1)\mathcal{L}_a^k(h)(x)d(x, x')
\end{align*}
where all the sums converge due to the exponential tail property (see Property~(5) in \cref{prop:Coding}).

To prove Property~(2), suppose $H \in \Lip\bigl(\Lambda_+, V_\rho^{\oplus \dim(\rho)}\bigr)$, $h \in \Lip_d(\Lambda_+, \mathbb R)$, and $B > 0$ are as in the lemma. A similar but more involved calculation as in \cref{eqn:FaLipEstimate} using $\Lip^{\mathrm{e}}(\GHol) \leq C_1'$ (see \cite[Lemma 5.2.3]{Sar22b}) gives that for any $\gamma\in \calH^k$, we have
\begin{align}
\label{eqn:GLipEstimate}
\begin{aligned}
\|\rho_b(\GHol^k(\gamma x)^{-1}) - \rho_b(\GHol^k(\gamma x')^{-1})\|_{\mathrm{op}} &\leq \|\rho_b\|d_{AM}(\GHol^k(\gamma x), \GHol^k(\gamma x')) \\
&\leq 2C_2\|\rho_b\|d(x, x').
\end{aligned}
\end{align}
Thus, using \cref{eqn:FaLipEstimate,eqn:FaExponentialEstimate,eqn:GLipEstimate}, we have
\begin{align}
\label{eqn:M difference}
&\bigl\|\mathcal{M}_{\xi, \rho}^k(H)(x) - \mathcal{M}_{\xi, \rho}^k(H)(x')\bigr\|_2 \nonumber\\
\leq{}&\sum_{\gamma \in \calH^k} \bigl\|e^{\FRet_k^{(a)}(\gamma x)} \rho_b(\GHol^k(\gamma x)^{-1}) H(\gamma x) - e^{\FRet_k^{(a)}(\gamma x')} \rho_b(\GHol^k(\gamma x')^{-1}) H(\gamma x')\bigr\|_2\nonumber \\
\leq{}&\sum_{\gamma \in \calH^k} \Bigl(\left|1 - e^{\FRet_k^{(a)}(\gamma x') - \FRet_k^{(a)}(\gamma x)}\right|e^{\FRet_k^{(a)}(\gamma x)} \|\rho_b(\GHol^k(\gamma x)^{-1}) H(\gamma x)\|_2\nonumber \\
&{}+ e^{\FRet_k^{(a)}(\gamma x')} \|(\rho_b(\GHol^k(\gamma x)^{-1}) - \rho_b(\GHol^k(\gamma x')^{-1})) H(\gamma x)\|_2 \nonumber\\
&{}+ e^{\FRet_k^{(a)}(\gamma x')} \|\rho_b(\GHol^k(\gamma x')^{-1}) (H(\gamma x) - H(\gamma x'))\|_2\Bigr)\nonumber \\
\leq{}&\sum_{\gamma \in \calH^k} \Bigl(\frac{A_0}{2}\delta_{1, \varrho}d(x, x')e^{\FRet_k^{(a)}(\gamma x)} \|H(\gamma x)\|_2\nonumber \\
&{}+ 2C_2 e^{\FRet_k^{(a)}(\gamma x')} \|\rho_b\| d(x, x') \|H(\gamma x)\|_2 \nonumber\\
&{}+ e^{\FRet_k^{(a)}(\gamma x')} \|H(\gamma x) - H(\gamma x')\|_2\Bigr)\nonumber \\
\leq{}&\sum_{\gamma \in \calH^k} \Bigl(\frac{A_0}{2}\|\rho_b\|d(x, x')e^{\FRet_k^{(a)}(\gamma x)} \|H(\gamma x)\|_2 + \frac{A_0}{2}\|\rho_b\|e^{\FRet_k^{(a)}(\gamma x)} d(x, x') \|H(\gamma x)\|_2\nonumber\\
&{}+ A_0B\lambda^k e^{\FRet_k^{(a)}(\gamma x)} h(\gamma x)d(x, x')\Bigr)\nonumber \\
\leq&A_0\bigl(B\lambda^k\mathcal{L}_a^k(h)(x) + \|\rho_b\|\mathcal{L}_a^k\|H\|(x)\bigr) d(x, x').
\end{align}
\end{proof}

We also need the following lemma.

\begin{lemma}
\label{lem:CylinderDiameterMeasureLowerBound}
For all $r > 0$, there exists $c > 0$ such that for all $x \in \Lambda_\Gamma \cap \Delta_0$, there exists a cylinder $\mathtt{C} \subset B(x,r)$ with $\diam(\mathtt{C}) > c$ and $\nu(\mathtt{C}) > c$.
\end{lemma}

\begin{proof}
Let $r > 0$. By compactness of $\Lambda_\Gamma \cap \overline{\Delta_0}$, it has a finite cover $\bigl\{B(x_j, r/2)\bigr\}_{j = 1}^n$ for some $\{x_j\}_{ j = 1}^n \subset \Lambda_\Gamma \cap \Delta_0$ and $n \in \N$. By \cref{lem:CylinderInOpenSet}, $B(x_j, r/2)$ contains a cylinder $\mathtt{C}_j$ for all $1 \leq j \leq n$. Fix $c = \min_{1 \leq j \leq n} \min\{\diam(\mathtt{C}_j), \nu(\mathtt{C}_j)\}$. Then the lemma follows because for any $x \in \Lambda_\Gamma \cap \Delta_0$, the ball $B(x, r)$ covers $B(x_j, r/2)$ for some $1 \leq j \leq n$.
\end{proof}

We now fix several fundamental constants for Dolgopyat's method. Recall the reference point $x_0 \in \Lambda_+ \subset \Delta_0$ from \cref{subsec:SymbolicModelForFrameFlows}. Fix $\delta_0 > 0$ such that
\begin{align}
\label{eqn:Constantdelta0}
B(x_0,2\delta_0) \subset \Delta_0.
\end{align}
Recall $\varepsilon_1$ from \cref{lem:maActionLowerBound} and $\varepsilon_2$ fixed after \cref{prop:LNIC}. Fix $R_0 > 0$ provided by \cref{prop:LDP}. Corresponding to $R_0$, fix $\varepsilon_3 \in (0, 1)$ to be the $\eta$ provided by \cref{prop:NCP}. Fix positive constants
\begin{align}
\label{eqn:Constant_b_0}
b_0 &= 1, \\
\label{eqn:ConstantE}
E &> \frac{2A_0}{\delta_{1, \varrho}}, \\
\delta_1 &< \frac{\varepsilon_1\varepsilon_2\varepsilon_3}{14} < \varepsilon_3, \\
\label{eqn:Constantepsilon1}
\epsilon_1 &< \min\left\{\delta_0 C_{\mathrm{cyl}}^{-1} C_{\Delta_0}^{-1} e^{-R_0}, \delta_1^{-1}\delta_0 \cdot 8A_0C_{\mathrm{cyl}} C_{\Delta_0} e^{R_0}, \frac{1}{\delta_1}, \frac{\delta_1}{C_{\BP}^2}, \frac{\delta_1 \delta_{1, \varrho}}{C_{\exp, \BP}}\right\}.
\end{align}
Fix
\begin{align}
\label{eqn:c0}
c_0 > 0
\end{align}
to be the constant provided by \cref{lem:CylinderDiameterMeasureLowerBound} corresponding to $r = \frac{\delta_1\epsilon_1}{8A_0 C_{\mathrm{cyl}}C_{\Delta_0}e^{R_0}}$. Fix
\begin{equation}\label{eqn:c}
\mathsf{c} = c_0 C_{\mathrm{cyl}}^{-1}C_{\Delta_0}^{-1}.
\end{equation}
Fix
\begin{align}
\label{eqn:Constantm}
m &\geq m_0 \text{ such that } \lambda^m < \max\left\{\mathsf{c}e^{-2R_0}, \frac{1}{8A_0}, \frac{1}{8E\epsilon_1}, \frac{\delta_1}{32E}\right\}.
\end{align}
Fix $\{\alpha_j\}_{j = 0}^{\jj} \subset \calH^m$ and the corresponding maps $\{\BP_j: \Delta_0 \times \Delta_0 \to AM\}_{j = 1}^{\jj}$ provided by \cref{prop:LNIC}. Also fix positive constants
\begin{align}
\label{eqn:ConstantT0}
T_0 &= \max\left\{\sup\left\{\bigl\|\FRet_m^{(a)}\bigr|_{\alpha_j\Delta_0}\bigr\|_\infty:|a| \leq a_0' \right\}:j \in \{0, 1, \dotsc, \jj\}\right\}, \\
\label{eqn:Constantmu}
\tau &< \min\left\{\frac{1}{4}, 2Ee^{-R_0} \mathsf{c}C_{\mathrm{cyl}}^{-1} \min_{j \in \{0, 1, \dotsc, \jj\}} \|d\alpha_j\|, \frac{\arccos\left(1 - \frac{(\delta_1 \epsilon_1)^2}{2}\right)^2}{16 \cdot 16e^{2T_0}}\right\}.
\end{align}

The following lemma provides the necessary cancellations for Dolgopyat's method. It is proved as in \cite[Lemma 8.1]{SW21} by combining \cref{prop:LNIC} (LNIC), \cref{prop:NCP} (NCP), and the lower bound in \cref{lem:maActionLowerBound}. However, notice that it is not true for all $x \in \Delta_0$ and an additional constraint coming from NCP is required. Consequently, cancellations only occur on balls centered at such $x$ or its partner point $y$.

\begin{lemma}
\label{lem:PartnerPointInZariskiDenseLimitSetForBPBound}
For all $(b, \rho) \in \widehat{M}_0(b_0)$, $x \in \Lambda_\Gamma \cap \Delta_0 - \overline{B\Bigl(\partial \Delta_0,\frac{\epsilon_1}{\|\rho_b\|}\Bigr)}$ with $u_x a_{\log\|\rho_b\|} \in \Omega_{R_0}$, and $\omega \in V_\rho^{\oplus \dim(\rho)}$ with $\|\omega\|_2 = 1$, there exist $1 \leq j \leq \jj$ and $y \in \Lambda_\Gamma \cap B\Bigl(x,\frac{\epsilon_1}{\|\rho_b\|}\Bigr) - B\Bigl(x,\frac{\delta_1\epsilon_1}{\|\rho_b\|}\Bigr)$ such that
\begin{align*}
\left\|d\rho_b\left(d(\BP_{j, x})_x(y - x)\right)(\omega)\right\|_2 > 7\delta_1\epsilon_1.
\end{align*}
\end{lemma}

\subsection{Good partitions and Dolgopyat operators}
\label{good partitions}
\begin{definition}[Partition, Atom, Finer, Coarser]
A \emph{$\nu$-measurable partition} $\calP$ of $\Delta_0$ is a set of mutually disjoint Borel subsets $\mathtt{C} \subset \Delta_0$ such that
\[\nu\biggl(\bigcup_{\mathtt{C}\in\calP}\mathtt{C}\biggr)=\nu(\Delta_0). \]
For any point $x \in \Delta_0$, its \emph{atom} is the unique element denoted by $\calP(x) \in \calP$ which contains $x$; if no such element of $\calP$ exists, define $\calP(x)=\varnothing$. For two $\nu$-measurable partitions $\calP$ and $\calQ$, we say $\calQ$ is \emph{finer} than $\calP$ or $\calP$ is \emph{coarser} than $\calQ$, denoted by $\calQ \succeq \calP$, if for all $\mathtt{D} \in \calQ$, there exists $\mathtt{C} \in \calP$ such that $\mathtt{D} \subset \mathtt{C}$.
\end{definition}

For all $(b, \rho) \in \widehat{M}_0(b_0)$, let $\calP_{(b, \rho)}$ be the $\nu$-measurable partition of $\Delta_0$ consisting of maximal cylinders $\mathtt{C} \subset \Delta_0$ with $\diam(\mathtt{C}) \leq \frac{e^{R_0}}{\|\rho_b\|}$. From definitions (see \cref{subsec:CombinatoricWords}), we have
\begin{align}
\label{CancellationSet}
\Omega^\dagger(\log\|\rho_b\|, R_0) = \left\{\mathtt{C} \in \calP_{(b, \rho)}: \diam(\mathtt{C}) \geq \frac{e^{-R_0}}{\|\rho_b\|}\right\}
\end{align}

\begin{lemma}
For all $(b, \rho) \in \widehat{M}_0(b_0)$ and $n \in \N$, we have $T^{-n}(\calP_{(b, \rho)}) \succeq \calP_{(b, \rho)}$.
\end{lemma}

\begin{proof}
Let $(b, \rho) \in \widehat{M}_0(b_0)$ and $n \in \N$. Let $\mathtt{C} \in T^{-n}(\calP_{(b, \rho)})$. Then $T^n(\mathtt{C}) \in \calP_{(b, \rho)}$ and so by the expanding property, we have
\begin{align*}
\diam(\mathtt{C}) \leq \diam(T^n(\mathtt{C})) \leq \frac{e^{R_0}}{\|\rho_b\|}.
\end{align*}
Thus, by definition of $\calP_{(b, \rho)}$, there exists a cylinder $\mathtt{D} \in \calP_{(b, \rho)}$ containing $\mathtt{C}$.
\end{proof}

\Cref{prop:CancellationCylinder} will introduce the sub-cylinders where cancellations can take place.

\begin{proposition}
\label{prop:CancellationCylinder}
Let $(b, \rho) \in \widehat{M}_0(b_0)$ and $\mathtt{J} \in \Omega^\dagger(\log\|\rho_b\|, R_0)$. There exists $x_1 \in \mathtt{J} \cap \Lambda_+$ such that $B\Bigl(x_1,\frac{\epsilon_1}{\|\rho_b\|}\Bigr)\subset \mathtt{J}$ and
for all $\omega \in V_\rho^{\oplus \dim(\rho)}$ with $\|\omega\|_2 = 1$, there exist $1 \leq j \leq \jj$ and $x_2 \in \Lambda_+ \cap B\Bigl(x_1,\frac{\epsilon_1}{\|\rho_b\|}\Bigr) - B\Bigl(x_1,\frac{\delta_1\epsilon_1}{\|\rho_b\|}\Bigr)$ such that
\begin{enumerate}
\item\label{itm:BPBound} $\|d\rho_b\left(d(\BP_{j, x_1})_{x_1}(x_2 - x_1)\right)(\omega)\|_2 \geq 7\delta_1\epsilon_1$;
\item 
for all $k \in \{1, 2\}$, there exists a $T^{-m}(\calP_{(b, \rho)})$-measurable subset $\mathtt{J}_k \subset \mathtt{J} \cap B\left(x_k,\frac{\delta_1\epsilon_1}{4A_0\|\rho_b\|}\right)$ such that
\begin{align*}
\diam(\mathtt{J}_k) &> \mathsf{c}\diam(\mathtt{J}), & \nu(\mathtt{J}_k) &> \mathsf{c}\nu(\mathtt{J}).
\end{align*}
\end{enumerate}
\end{proposition}

\begin{proof}
Let $(b, \rho)$ and $\mathtt{J}$ be as in the proposition. Write $\mathtt{J} = \gamma \Delta_0$ for some $\gamma \in \calH^n$ and $n \in \N$. By \cref{CancellationSet,lem:DerivativeEstimate}, we have
\begin{equation}
\label{inequality derivative}
C_{\Delta_0}^{-1} e^{-R_0}\|\rho_b\|^{-1}\leq\|d\gamma\|\leq C_{\Delta_0}e^{R_0} \|\rho_b\|^{-1}.
\end{equation}

Take $x_1 = \gamma x_0 \in \mathtt{J} \cap \Lambda_+$. Recall that $B(x_0,2\delta_0)\subset \Delta_0$ from \cref{eqn:Constantdelta0}. For $\delta_0' := \frac{\delta_0}{C_{\mathrm{cyl}}C_{\Delta_0} e^{R_0}}$, we have
\begin{align*}
B\biggl(x_1,\frac{\epsilon_1}{\|\rho_b\|}\biggr)\subset B\biggl(x_1,\frac{\delta_0'}{\|\rho_b\|}\biggr) \subset \gamma B(x_0,\delta_0) \subset \mathtt{J}
\end{align*}
by \cref{inequality derivative}, a similar argument as in \cref{lem:CylinderEstimate}, and inequality $\epsilon_1<\delta_0'$ (see \cref{eqn:Constantepsilon1}). As $x_1 \in \Omega(\log\|\rho_b\|, R_0)$, \cref{lem:goodpartitioncusp} gives $u_{x_1} a_{\log \|\rho_b\|}\in\Omega_{R_0}$. Let $\omega$ be as in the proposition. Using \cref{lem:PartnerPointInZariskiDenseLimitSetForBPBound} and the fact that $\Lambda_+\subset \Delta_0$ is of full measure, we obtain a point $x_2$ such that 
\begin{equation*}
x_2\in \Lambda_+ \cap B\biggl(x_1,\frac{\epsilon_1}{\|\rho_b\|}\biggr) - B\biggl(x_1,\frac{\delta_1\epsilon_1}{\|\rho_b\|}\biggr) \subset \mathtt{J} \cap \Lambda_+
\end{equation*}
and Property~(1) holds.

Let $k \in \{1, 2\}$. By \cref{lem:CylinderDiameterMeasureLowerBound,eqn:c0}, there exists a cylinder $\mathtt{J}_k' \subset B\Bigl(\gamma^{-1}x_k,\frac{\delta_1\epsilon_1}{8A_0 C_{\mathrm{cyl}}C_{\Delta_0}e^{R_0}}\Bigr)$ with $\diam(\mathtt{J}_k') > c_0$ and $\nu(\mathtt{J}_k') > c_0$. Take $\mathtt{J}_k = \gamma \mathtt{J}_k' \subset B\Bigl(x_k,\frac{\delta_1\epsilon_1}{4A_0\|\rho_b\|}\Bigr)$ where we derive the containment using \cref{inequality derivative,lem:CylinderEstimate}. Again using \cref{lem:CylinderEstimate,eqn:c}, we have
\begin{align}
\diam(\mathtt{J}_k) &\geq C_{\mathrm{cyl}}^{-1}\|d\gamma\|\diam(\mathtt{J}_k') > c_0 C_{\mathrm{cyl}}^{-1}C_{\Delta_0}^{-1}\diam(\mathtt{J}) = \mathsf{c}\diam(\mathtt{J}), \\
\nu(\mathtt{J}_k) &\geq C_{\mathrm{cyl}}^{-1}\|d\gamma\|^{\delta} \nu(\mathtt{J}_k') > c_0 C_{\mathrm{cyl}}^{-1}C_{\Delta_0}^{-1}\nu(\mathtt{J}) = \mathsf{c}\nu(\mathtt{J}).
\end{align}
Using \cref{eqn:Constantm}, we have
\begin{align*}
\diam(T^m(\mathtt{J}_k)) \geq \lambda^{-m}\diam(\mathtt{J}_k) > \mathsf{c}\lambda^{-m}\diam(\mathtt{J}) \geq \frac{\mathsf{c}\lambda^{-m}}{e^{R_0}\|\rho_b\|} > \frac{e^{R_0}}{\|\rho_b\|}
\end{align*}
which implies $T^m(\mathtt{J}_k)$ is $\calP_{(b, \rho)}$-measurable.
\end{proof}

We construct a function indicating at which sub-cylinders we gain decay. 
Let $(b, \rho) \in \widehat{M}_0(b_0)$, $\mathtt{J} \in \Omega^\dagger(\log\|\rho_b\|, R_0)$, and $H \in \Lip\bigl(\Lambda_+, V_\rho^{\oplus \dim(\rho)}\bigr)$. We introduce the following notations:
\begin{itemize}
\item 
denote by $x_1^{\mathtt{J}, H} \in \mathtt{J}$ the $x_1$ provided by \cref{prop:CancellationCylinder} and apply the same proposition to the unit vector
\begin{align}
\label{eqn:CancellationVector}
\omega = \frac{\rho_b\bigl(\GHol^m\bigl(\alpha_0x_1^{\mathtt{J}, H}\bigr)^{-1}\bigr) H\bigl(\alpha_0x_1^{\mathtt{J}, H}\bigr)}{\bigl\|H\bigl(\alpha_0x_1^{\mathtt{J}, H}\bigr)\bigr\|_2} \in V_\rho^{\oplus \dim(\rho)};
\end{align}
\item denote by 
\begin{align}
\label{eqn:CancellationPartnerPoint}
&x_2^{\mathtt{J}, H} \in \Lambda_+ \cap B\left(x_1^{\mathtt{J}, H},\frac{\epsilon_1}{\|\rho_b\|}\right) \subset \mathtt{J} \cap \Lambda_+,\\
\label{eqn:CancellationIndex}
&1 \leq j^{\mathtt{J}, H} \leq \jj, \\
\label{eqn:CancellationCylinder}
&\mathtt{J}_k^H \subset \mathtt{J} \cap B\left(x_k^{\mathtt{J}, H},\frac{\delta_1\epsilon_1}{4A_0\|\rho_b\|}\right),
\end{align}
the corresponding $x_2$, $j$, and $\mathtt{J}_k$ provided by \cref{prop:CancellationCylinder} for all $k \in \{1, 2\}$; 
\item define $\Xi(b, \rho) = \Omega^\dagger(\log\|\rho_b\|, R_0) \times \{1, 2\} \times \{1, 2\}$;
\item define $\Omega_J^H = \bigcup_{(\mathtt{J}, k, l) \in J} \mathtt{J}_k^H$ for all $J \subset \Xi(b, \rho)$;
\item define $\mathtt{J}_{k, 1}^H = \alpha_0 \mathtt{J}_k^H$ and $\mathtt{J}_{k, 2}^H = \alpha_{j^{\mathtt{J}, H}}\mathtt{J}_k^H$ for all $k \in \{1, 2\}$;
\item define the following function on $\Lambda_+$:
\begin{align}
\label{eqn:CancellationIndicatorFunction}
\beta_J^H = \chi_{\Lambda_+} - \tau\sum_{(\mathtt{J}, k, l) \in J} \chi_{\mathtt{J}_{k, l}^H},
\end{align}
\end{itemize}
where we view the characteristic functions $\chi_{\mathtt{J}_{k, l}^H}$ as functions on $\Lambda_+$.

Note that $x_1^{\mathtt{J}, H}$ and $\mathtt{J}_1^H$ do not depend on $H$. Although $\beta_{J}^H$ is defined using characteristic functions, it is indeed Lipschitz with respect to $D$, which is the advantage of the distance function $D$. More precisely, using \cref{lem:CylinderEstimate}, we obtain the following lemma.

\begin{lemma}
Let $(b, \rho) \in \widehat{M}_0(b_0)$, $J \subset \Xi(b, \rho)$, and $H \in \Lip\bigl(\Lambda_+, V_\rho^{\oplus \dim(\rho)}\bigr)$. We have $\beta_J^H\in \Lip_D(\Lambda_+,\mathbb{R})$ with
\begin{align}
\label{eqn:LipDbetaJH}
\Lip_D\bigl(\beta_J^H\bigr) \leq \frac{\tau \|\rho_b\|e^{R_0}}{\mathtt{c}C_{\mathrm{cyl}}^{-1} \min_{j \in \{0, 1, \dotsc, \jj\}} \|d\alpha_j\|}.
\end{align}
\end{lemma}

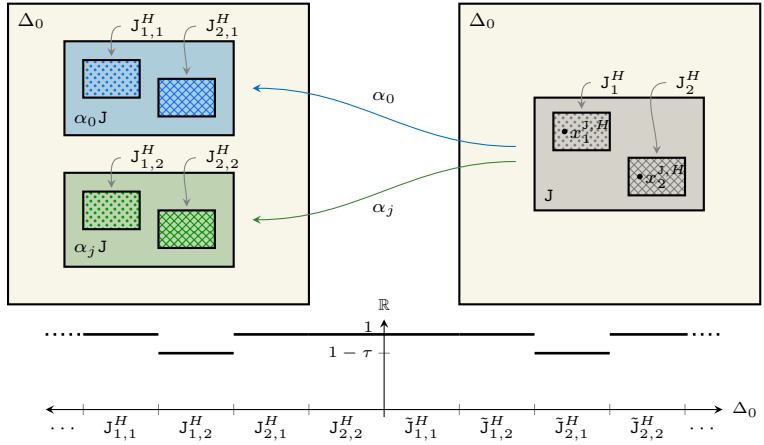
\begin{figure}[h]
\centering
\begin{tikzpicture}[>=stealth]
\definecolor{Sand}{RGB}{248,246,233}

\definecolor{Top}{RGB}{248,246,233}
\definecolor{Bottom}{RGB}{248,246,233}

%Left Side
%\Delta_0
\draw[black, thick, fill = Sand, fill opacity = 1] (-5, -2) rectangle (-1, 2);
%Branches of Cylinder J
\draw[black, thick, fill = NavyBlue, fill opacity = 0.3] (-2, 0.25) rectangle (-4.25, 1.5);
\draw[black, thick, fill = OliveGreen, fill opacity = 0.3] (-2, -1.5) rectangle (-4.25, -0.25);
%Subcylinders of Branches of  J
\draw[black, thick, pattern = crosshatch, pattern color = NavyBlue] (-2.25, 0.5) rectangle (-3, 1);
\draw[black, thick, pattern = crosshatch dots, pattern color = NavyBlue] (-3.25, 0.75) rectangle (-4, 1.25);
\draw[black, thick, pattern = crosshatch, pattern color = OliveGreen] (-2.25, -1.25) rectangle (-3, -0.75);
\draw[black, thick, pattern = crosshatch dots, pattern color = OliveGreen] (-3.25, -1) rectangle (-4, -0.5);

%Right Side
% \Delta_0
\draw[black, thick, fill = Sand, fill opacity = 1] (1, -2) rectangle (5, 2);
%Cylinder J
\draw[black, thick, fill = gray, fill opacity = 0.3] (2, -0.75) rectangle (4.25, 0.75);
%Subcylinders of J
\draw[black, thick, pattern = crosshatch dots, pattern color = gray] (2.25, 0.05) rectangle (3, 0.55);
\draw[black, thick, pattern = crosshatch, pattern color = gray] (3.25, -0.55) rectangle (4, -0.05);

%Points (were at (2.625, 0.3) and (3.625, 0.3))
\draw[fill = black] (2.4, 0.3) circle  [radius = 0.03];
\draw[fill = black] (3.4, -0.3) circle  [radius = 0.03];

%Curved Arrows
\draw[NavyBlue, ->] (1.75, 0.1) to[out=180,in=0] (-1.75, 0.875);
\draw[OliveGreen, ->] (1.75, -0.1) to[out=180,in=0] (-1.75, -0.875);

%Text
\node[below right] at (-5, 2) {\scriptsize $\Delta_0$};
\node[below right] at (1, 2) {\scriptsize $\Delta_0$};
\node[above right] at (2, -0.75) {\scriptsize $\mathtt{J}$};
\node[above right] at (-4.25, 0.25) {\scriptsize $\alpha_0\mathtt{J}$};
\node[above right] at (-4.25, -1.5) {\scriptsize $\alpha_j\mathtt{J}$};
%For subcylinders
\node[right] at (2.75,0.95) {\scriptsize $\mathtt{J}_1^H$};
\node[right] at (3.75, 0.95) {\scriptsize $\mathtt{J}_2^H$};
%Curves associated to the text
\draw[gray,->] (2.75, 0.95) to[out = 180, in = 90] (2.625, 0.6);
\draw[gray,->] (3.75, 0.95) to[out = 180, in = 90] (3.625, 0);
%For first brach of subcylinder
\node[right] at (-3.5, 1.7) {\scriptsize $\mathtt{J}_{1, 1}^H$};
\node[right] at (-2.5,1.7) {\scriptsize $\mathtt{J}_{2, 1}^H$};
%Curves associated to the text
\draw[gray,->] (-3.5, 1.7) to[out = 180, in = 90] (-3.625, 1.3);
\draw[gray,->] (-2.5, 1.7) to[out = 180, in = 90] (-2.625, 1.05);
%For second branch of subcylinder
\node[right] at (-3.5, -0.05) {\scriptsize $\mathtt{J}_{1, 2}^H$};
\node[right] at (-2.5,-0.05) {\scriptsize $\mathtt{J}_{2, 2}^H$};
%Curves associated to the text
\draw[gray,->] (-3.5, -0.05) to[out = 180, in = 90] (-3.625, -0.45);
\draw[gray,->] (-2.5, -0.05) to[out = 180, in = 90] (-2.625, -0.7);
%For arrows
\node[] at (0, 0.75) {\scriptsize $\alpha_0$};
\node[] at (0, -0.75) {\scriptsize $\alpha_j$};
%For points
\node[] at (2.75, 0.3) {\scriptsize $x_1^{\mathtt{J}, H}$};
\node[] at (3.75, -0.3) {\scriptsize $x_2^{\mathtt{J}, H}$};

\begin{axis}[axis x line=center,
axis y line=center,
x axis line style={<->},
xtick distance=1,
%	xmajorticks=false,
xlabel={$\Delta_0$},
ylabel={$\mathbb R$},
xlabel style={right, font = \tiny},
xticklabels={,,},
ylabel style={above, font = \tiny},
ytick={0.75,1},
yticklabels={$1 - \tau$,\smash{\raisebox{1pt}{1}}},
x tick label style={font = \tiny},
y tick label style={font = \tiny, xshift=1.5}, %yshift=2.5,xshift=2.5
xmin=-4.5,
xmax=4.5,
ymin=-0.1,
ymax=1.2,
height=1.3cm,
width=9cm,
scale only axis,
at={(-4.5cm,-3.5cm)}]

\addplot [color=black,dotted,line width=1pt,samples=500,domain=-4.5:-4]{1};
\addplot [color=black,line width=1pt,samples=500,domain=-4:-3]{1};
\addplot [color=black,line width=1pt,samples=500,domain=-3:-2]{0.75};
\addplot [color=black,line width=1pt,samples=500,domain=-2:-1]{1};
\addplot [color=black,line width=1pt,samples=500,domain=-1:0]{1};
\addplot [color=black,line width=1pt,samples=500,domain=0:1]{1};
\addplot [color=black,line width=1pt,samples=500,domain=1:2]{1};
\addplot [color=black,line width=1pt,samples=500,domain=2:3]{0.75};
\addplot [color=black,line width=1pt,samples=500,domain=3:4]{1};
\addplot [color=black,dotted,line width=1pt,samples=500,domain=4:4.5]{1};
\end{axis}

\node[] at (-4.25, -3.65) {\scriptsize $\cdots$};
\node[] at (-3.5, -3.65) {\scriptsize $\mathtt{J}_{1, 1}^H$};
\node[] at (-2.5, -3.65) {\scriptsize $\mathtt{J}_{1, 2}^H$};
\node[] at (-1.5, -3.65) {\scriptsize $\mathtt{J}_{2, 1}^H$};
\node[] at (-0.5, -3.65) {\scriptsize $\mathtt{J}_{2, 2}^H$};
\node[] at (0.5, -3.65) {\scriptsize $\tilde{\mathtt{J}}_{1, 1}^H$};
\node[] at (1.5, -3.65) {\scriptsize $\tilde{\mathtt{J}}_{1, 2}^H$};
\node[] at (2.5, -3.65) {\scriptsize $\tilde{\mathtt{J}}_{2, 1}^H$};
\node[] at (3.5, -3.65) {\scriptsize $\tilde{\mathtt{J}}_{2, 2}^H$};
\node[] at (4.25, -3.65) {\scriptsize $\cdots$};

\end{tikzpicture}
\caption{An illustration of the construction of $\beta_{J}^H$. The shapes and sizes of the cylinders are for simplicity.}
\end{figure}

\begin{definition}[Dolgopyat operator]
For all $\xi=a+ib \in \mathbb C$ with $|a| < a_0'$, if $(b, \rho) \in \widehat{M}_0(b_0)$, then for all $J \subset \Xi(b, \rho)$ and $H \in \Lip\bigl(\Lambda_+, V_\rho^{\oplus \dim(\rho)}\bigr)$, we define the \emph{Dolgopyat operator} $\mathcal{N}_{a, J}^H: \Lip_D(\Lambda_+, \R) \to \Lip_D(\Lambda_+, \R)$ by
\begin{align*}
\mathcal{N}_{a, J}^H(h) = \mathcal{L}_a^m\bigl(\beta_J^H h\bigr) \qquad \text{for all $h \in \Lip_D(\Lambda_+, \R)$}.
\end{align*}
\end{definition}

Actually, the fact that $\mathcal{N}_{a,J}^H(h)$ is Lipschitz with respect to $D$ for any $h \in \Lip_D(\Lambda_+, \R)$ can be proved similar to \cref{lem:LasotaYorke}.

\begin{lemma}
\label{lem:dolgopyatoperator}
There exist $a_0 > 0$, $\eta \in (0, 1)$, and $C > 0$ such that for all $\xi=a+ib \in \mathbb C$ with $|a| < a_0$, if $(b, \rho) \in \widehat{M}_0(b_0)$, then for all $J \subset \Xi(b, \rho)$ and $(H, h) \in \Lip\bigl(\Lambda_+, V_\rho^{\oplus \dim(\rho)}\bigr) \times \Lip_D(\Lambda_+, \mathbb R)$, we have
\begin{align*}
\mathcal{N}_{a, J}^H(h)^2(x) \leq
\begin{cases}
\eta \mathcal{L}_0^m(h^2)(x), & x \in \Omega_J^H \\
\zeta(a)\mathcal{L}_0^m(h^2)(x), & x \in \Lambda_+ - \Omega_J^H,
\end{cases}
\end{align*}
where $\zeta: [-a_0, a_0] \to \R$ is a continuous function with $\zeta(0)=1$.
\end{lemma}

\begin{proof}
Recall the constant $T_0$ from \cref{eqn:ConstantT0}. Fix $\eta = \sqrt{1 - \tau e^{-T_0}}$. We introduce the function $\zeta:[-a_0', a_0'] \to \R$ defined by
\begin{align*}
\zeta(a) = (\lambda_a^{-2}\lambda_{2a})^m \left\|\frac{h_0}{h_a}\right\|_\infty \left\|\frac{h_a}{h_0}\right\|_\infty \left\|\frac{h_a}{h_{2a}}\right\|_\infty \left\|\frac{h_{2a}}{h_a}\right\|_\infty
\end{align*}
whose purpose will become clear shortly. Since $\zeta$ is continuous and $\zeta(0) = 1$, we can fix a sufficiently small $a_0 \in (0, a_0'/2)$ such that $\eta\zeta(a) \leq 1$ for all $|a| < a_0$. Let $\xi$, $(b, \rho)$, $J$, and $(H, h)$ be as in the lemma. By the Cauchy--Schwarz inequality, we have
\begin{align*}
\mathcal{N}_{a, J}^H(h)^2 &= \left(\sum_{\gamma \in \calH^m} e^{\FRet_m^{(a)}(\gamma x)} \beta_J^H(\gamma x) h(\gamma x)\right)^2\\
&\leq \left(\sum_{\gamma \in \calH^m} e^{(\FRet_m^{(a)} -a\Ret_m)(\gamma x)} \beta_J^H(\gamma x)^2\right) \left(\sum_{\gamma \in \calH^m} e^{(\FRet_m^{(a)} +a\Ret_m)(\gamma x)} h(\gamma x)^2\right) \\
&\leq \zeta(a)\mathcal{L}_{2a}^m\big(\big(\beta_J^H\big)^2\big) \mathcal{L}_0^m(h^2).
\end{align*}
Note that $\mathcal{L}_{2a}^m(\chi_{\Lambda_+}) = \chi_{\Lambda_+}$ by virtue of the normalization of the transfer operators. Hence, $\mathcal{L}_{2a}^m\big(\big(\beta_J^H\big)^2\big) \leq 1$ which gives the trivial bound $\mathcal{N}_{a, J}^H(h)^2 \leq \zeta(a)\mathcal{L}_0^m(h^2)$. Now, let $x \in \Omega_J^H$. Then $x \in \mathtt{J}_{{k}}^H$ for some $({\mathtt{J}}, {k}, {l}) \in J$. Using $\big(\beta_J^H\big)^2 \leq \chi_{\Lambda_+} - \tau \chi_{{\mathtt{J}}_{{k}, {l}}^H}$ and \cref{eqn:ConstantT0}, we have
\begin{align*}
\mathcal{L}_{2a}^m\big(\big(\beta_J^H\big)^2\big)(x) \leq \mathcal{L}_{2a}^m(\chi_{\Lambda_+})(x) - \tau\mathcal{L}_{2a}^m\Bigl(\chi_{{\mathtt{J}}_{{k}, {l}}^H}\Bigr)(x) \leq 1 - \tau e^{-T_0} = \eta^2.
\end{align*}
Thus, using this in the previous inequality gives
\begin{align*}
\mathcal{N}_{a, J}^H(h)^2(x) \leq \eta^2\zeta(a)\mathcal{L}_0^m(h^2)(x) \leq \eta \mathcal{L}_0^m(h^2)(x).
\end{align*}
\end{proof}

\subsection{Invariance of cone}
First we introduce the following two related notions of density following Tsujii--Zhang \cite{TZ23} and Dolgopyat \cite{Dol98}.

\begin{definition}[Dense]
Let $\calP$ be a $\nu$-measurable partition of $\Delta_0$. A subset $B \subset \Delta_0$ is said to be \emph{$\calP$-measurable} if $\nu\bigl(B \triangle \bigcup_{x\in B} \calP(x)\bigr) = 0$. Furthermore, we say that a $\nu$-measurable subset $A \subset B$ is \emph{$(\calP, c)$-dense} in $B$ for some $c > 0$ if
\begin{align*}
\nu(A \cap \calP(x))>c\nu(\calP(x)) \qquad \text{for $\nu$-almost all $x\in B$}.
\end{align*}
\end{definition}

\begin{definition}[Dense index set]
Let $(b, \rho) \in \widehat{M}_0(b_0)$. We say that $J \subset \Xi(b, \rho)$ is a \emph{dense index set} , if for all $\mathtt{J} \in \Omega^\dagger(\log\|\rho_b\|, R_0)$, there exists $(k, l) \in \{1, 2\} \times \{1, 2\}$ such that $(\mathtt{J}, k, l) \in J$.
\end{definition}

For all $(b, \rho) \in \widehat{M}_0(b_0)$, if $J \subset \Xi(b, \rho)$ is a dense index set, then by \cref{prop:CancellationCylinder}, we have that $\Omega_J^H$ is $(\calP_{(b, \rho)}, \mathsf{c})$-dense in $\Omega(\log\|\rho_b\|, R_0)$ for all $H \in \Lip\bigl(\Lambda_+, V_\rho^{\oplus \dim(\rho)}\bigr)$.

Define $\mathcal{J}(b, \rho) = \{J \subset \Xi(b, \rho): J \text{ is a dense index set}\}$. In this subsection we introduce a cone and prove that it is preserved by a pair consisting of a transfer operator with holonomy and a Dolgopyat operator for some $J\in \mathcal{J}(b, \rho)$. This is the heart of the Dolgopyat's method where we obtain cancellations of the summands of the transfer operator with holonomy via \cref{prop:CancellationCylinder} whose source is \cref{prop:LNIC} (LNIC), \cref{prop:NCP} (NCP), and the lower bound in \cref{lem:maActionLowerBound}.

For all $(b, \rho) \in \widehat{M}_0(b_0)$, define the cone
\newcounter{conerow}
\begin{align}
\label{eqn:Cone}
\mathcal{C}_{(b, \rho)}(\Lambda_+) =
\left\{
\begin{array}{ll}
\multicolumn{2}{l}{(H, h) \in \Lip\bigl(\Lambda_+, V_\rho^{\oplus \dim(\rho)}\bigr) \times \Lip_D(\Lambda_+, \R):}\\
\stepcounter{conerow}(\arabic{conerow}) & h > 0, \\
\stepcounter{conerow}(\arabic{conerow}) & |h(x) - h(x')| \leq E\|\rho_b\|h(x) D(x, x') \,\,\,\text{for all}\,\,\, x, x' \in \Lambda_+, \\
\stepcounter{conerow}(\arabic{conerow}) & \|H(x)\|_2 \leq h(x) \text{ for all },x\in \Lambda_+, \\
\stepcounter{conerow}(\arabic{conerow}) & \|H(x) - H(x')\|_2 \leq E\|\rho_b\|h(x) D(x, x') \,\,\, \text{for all}\,\,\, x, x' \in \Lambda_+
\end{array}
\right\}
\end{align}
using the constant $E$ defined in \cref{eqn:ConstantE}.

\begin{proposition}
\label{prop:PreservingCone}
Let $\xi=a+ib \in \mathbb C$ with $|a| < a_0'$ and suppose $(b, \rho) \in \widehat{M}_0(b_0)$. For all $(H, h) \in \mathcal{C}_{(b, \rho)}(\Lambda_+)$, there exists $J \in \mathcal J(b, \rho)$ such that $\bigl(\mathcal{M}_{\xi, \rho}^m(H), \mathcal{N}_{a, J}^H(h)\bigr) \in \mathcal{C}_{(b, \rho)}(\Lambda_+)$.
\end{proposition}

The goal of the rest of the subsection is to prove \cref{prop:PreservingCone}. Preservation of Property~(1) in \cref{eqn:Cone} is trivial. Preservation of Properties~(2) and (4)  in \cref{eqn:Cone} follow from the following two lemmas. The proofs are similar to \cite[Lemmas 9.1 and 9.2]{SW21}. They are analogous to the original \cite[Proposition 6 and Lemma 11]{Dol98}. Preservation of Property~(3) will be given later.

\begin{lemma}
\label{lem:PreserveProperty2}
For all $\xi=a+ib \in \mathbb C$ with $|a| < a_0'$, if $(b, \rho) \in \widehat{M}_0(b_0)$, and if $h \in \Lip_D(\Lambda_+, \R)$ satisfies Property~(2) in \cref{eqn:Cone}, then for all $J \subset \Xi(b, \rho)$ and $H \in \Lip\bigl(\Lambda_+, V_\rho^{\oplus \dim(\rho)}\bigr)$, $\mathcal{N}_{a, J}^H(h)$ also satisfies Property~(2) in \cref{eqn:Cone}.
\end{lemma}

\begin{proof}
Let $\xi$, $(b, \rho)$, $h$, $J$, and $H$ be as in the lemma. Let $x, x' \in \Lambda_+$. Using \cref{eqn:Constantmu,eqn:LipDbetaJH}, and $\beta_J^H\geq 1-\tau$, we have
\begin{align*}
&\bigl|\bigl(\beta_J^H h\bigr)(x) - \bigl(\beta_J^H h\bigr)(x')\bigr| \leq |h(x) - h(x')| + h(x)\bigl|\beta_J^H(x) - \beta_J^H(x')\bigr| \\
\leq{}&E\|\rho_b\|h(x)D(x, x') + h(x) \cdot \frac{\tau \|\rho_b\|e^{R_0}}{\mathsf{c}C_{\mathrm{cyl}}^{-1} \min_{j \in \{0, 1, \dotsc, \jj\}} \|d\alpha_j\|} \cdot D(x, x') \\
\leq{}&3E\|\rho_b\| \cdot \frac{\beta_J^H(x)}{1 - \tau} \cdot h(x)D(x, x') \\
\leq{}&4E\|\rho_b\|\bigl(\beta_J^H h\bigr)(x)D(x, x').
\end{align*}
Now applying \cref{lem:LasotaYorke,eqn:ConstantE,eqn:Constantm}, we have
\begin{align*}
\bigl|\mathcal{N}_{a, J}^H(h)(x) - \mathcal{N}_{a, J}^H(h)(x')\bigr| &= \bigl|\mathcal{L}_a^m\bigl(\beta_J^H h\bigr)(x) - \mathcal{L}_a^m\bigl(\beta_J^H h\bigr)(x')\bigr| \\
&\leq A_0 (4E\|\rho_b\|\lambda^m + 1) \mathcal{L}_a^m\bigl(\beta_J^H h\bigr)(x) D(x, x') \\
&= E\|\rho_b\|\mathcal{N}_{a, J}^H(h)(x)D(x, x').
\end{align*}
\end{proof}

\begin{lemma}
\label{lem:PreserveProperty4}
For all $\xi=a+ib \in \mathbb C$ with $|a| < a_0'$, if $(b, \rho) \in \widehat{M}_0(b_0)$, and if $(H, h) \in \Lip\bigl(\Lambda_+, V_\rho^{\oplus \dim(\rho)}\bigr) \times \Lip_D(\Lambda_+, \mathbb R)$ satisfies Properties~(3) and (4) in \cref{eqn:Cone}, then for all $J \subset \Xi(b, \rho)$, $\bigl(\mathcal{M}_{\xi, \rho}^m(H), \mathcal{N}_{a, J}^H(h)\bigr)$ also satisfies Property~(4) in \cref{eqn:Cone}.
\end{lemma}

\begin{proof}
Let $\xi$, $(b, \rho)$, $(H, h)$, and $J$ be as in the lemma. Let $x, x' \in \Lambda_+$. Applying \cref{lem:LasotaYorke,eqn:ConstantE,eqn:Constantm,eqn:Constantmu}, and $\beta_J^H\geq 1-\tau$, we have
\begin{align*}
&\bigl\|\mathcal{M}_{\xi, \rho}^m(H)(x) - \mathcal{M}_{\xi, \rho}^m(H)(x')\bigr\|_2 \\
\leq{}&A_0\left(E\|\rho_b\|\lambda^m \mathcal{L}_a^m(h)(x) + \|\rho_b\|\mathcal{L}_a^m(\|H\|_2)(x)\right)D(x, x') \\
\leq{}&A_0\left(\frac{E\|\rho_b\|}{8A_0} + \frac{E\|\rho_b\|}{2A_0}\right)\mathcal{L}_a^m(h)(x)D(x, x') \\
\leq{}&\left(\frac{E\|\rho_b\|}{8(1 - \tau)} + \frac{E\|\rho_b\|}{2(1 - \tau)}\right)\mathcal{L}_a^m\bigl(\beta_J^H h\bigr)(x)D(x, x') \\
\leq{}&E\|\rho_b\|\mathcal{N}_{a, J}^H(h)(x)D(x, x').
\end{align*}
\end{proof}

The preservation of Property~(3) in \cref{eqn:Cone} is crucial and this is where the aforementioned cancellations occur. It follows from the next four lemmas. For all $\xi=a+ib \in \mathbb C$ with $|a| < a_0'$, if $(b, \rho) \in \widehat{M}_0(b_0)$, then for all $(H, h) \in \Lip\bigl(\Lambda_+, V_\rho^{\oplus \dim(\rho)}\bigr) \times \Lip_D(\Lambda_+, \R)$, and $1 \leq j \leq\jj$, we define the functions $\chi_{j, 1}^{[\xi, \rho, H, h]}, \chi_{j, 2}^{[\xi, \rho, H, h]}: \Lambda_+ \to \R$ by
\begin{comment}
Unfortunately, this gives an overfull box. If the it fits in one piece for a future journal version, this version can be used instead.
\begin{align*}
\chi_{j, 1}^{[\xi, \rho, H, h]}(x) &= \frac{\bigl\|e^{\FRet_m^{(a)}(\alpha_0x)} \rho_b(\GHol^m(\alpha_0x)^{-1})H(\alpha_0x) + e^{\FRet_m^{(a)}(\alpha_jx)} \rho_b(\GHol^m(\alpha_jx)^{-1})H(\alpha_jx)\bigr\|_2}{(1 - \tau)e^{\FRet_m^{(a)}(\alpha_0x)}h(\alpha_0x) + e^{\FRet_m^{(a)}(\alpha_jx)}h(\alpha_jx)}, \\
\chi_{j, 2}^{[\xi, \rho, H, h]}(x) &= \frac{\bigl\|e^{\FRet_m^{(a)}(\alpha_0x)} \rho_b(\GHol^m(\alpha_0x)^{-1})H(\alpha_0x) + e^{\FRet_m^{(a)}(\alpha_jx)} \rho_b(\GHol^m(\alpha_jx)^{-1})H(\alpha_jx)\bigr\|_2}{e^{\FRet_m^{(a)}(\alpha_0x)}h(\alpha_0x) + (1 - \tau)e^{\FRet_m^{(a)}(\alpha_jx)}h(\alpha_jx)}
\end{align*}
for all $x \in \Delta_0$.
\end{comment}
\begin{align*}
&\chi_{j, 1}^{[\xi, \rho, H, h]}(x) \\
={}& \frac{\bigl\|e^{\FRet_m^{(a)}(\alpha_0x)} \rho_b(\GHol^m(\alpha_0x)^{-1})H(\alpha_0x) + e^{\FRet_m^{(a)}(\alpha_jx)} \rho_b(\GHol^m(\alpha_jx)^{-1})H(\alpha_jx)\bigr\|_2}{(1 - \tau)e^{\FRet_m^{(a)}(\alpha_0x)}h(\alpha_0x) + e^{\FRet_m^{(a)}(\alpha_jx)}h(\alpha_jx)}, \\
&\chi_{j, 2}^{[\xi, \rho, H, h]}(x) \\
={}& \frac{\bigl\|e^{\FRet_m^{(a)}(\alpha_0x)} \rho_b(\GHol^m(\alpha_0x)^{-1})H(\alpha_0x) + e^{\FRet_m^{(a)}(\alpha_jx)} \rho_b(\GHol^m(\alpha_jx)^{-1})H(\alpha_jx)\bigr\|_2}{e^{\FRet_m^{(a)}(\alpha_0x)}h(\alpha_0x) + (1 - \tau)e^{\FRet_m^{(a)}(\alpha_jx)}h(\alpha_jx)}
\end{align*}
for all $x \in \Lambda_+$. These function are for encoding that when they are at most $1$, cancellations occur among the vectors in the numerator. The second subscript is to indicate for which of the two inverse branches $\alpha_0$ or $\alpha_j$ the cancellations occur.

The following lemma can be proved similar to \cite[Lemma 9.8]{SW21}. It is analogous to the original \cite[Lemma 14]{Dol98}.

\begin{lemma}
\label{lem:HTrappedByh}
Let $(b, \rho) \in \widehat{M}_0(b_0)$. Suppose $(H, h) \in \mathcal{C}_{(b, \rho)}(\Lambda_+)$. Then, for all $(\mathtt{J}, k, l) \in \Xi(b, \rho)$, letting $j=0$ if $l=1$ and $j=j^{\mathtt{J},H}$ if $l=2$, we have
\begin{align*}
\frac{1}{2} \leq \frac{h(\alpha_jx)}{h(\alpha_jx')} \leq 2 \qquad \text{for all $x, x' \in \mathtt{J}_k^H$}
\end{align*}
and also either of the alternatives
\begin{enumerate}
\item $\|H(\alpha_jx)\|_2 \leq \frac{3}{4}h(\alpha_jx)$ for all $x \in \mathtt{J}_k^H$;
\item $\|H(\alpha_jx)\|_2 \geq \frac{1}{4}h(\alpha_jx)$ for all $x \in \mathtt{J}_k^H$.
\end{enumerate}
\end{lemma}

For any $k \geq 2$, denote by $\Theta(w_1, w_2) = \arccos\left(\frac{\langle w_1, w_2\rangle}{\|w_1\| \cdot \|w_2\|}\right) \in [0, \pi]$ the angle between $w_1, w_2 \in \mathbb R^k - \{0\}$, where we use the standard inner product and norm. The following is a basic lemma in Euclidean geometry (see \cite[Lemma 5.4.9]{Sar22b} for a proof).

\begin{lemma}
\label{lem:StrongTriangleInequality}
Let $k \geq 2$. If $w_1, w_2 \in \mathbb R^k - \{0\}$ such that $\Theta(w_1, w_2) \geq \alpha$ and $\frac{\|w_1\|}{\|w_2\|} \leq L$ for some $\alpha \in [0, \pi]$ and $L \geq 1$, then we have
\begin{align*}
\|w_1 + w_2\| \leq \left(1 - \frac{\alpha^2}{16L}\right)\|w_1\| + \|w_2\|.
\end{align*}
\end{lemma}

Recall from \cref{eqn:CancellationPartnerPoint,eqn:CancellationIndex,eqn:CancellationCylinder} that the index $1 \leq j^{\mathtt{J},H} \leq \jj$ indicates that for $H \in \Lip\bigl(\Delta_0, V_\rho^{\oplus \dim(\rho)}\bigr)$, the inverse branches $\alpha_0$ and $\alpha_{j^{\mathtt{J},H}}$ should be compared in order to obtain cancellations inside $\mathtt{J}$ which occur for either $\mathtt{J}_1^H$ or $\mathtt{J}_2^H$.

\begin{lemma}
\label{lem:chiLessThan1}
Let $\xi=a+ib \in \mathbb C$ with $|a| < a_0'$ and $(b, \rho) \in \widehat{M}_0(b_0)$. Suppose $(H, h) \in \mathcal{C}_{(b, \rho)}(\Lambda_+)$. For all $\mathtt{J} \in \Omega^\dagger(\log\|\rho_b\|, R_0)$, denoting $j=j^{\mathtt{J},H}$, there exists $(k, l) \in \{1, 2\} \times \{1, 2\}$ such that $\chi_{j, l}^{[\xi, \rho, H, h]}(x) \leq 1$ for all $x \in \mathtt{J}_k^H\cap \Lambda_+$.
\end{lemma}

\begin{proof}
Let $\xi$, $(b, \rho)$, $(H, h)$, and $\mathtt{J}$ be as in the lemma. To simplify notation, we write $x_k$ for $x_k^{\mathtt{J}, H}$ for all $k \in \{1, 2\}$ and $j$ for $j^{\mathtt{J},H}$. 

Now, if Alternative~(1) in \cref{lem:HTrappedByh} holds for $(\mathtt{J}, k, l) \in \Xi(b, \rho)$ for some $(k, l) \in \{1, 2\} \times \{1, 2\}$, then it can be checked that $\chi_{j, l}^{[\xi, \rho, H, h]}(x) \leq 1$ for all $x \in \mathtt{J}_k^H\cap \Lambda_+$. Otherwise, Alternative~(2) in \cref{lem:HTrappedByh} holds for $(\mathtt{J}, 1, 1),\allowbreak (\mathtt{J}, 1, 2),\allowbreak (\mathtt{J}, 2, 1),\allowbreak (\mathtt{J}, 2, 2) \in \Xi(b, \rho)$. We would like to use \cref{lem:StrongTriangleInequality} but first we need to establish bounds on relative angle and relative size. We start with the former.

Define $\omega_\ell(x) = \frac{H(\alpha_\ell x)}{\|H(\alpha_\ell x)\|_2}$ and $\phi_\ell(x) = \GHol^m(\alpha_\ell x)$ for all $x \in \Lambda_+$ and $\ell \in \{0, j\}$. Let $\ell\in \{0,j\}$. For any two points $y, z \in (\mathtt{J}_1^H \sqcup \mathtt{J}_2^H)\cap\Lambda_+$, without loss of generality, we may assume $\|H(\alpha_\ell y)\|_2\leq \|H(\alpha_\ell  z)\|_2$. Using the sine law, Alternative~(2) in \cref{lem:HTrappedByh}, and \cref{eqn:Cone,eqn:CancellationPartnerPoint,eqn:Constantm}, we have
\begin{align*}
\sin(\Theta(\omega_\ell(y), \omega_\ell(z))) &\leq \frac{\|H(\alpha_\ell y) - H(\alpha_\ell z)\|_2}{\max\{\|H(\alpha_\ell y)\|_2, \|H(\alpha_\ell z)\|_2\}} \\
&\leq \frac{E\|\rho_b\|h(\alpha_\ell z) D(\alpha_\ell y, \alpha_\ell z)}{\|H(\alpha_\ell z)\|_2} \\
&\leq 4E\|\rho_b\| \cdot \frac{\epsilon_1 \lambda^m}{\|\rho_b\|} \\
&\leq 4E\epsilon_1\lambda^m \in (0, 1/2),
\end{align*}
Using the cosine law and \cref{eqn:Constantepsilon1,eqn:Constantm}, we have
\begin{align}
\begin{aligned}
\label{eqn:omega_ellLipschitzBound}
\|\omega_\ell(y) - \omega_\ell(z)\| &= \sqrt{2 - 2\cos(\Theta(\omega_\ell(y), \omega_\ell(z)))} \\
&= \sqrt{2 - 2\sqrt{1 - \sin^2(\Theta(\omega_\ell(y), \omega_\ell(z)))}} \\
&\leq \sqrt{2 - 2\sqrt{1 - (4E\epsilon_1\lambda^m)^2}} \\
&\leq \frac{\delta_1\epsilon_1}{4}.
\end{aligned}
\end{align}
For all $x\in \Lambda_+$ and $\ell\in \{0,j\}$, define
\begin{align*}
V_\ell(x) &= e^{\FRet_m^{(a)}(\alpha_\ell x)} \rho_b(\phi_\ell(x)^{-1})H(\alpha_\ell x), \\
\hat{V}_\ell(x) &= \frac{V_\ell(x)}{\|V_\ell(x)\|_2} = \rho_b(\phi_\ell(x)^{-1})\omega_\ell(x).
\end{align*}
Using \cref{eqn:omega_ellLipschitzBound}, we have
\begin{align*}
&\big\|\hat{V}_0(x_2) - \hat{V}_j(x_2)\big\|_2 \\
={}&\|\rho_b(\phi_0(x_2)^{-1})\omega_0(x_2) - \rho_b(\phi_j(x_2)^{-1})\omega_j(x_2)\|_2 \\
={}&\|\rho_b(\phi_j(x_2)\phi_0(x_2)^{-1})\omega_0(x_2) - \omega_j(x_2)\|_2 \\
\geq{}&\|\rho_b(\phi_j(x_2)\phi_0(x_2)^{-1})\omega_0(x_1) - \omega_j(x_1)\|_2 \\
{}&- \|\rho_b(\phi_j(x_2)\phi_0(x_2)^{-1})\omega_0(x_2) - \rho_b(\phi_j(x_2)\phi_0(x_2)^{-1})\omega_0(x_1)\|_2 \\
{}&- \|\omega_j(x_2) - \omega_j(x_1)\|_2 \\
={}&\|\rho_b(\phi_j(x_2)\phi_0(x_2)^{-1})\omega_0(x_1) - \omega_j(x_1)\|_2 - \|\omega_0(x_2) - \omega_0(x_1)\|_2 \\
{}&- \|\omega_j(x_2) - \omega_j(x_1)\|_2 \\
\geq{}&\|\rho_b(\phi_j(x_2)\phi_0(x_2)^{-1})\omega_0(x_1) - \rho_b(\phi_j(x_1)\phi_0(x_1)^{-1})\omega_0(x_1)\|_2 \\
{}&- \|\rho_b(\phi_j(x_1)\phi_0(x_1)^{-1})\omega_0(x_1) - \omega_j(x_1)\|_2 - \delta_1\epsilon_1 \\
={}&\|\rho_b(\phi_0(x_1)^{-1})\omega_0(x_1) - \rho_b(\phi_0(x_1)^{-1}\phi_0(x_2)\phi_j(x_2)^{-1}\phi_j(x_1)\phi_0(x_1)^{-1})\omega_0(x_1)\|_2 \\
{}&- \|\rho_b(\phi_0(x_1)^{-1})\omega_0(x_1) - \rho_b(\phi_j(x_1)^{-1})\omega_j(x_1)\|_2 - \delta_1\epsilon_1 \\
\geq{}&\|\rho_b(\phi_0(x_1)^{-1})\omega_0(x_1) - \rho_b(\BP_j(x_1, x_2))\rho_b(\phi_0(x_1)^{-1})\omega_0(x_1)\|_2 \\
&{}- \big\|\hat{V}_0(x_1) - \hat{V}_j(x_1)\big\|_2 - \delta_1\epsilon_1.
\end{align*}
Recall that we applied \cref{prop:CancellationCylinder} to the unit vector $\omega = \rho_b(\phi_0(x_1)^{-1})\omega_0(x_1)$ in \cref{eqn:CancellationVector}. Let $Z = d(\BP_{j, x_1})_{x_1}(x_2 - x_1)$. Continuing to bound the first term above, we have
\begin{align*}
&\|\omega - \rho_b(\BP_j(x_1, x_2))(\omega)\|_2 \\
\geq{}&\|\omega - \rho_b(\exp(Z))(\omega)\|_2 - \|\rho_b(\exp(Z))(\omega) - \rho_b(\BP_j(x_1, x_2))(\omega)\|_2 \\
\geq{}&\|\omega - \exp(d\rho_b(Z))(\omega)\|_2 - \|\rho_b\| \cdot d_{AM}(\exp(Z), \BP_j(x_1, x_2)) \\
\geq{}&\|d\rho_b(Z)(\omega)\|_2 - \|\rho_b\|^2 \|Z\|^2 - \|\rho_b\| \cdot d_{AM}(\exp(Z), \BP_j(x_1, x_2)) \\
\geq{}&\|d\rho_b(Z)(\omega)\|_2 - \|\rho_b\|^2 C_{\BP}^2 d(x_1, x_2)^2 - C_{\exp, \BP} \cdot \|\rho_b\| \cdot d(x_1, x_2)^2 \quad\text{(\cref{lem:ComparingExpWithBP})} \\
\geq{}&7\delta_1\epsilon_1 - \delta_1\epsilon_1 - \delta_1\epsilon_1 \geq 5\delta_1\epsilon_1. \qquad \text{(by \cref{prop:CancellationCylinder,eqn:Constantepsilon1} )}
\end{align*}
Hence, we have
\begin{align*}
\big\|\hat{V}_0(x_1) - \hat{V}_j(x_1)\big\|_2 + \big\|\hat{V}_0(x_2) - \hat{V}_j(x_2)\big\|_2 \geq 4\delta_1\epsilon_1.
\end{align*}
This implies that there exists $k\in\{1,2\}$ such that
\begin{equation*}
\big\|\hat{V}_0(x_k) - \hat{V}_j(x_k)\big\|_2 \geq 2\delta_1\epsilon_1.
\end{equation*}
For any $x\in \mathtt{J}^{H}_k\cap \Lambda_+$ and $\ell\in \{0,j\}$, using estimates from \cref{eqn:GLipEstimate,eqn:Constantepsilon1,eqn:omega_ellLipschitzBound}, we have
\begin{align*}
&\big\|\hat{V}_\ell(x_k) - \hat{V}_\ell(x)\big\|_2 \\
\leq{}&\|(\rho_b(\phi_\ell(x_k)^{-1}) - \rho_b(\phi_\ell(x)^{-1}))\omega_\ell(x)\|_2 + \|\rho_b(\phi_\ell(x)^{-1})(\omega_\ell(x_k) - \omega_\ell(x))\|_2 \\
\leq{}&A_0 \|\rho_b\| \cdot \|x - x_k\| + \|\omega_\ell(x_k) - \omega_\ell(x)\|_2 \\
\leq{}&A_0\|\rho_b\| \cdot \frac{\delta_1\epsilon_1}{4A_0\|\rho_b\|} + \frac{\delta_1\epsilon_1}{4} \\
={}&\frac{\delta_1\epsilon_1}{2}.
\end{align*}
Hence for all $x \in \mathtt{J}_k^H\cap \Lambda_+$, we have
\begin{equation*}
\big\|\hat{V}_0(x) - \hat{V}_j(x)\big\|_2 \geq \delta_1\epsilon_1 \in (0, 1).
\end{equation*}
Then using the cosine law, the required bound for relative angle is
\begin{align*}
\Theta(V_0(x), V_j(x)) = \Theta(\hat{V}_0(x), \hat{V}_j(x)) \geq \arccos\left(1 - \frac{(\delta_1\epsilon_1)^2}{2}\right) \in (0, \pi).
\end{align*}

We prove the bound on relative size. Choose any $y_0\in \mathtt{J}^H_k\cap \Lambda_+$. We have either $h(\alpha_0y_0)\leq h(\alpha_jy_0)$ or $h(\alpha_jy_0)\leq h(\alpha_0y_0)$. Let $(\ell, \ell') = (0, j)$ and $l = 1$ for the first case and $(\ell, \ell') = (j, 0)$ and $l = 2$ for the second case. Recalling that $\rho_b$ is a unitary representation, by \cref{lem:HTrappedByh}, we have
\begin{align*}
\frac{\|V_\ell(x)\|_2}{\|V_{\ell'}(x)\|_2} &= \frac{e^{\FRet_m^{(a)}(\alpha_\ell x)}\|H(\alpha_\ell x)\|_2}{e^{\FRet_m^{(a)}(\alpha_{\ell'} x)}\|H(\alpha_{\ell'} x)\|_2} \leq \frac{4e^{\FRet_m^{(a)}(\alpha_\ell x) - \FRet_m^{(a)}(\alpha_{\ell'} x)}h(\alpha_\ell x)}{h(\alpha_{\ell'} x)} \\
&\leq \frac{16e^{2T_0}h(\alpha_\ell y_0)}{h(\alpha_{\ell'} y_0)} \leq 16e^{2T_0}
\end{align*}
for all $x \in \mathtt{J}_k^H\cap \Lambda_+$, which is the required bound on relative size. Now using \cref{lem:StrongTriangleInequality}, \cref{eqn:Constantmu}, and $\|H\| \leq h$ for $\|V_0(x) + V_{j}(x)\|_2$ gives 
\begin{equation*}
\chi_{j, l}^{[\xi, \rho, H, h]}(x) \leq 1 \qquad \text{for all $x \in \mathtt{J}_k^H\cap \Lambda_+$}.
\end{equation*}
\end{proof}

For each $\mathtt{J}\in \Omega^\dagger(\log\|\rho_b\|, R_0) $, we use \cref{lem:chiLessThan1} to find $(k,l)$, whose union gives $J\in \mathcal{J}(b, \rho)$ for $(H,h)$. A straightforward derivation using definitions and \cref{lem:chiLessThan1} gives the following lemma (see the proof of \cite[Lemma 5.3]{Sto11} and \cite[Lemma 9.11]{SW21}).

\begin{lemma}
\label{lem:PreserveProperty3}
For all $\xi=a+ib \in \mathbb C$ with $|a| < a_0'$, if $(b, \rho) \in \widehat{M}_0(b_0)$, and if $(H, h) \in \mathcal{C}_{(b, \rho)}(\Lambda_+)$, then there exists $J \in \mathcal{J}(b, \rho)$ such that $\bigl(\mathcal{M}_{\xi, \rho}^m(H), \mathcal{N}_{a, J}^H(h)\bigr)$ also satisfies Property~(3) in \cref{eqn:Cone}.
\end{lemma}

Combining \cref{lem:PreserveProperty2,lem:PreserveProperty4,lem:PreserveProperty3} completes the proof of \cref{prop:PreservingCone}.

\subsection{\texorpdfstring{Stochastic dominance and the proof of \cref{thm:FrameFlowDolgopyat}}{Stochastic dominance and the proof of \autoref{thm:FrameFlowDolgopyat}}}
In this subsection we put together the various components to finish proving \cref{thm:FrameFlowDolgopyat}. Property~(3) in \cref{thm:FrameFlowDolgopyat} is derived from \cref{prop:LDP} but the proof also requires the stochastic dominance technique of Tsujii--Zhang \cite[Section 14]{TZ23}.

Let us begin the proof of \cref{thm:FrameFlowDolgopyat} by fixing for the rest of this subsection, $\eta$ and $a_0$ provided by \cref{lem:dolgopyatoperator}, and $\tilde{\kappa}$ to be the $\kappa$ corresponding to $R_0$ provided by \cref{prop:LDP}. We defer fixing $\kappa$ of \cref{thm:FrameFlowDolgopyat} to end of this subsection. Also let $(b, \rho) \in \widehat{M}_0(b_0)$ and $H \in \Lip\bigl(\Lambda_+, V_\rho^{\oplus \dim(\rho)}\bigr)$ for the rest of this subsection. We then inductively define
\begin{align*}
(H_0, h_0) &= \bigl(H, \|H\|_{1, \|\rho_b\|} \cdot \chi_{\Lambda_+}\bigr) \in \mathcal{C}_{(b, \rho)}(\Lambda_+), \\
(H_j, h_j) &= \bigl(\mathcal{M}_{\xi, \rho}^m(H_{j - 1}), \mathcal{N}_{a, J_{j - 1}}^{H_{j - 1}}(h_{j - 1})\bigr) \in \mathcal{C}_{(b, \rho)}(\Lambda_+) \qquad \text{for all $j \in \N$},
\end{align*}
where $J_{j - 1} \in \mathcal{J}(b, \rho)$ is inductively provided by \cref{prop:PreservingCone} corresponding to $(H_{j -1}, h_{j -1}) \in \mathcal{C}_{(b, \rho)}(\Lambda_+)$ for all $j \in \N$. Fix $\Omega_j = \Omega_{J_{j - 1}}^{H_{j - 1}}$ for all $j \in \N$. Properties~(1) and (2) in \cref{thm:FrameFlowDolgopyat} then follow from \cref{lem:dolgopyatoperator,prop:PreservingCone} respectively.

It remains to prove Property~(3) in \cref{thm:FrameFlowDolgopyat}. We present an adaptation of \cite[Section 14]{TZ23} for our setting. To ease notation, set $\calP := \calP_{(b, \rho)}$, $\calP' := T^{-m}(\calP) \succeq \calP$, and the $\calP$-measurable set $\tilde{\Omega} := \Omega(\log\|\rho_b\|,R_0)$. By \cref{prop:CancellationCylinder}, for all $j \in \N$, the set $\Omega_j \subset \tilde{\Omega}$ is $\calP'$-measurable and $(\calP, \mathsf{c})$-dense. To proceed further, we need \cref{lem:StochasticDominance}. Define
\begin{align*}
\begin{aligned}
\tau(x) &= \inf\{j \in \N: T^{jm}(x) \in \tilde{\Omega}\}, \\
\sigma(x) &= T^{\tau(x)m}(x),
\end{aligned}
\qquad
\text{for all $x \in \Lambda_+$},
\end{align*}
and use similar notation as in \cref{eqn:BirkhoffSum}. Let $n \in \N$, $\kappa \in (0, 1)$, and $\kappa' \in (0, 1)$ be arbitrary. We introduce the sets
\begin{align*}
A_{n, \kappa} &= \{x \in \Lambda_+: \#\{j \in \N: j \leq n, T^{jm}(x) \in \Omega_j\} < \kappa n\} \subset \Lambda_+, \\
B_{n, \kappa} &= \{x \in \Lambda_+: \#\{j \in \N: j \leq n, T^{jm}(x) \in \tilde{\Omega}\} < \kappa n\} \subset A_{n, \kappa}, \\
C_{n, \kappa} &= \{x \in \Lambda_+: \#\{j \in \N: j \leq n, \sigma^j(x) \in \Omega_{\tau_j(x)}\} < \kappa n\} \subset \Lambda_+.
\end{align*}
We first derive \cref{eqn:RelationshipABC}, which is a useful relationship between these sets. For all $x \in \Lambda_+ - B_{n, \kappa}$, we have $\sigma^j(x) = T^{\tau_j(x)m}(x) \in \tilde{\Omega}$ for all $1 \leq j \leq \lceil \kappa n\rceil$ but $\tau_{\lceil \kappa n\rceil}(x) \leq n$. Consequently, for all $x \in A_{n, \kappa\kappa'} - B_{n, \kappa}$, we have
\begin{align*}
\#\{j \in \N: j \leq \lceil \kappa n\rceil, \sigma^{j}(x) \in \Omega_{\tau_j(x)}\} \leq \#\{j \in \N: j \leq n, T^{jm}(x) \in \Omega_j\} < \kappa \kappa' n.
\end{align*}
Thus, we obtain
\begin{align}
\label{eqn:RelationshipABC}
A_{n, \kappa\kappa'} - B_{n, \kappa} \subset C_{\lceil \kappa n\rceil, \kappa'}.
\end{align}
Now, Property~(3) in \cref{thm:FrameFlowDolgopyat} amounts to obtaining appropriate control of $\nu(A_{n, \kappa\kappa'})$. Since we can already control $\nu(B_{n, \kappa})$ by \cref{prop:LDP}, it suffices to gain control of $\nu(C_{\lceil \kappa n\rceil, \kappa'})$. This is the content of the following lemma.

\begin{lemma}
\label{lem:StochasticDominance}
There exists $\kappa \in (0, 1)$ such that $\nu(C_{n, \kappa}) < e^{-\kappa n}$ for all $n \in \N$.
\end{lemma}

\begin{proof}
Fix $\kappa = \frac{\mathsf{c}}{C_{\mathrm{cyl}}^2}$. Define the set $\Upsilon_k = \{x \in \Delta_0: \sigma^k(x) \in \Omega_{\tau_k(x)}\}$ for all $k \in \N$ and the partition $\calQ_k = \sigma^{-k}(\calP)$ for all $k \in \Z_{\geq 0}$. We first argue that $\Upsilon_k$ is $\calQ_{k + 1}$-measurable for all $k \in \N$. For any atom $\calQ_k(x)$, for some $x \in \Delta_0$, there exists an inverse branch $\gamma \in \calH^{\tau_k(x)}$ and $x' \in \Delta_0$ such that $\calQ_k(x) = \gamma\calQ_0(x')$. Note that $\tau_k$ is constant on $\calQ_k(x)$. Thus, we have $\Upsilon_k \cap \calQ_k(x) = \gamma(\Omega_{\tau_k(x)} \cap \calQ_0(x'))$ where $\Omega_{\tau_k(x)} \cap \calQ_0(x')$ is $\calP'$-measurable as desired. Since $\Omega_{\tau_k(x)}$ is $(\calQ_0, \mathsf{c})$-dense, using \cref{lem:CylinderEstimate}, we estimate the expected value
\begin{align*}
\mathbb E(\chi_{\Upsilon_k} | \calQ_k)(x) = \frac{\nu(\Upsilon_k \cap \calQ_k(x))}{\nu(\calQ_k(x))} \geq \frac{\nu(\Omega_{\tau_k(x)} \cap \calQ_0(x'))}{C_{\mathrm{cyl}}^2\nu(\calQ_0(x'))} > \frac{\mathsf{c}}{C_{\mathrm{cyl}}^2} = \kappa.
\end{align*}
Thus our stochastic process is dominated by an i.i.d. coin flipping process (which is a Markov chain) with rate $\kappa$ which is known to satisfy LDP again by transfer operator techniques (see \cite[Section 3.1]{DZ10}).
\end{proof}

Fix $\tilde{\kappa}'$ to be the $\kappa$ provided by \cref{lem:StochasticDominance}. Fix $\kappa = \tilde{\kappa}\tilde{\kappa}'$ for \cref{thm:FrameFlowDolgopyat}. Then, using \cref{eqn:RelationshipABC} and LDP in \cref{prop:LDP,lem:StochasticDominance}, we have
\begin{align*}
\nu(A_{n, \kappa}) &= \nu(A_{n, \tilde{\kappa}\tilde{\kappa}'}) \leq \nu(B_{n, \tilde{\kappa}}) + \nu(C_{\lceil \tilde{\kappa} n\rceil, \tilde{\kappa}'}) < e^{-\tilde{\kappa} n} + e^{-\tilde{\kappa}' n} \leq 2e^{-\kappa n}
\end{align*}
which completes the proof of Property~(3) in \cref{thm:FrameFlowDolgopyat}.

\section{Exponential mixing of the frame flow}
\label{sec:ExponentialMixing}
With \cref{cor:factormap} available, we use a convolution argument using a bump function as in the proofs of \cite[Theorem 3.1]{LP22} and \cite[Theorem 3.1.4]{Sar22b} to obtain \cref{thm:ExponentialMixing} from the exponential mixing of the semiflow $\{\hat{T}_t\}_{t \geq 0}$ on $\Lambda^\Ret\times M$ with respect to the measure $\hat{\nu}^\Ret\otimes m^{\mathrm{Haar}}$. By integrating out the strong stable direction as in \cite[Section 8.2]{AGY06} or \cite[Section 10.3]{SW21}, it suffices to prove exponential mixing for the semiflow $\{T_t\}_{t \geq 0}$ on $\Lambda_+^\Ret\times M$. More precisely, we first define the space
\begin{equation*}
\Lambda_+^{\Ret}\times M=\{(x,m,s)\in \Lambda_+\times M \times \mathbb{R}:0\leq s<\Ret(x)\}.
\end{equation*}
Since the holonomy map $\mathscr{H}:\Delta_{\sqcup}\times \Lambda_-\to M$ is independent of the contracting direction by \cref{lem:stable direction}, we can regard $\mathscr{H}$ as a map on $\Delta_{\sqcup}$. Define the semiflow $\{T_t\}_{t\geq 0}$ by
\begin{equation*}
T_t(x,m,s)=(T^nx,\Hol^n(x)m,s+t-\Ret_n(x))
\end{equation*}
for all $(x,m,s)\in \Lambda_+^{\Ret}\times M$, where $n \in \Z_{\geq 0}$ such that $0\leq s+t-\Ret_n(x)<\Ret(T^nx)$. Recall that we equipped $\Lambda^{\Ret}\times M$ with the measure $d(\hat{\nu}^{\Ret}\otimes m^{\mathrm{Haar}}):=d\hat{\nu}\, dm \, d\Leb/\bar{\Ret}$, where $\Leb$ is the Lebesgue measure and $\bar{\Ret}=\hat{\nu}(\Ret)$. Similarly, on $\Lambda_+^{\Ret}\times M$, we consider the measure $\nu^{\Ret}\otimes m^{\mathrm{Haar}}$ defined by $d\bigl({\nu}^{\Ret}\otimes m^{\mathrm{Haar}}\bigr):=d{\nu} \, dm \, d\Leb/\bar{\Ret}$. Compared to the symbolic frame flow model $\bigl(\Lambda^\Ret,\{\hat{T}_t\}_{t \geq 0},\hat{\nu}^\Ret\bigr)$,  the dynamical system $\bigl(\Lambda_+^\Ret,\{{T}_t\}_{t \geq 0},{\nu}^\Ret\bigr)$ is to forget the contracting direction.

We introduce some norms. Denote by $\nabla_M$ the Levi-Civita connection on $M$. For a function $\phi: \Lambda_+^\Ret\times M \to \R$ differentiable in the $\mathbb{R}$ and $M$ coordinates, we define the norms
\begin{align*}
\|\phi\|_{C^r_M}
={}&\sup\{\|\nabla_M^k\phi \|_{\infty}: 0 \leq k \leq r\},\\
\|\phi\|_{C^r_MC^1_\R},
={}&\sup\{\|\nabla_M^k\partial_t\phi \|_{\infty}+\|\nabla_M^k\phi \|_{\infty}: 0 \leq k \leq r\}.
\end{align*}
For any function $\psi: \Lambda_+^\Ret\times M \to \R$, we define the Lipschitz norm
\begin{equation*}
\|\psi\|_{\Lip}=\|\psi\|_{\infty}+\sup\left\{\frac{|\psi(x,m,t)-\psi(x',m',t')|}{d_{\mathrm{E}}(x,x')+d_M(m,m')+|t-t'|}: (x,m,t)\neq (x',m',t')\right\}.
\end{equation*}

\begin{theorem}\label{thm:semiflowexponential}
There exist $C>0$, $\eta>0$, and $r\in\N$ such that for all functions $\phi, \psi: \Lambda_+^\Ret\times M \to \R$ with bounded $C^r_MC^1_\R$ norm and bounded Lipschitz norm respectively, and $t>0$, we have
\begin{multline*}
\left|\int_{\Lambda_+^\Ret\times M} (\phi \circ {T}_t) \cdot\psi \, \dd\bigl({\nu}^\Ret\otimes m^{\mathrm{Haar}}\bigr)- \bigl({\nu}^\Ret\otimes m^{\mathrm{Haar}}\bigr)(\phi) \cdot \bigl({\nu}^\Ret\otimes m^{\mathrm{Haar}}\bigr)(\psi)\right|\\
\leq Ce^{-\eta t}\|\phi\|_{C^r_MC^1_\R}\|\psi\|_{\Lip}.
\end{multline*}
\end{theorem}

The proof is similar to \cite[Lemma 10.3]{SW21}, and the main difference is that we utilize an infinite symbolic coding with an unbounded return time map and the function is only Lipschitz in the direction $\Lambda_+$. We write only the required modifications in detail.

Let $\phi \in C(\Lambda_+^\Ret\times M, \mathbb R)$ and $\xi \in \mathbb C$. Define $\hat{\phi}_\xi \in B(\Lambda_+, L^2(M, \mathbb C))$ by
\begin{align*}
\hat{\phi}_\xi(x)(m) = \int_0^{\Ret(x)} \phi(x, m, t)e^{-\xi t} \, d\Leb(t) \qquad \text{for all $m \in M$ and $x \in \Lambda_+$}.
\end{align*}
We can decompose it further as $\hat{\phi}_\xi(x) = \sum_{\rho \in \widehat{M}} \hat{\phi}_{\xi, \rho}(x) \in \operatorname*{\widehat{\bigoplus}}_{\rho \in \widehat{M}} V_\rho^{\oplus \dim(\rho)}$ for all $x \in \Lambda_+$. Let $\rho \in \widehat{M}$. Defining $\phi_\rho \in C(\Lambda_+^\Ret\times M, \mathbb R)$ by the projection $\phi_\rho(x, \cdot, t) = [\phi(x, \cdot, t)]_\rho \in V_\rho^{\oplus \dim(\rho)}$ for all $x \in \Lambda_+$ and $t \in \mathbb R_{\geq 0}$, we have
\begin{align*}
\hat{\phi}_{\xi, \rho}(x)(m) = \int_0^{\Ret(x)} \phi_\rho(x, m, t) e^{-\xi t} \, d\Leb(t) \qquad \text{for all $m \in M$ and $x \in \Lambda_+$}.
\end{align*}

The following lemma serves as \cite[Lemma 10.2]{SW21} in our setting.
\begin{lemma}
\label{lem:ExtractNormOfLaplaceTransformDecay}
There exist $C > 0$ and $a_0 > 0$ such that for all $\rho \in \widehat{M}$, functions $\phi, \psi: \Lambda_+^\Ret\times M \to \R$ with bounded $C^r_MC^1_\R$ norm and bounded Lipschitz norm respectively, and $\xi=a+ib \in \mathbb C$ with $|a| \leq a_0$, we have
\begin{align*}
\left(\int_{\Lambda_+} \big\|\hat{\phi}_\xi(x)\big\|_{C^r_M}^2 \, \dd\nu(x)\right)^{\frac{1}{2}} &\leq C \frac{\|\phi\|_{C^r_MC^1_\R}}{\max\{1, |b|\}},\\
\big\|\mathcal{M}_{\xi, \rho}\big(\hat{\psi}_{-\xi, \rho}\big)\big\|_{1, \|\rho_b\|} &\leq C \frac{\|{\psi}\|_{\Lip}}{\max\{1, |b|\}}.
\end{align*}
\end{lemma}

\begin{proof}
Recall the constants $C_2 > 0$ from \cref{eqn:ConstantC1'C2} and $A_0 > 0$ from \cref{lem:LasotaYorke}. Fix $C_5=\frac{9\max\{1, C_2\}}{a_0^2} \int_{\Lambda_+} e^{(\epsilon_0/2)\Ret(x)} \, \dd\nu(x)<\infty$ and $C = 2C_5(1 + A_0)$. Let $a_0=\epsilon_0/8$.
We show the first inequality. If $|b| \leq 1$, from the definition of $\hat{\phi}_\xi(x)$ we have
\begin{align*}
\int_{\Lambda_+} \big\|\hat{\phi}_\xi(x)\big\|_{C^r_M}^2 \, d\nu(x) \leq \int_{\Lambda_+} e^{a_0\Ret(x)}\|{\phi}(x)\|_{C^r_M}^2 \, d\nu(x) \leq C_5 \|\phi\|_{C^{r}_{M}C^1_{\mathbb{R}}}^2.
\end{align*}
If $|b| \geq 1$, integrating by parts gives
\begin{align*}
&\nabla^k_{M} \hat{\phi}_\xi(x) = \int_0^{\Ret(x)} \nabla^k_{M} \phi(x, \cdot, t)e^{-\xi t} \, d\Leb(t) \\
={}&\left[-\frac{1}{\xi} \nabla^k_{M} \phi(x, \cdot, t)e^{-\xi t}\right]_{t = 0}^{t \nearrow \Ret(x)} + \frac{1}{\xi}\int_0^{\Ret(x)} \left.\frac{d}{dt'}\right|_{t' = t} \nabla^k_{M} \phi(x, \cdot, t') \cdot e^{-\xi t} \, d\Leb(t)
\end{align*}
for all $x \in \Lambda_+$ and $0 \leq k \leq r$. Hence using the exponential tail property (see Property~(5) in \cref{prop:Coding}),
\begin{align*}
&\int_{\Lambda_+} \big\|\hat{\phi}_\xi(x)\big\|_{C^r_M}^2 \, d\nu(x)\\
\leq{}&\frac{1}{|b|^2} \int_{\Lambda_+} \left(2\|{\phi}(x)\|_{C^r_M} e^{a_0\Ret(x)} + \|{\phi}(x)\|_{C^{r}_{M}C^1_{\mathbb{R}}} \Ret(x)e^{a_0\Ret(x)}\right)^2 \, d\nu(x) \\
\leq{}&C_5 \frac{\|{\phi}\|_{C^{r}_{M}C^1_{\mathbb{R}}}^2}{|b|^2}.
\end{align*}

Now we show the second inequality.
\begin{align*}
\mathcal{M}_{\xi, \rho}\big(\hat{\psi}_{-\xi, \rho}\big)(x) = \sum_{\gamma\in\calH } e^{\FRet^{(a)}(\gamma x)} \rho_b(\GHol(\gamma x)^{-1}) \hat{\psi}_{-\xi, \rho}(\gamma x)
\end{align*}
for all $x \in \Lambda_+$. We first bound the $L^\infty$ norm. Using similar estimates as for the first proven inequality, we have
\begin{align}
\label{eqn:L2BoundFor_psi_hat}
\left\|\hat{\psi}_{-\xi, \rho}(x)\right\|_2 \leq \left\|\hat{\psi}_{-\xi}(x)\right\|_2 \leq \left\|\hat{\psi}_{-\xi}(x)\right\|_\infty \leq  \frac{\max\{1, \Ret(x)\}e^{a_0\Ret(x)}\|{\psi}\|_{\Lip}}{\max\{1, |b|\}}.
\end{align}
So, as $\rho_b$ is a unitary representation, we have
\begin{align*}
\big\|\mathcal{M}_{\xi, \rho}\big(\hat{\psi}_{-\xi, \rho}\big)\big\|_\infty \leq  \frac{\|{\psi}\|_{\Lip}}{\max\{1, |b|\}}\mathcal{L}_a\bigl(e^{2a_0\Ret}\bigr) \leq  \frac{\|{\psi}\|_{\Lip}}{\max\{1, |b|\}}\mathcal{L}_{a - 2a_0}(\chi_{\Lambda_+}).
\end{align*}
Again by the exponential tail property,
\begin{align*}
\big\|\mathcal{M}_{\xi, \rho}\big(\hat{\psi}_{-\xi, \rho}\big)\big\|_\infty \leq C_5 \frac{\|{\psi}\|_{\Lip}}{\max\{1, |b|\}}.
\end{align*}

Next, we deal with the Lipschitz norm. By the same computation as in the proof of Property~(2) in \cref{lem:LasotaYorke} for $H=\hat{\psi}_{-\xi, \rho}$, we obtain \cref{eqn:M difference} and we show how to estimate the three terms in the last but two equation. Similar to the previous estimate, in the first two terms, we can use \cref{eqn:L2BoundFor_psi_hat} to estimate $\|H(\gamma x)\|_2$ and obtain
\begin{align*}
&\sum_{\gamma \in \calH} \Bigl(\frac{A_0}{2}\delta_{1, \varrho}d_{\mathrm{E}}(x, x')e^{\FRet^{(a)}(\gamma x)} \|H(\gamma x)\|_2 + 2C_2e^{\FRet^{(a)}(\gamma x')} \|\rho_b\| d_{\mathrm{E}}(x, x') \|H(\gamma x)\|_2\Bigl)  \\
\leq{}&A_0 \cdot \frac{1 + \|\rho_b\|}{\max\{1,|b|\}} \cdot \|\psi\|_{\Lip} \cdot d_{\mathrm{E}}(x,x')\mathcal{L}_{a - 2a_0}(\chi_{\Lambda_+})\\
\leq{}&C_5A_0 \cdot \frac{1 + \|\rho_b\|}{\max\{1,|b|\}}\cdot\|\psi\|_{\Lip} \cdot d_{\mathrm{E}}(x,x').
\end{align*}

For the last term, we first need to estimate $\|H(\gamma x)-H(\gamma x')\|_2$. Without loss of generality, we may assume $\Ret(\gamma x)\leq \Ret(\gamma x')$. Then
\begin{align*}
&\| H(\gamma x)-H(\gamma x')\|_2=\left\| \hat{\psi}_{-\xi, \rho}(\gamma x)-\hat{\psi}_{-\xi, \rho}(\gamma x')\right\|_2\\
\leq{}&|\Ret(\gamma x)-\Ret(\gamma x')|\cdot\|\psi_\rho\|_2e^{a_0\Ret(\gamma x)}\\
{}&+\int_{0}^{\Ret(\gamma x)}\|\psi_\rho(\gamma x,t)-\psi_\rho(\gamma x',t)\|_2e^{a_0t} \, d\Leb(t) \\
\leq{}&\bigl(C_2e^{a_0\Ret(\gamma x')}+\Ret(\gamma x)e^{a_0\Ret(\gamma x)}\bigr)\|\psi\|_{\Lip} d_{\mathrm{E}}(x,x'),
\end{align*}
by an estimate as in \cref{eqn:FaLipEstimate}. Thus, by a similar calculation as above, we have
\begin{align*}
\sum_{\gamma \in \calH} e^{\FRet^{(a)}(\gamma x)} \|H(\gamma x)-H(\gamma x')\|_2 \leq C_5\|\psi\|_{\Lip} d_{\mathrm{E}}(x,x')
\end{align*}
Summing all the terms, we obtain
\[ \frac{1}{\max\{1, \|\rho_b\|\}}\Lip\left(\mathcal{M}_{\xi, \rho}\left(\hat{\psi}_{-\xi, \rho}\right)\right)\leq C_5(1 + 2A_0)\frac{\|\psi\|_{\Lip}}{\max\{1, |b|\}}. \]
The proof is completed by using the definition of $\|\cdot\|_{1,\|\rho_b\|}$ and \cref{lem:LieTheoreticNormBounds}.
\end{proof}

We can prove an analog of \cite[Lemma 10.3]{SW21} by replacing $\sup_{u\in U}\|\hat{\phi}_\xi(u) \|_{C^r_M} $ with $\left(\int_{\Lambda_+} \big\|\hat{\phi}_\xi(x)\big\|_{C^r_M}^2 \, \dd\nu(x)\right)^{\frac{1}{2}}$ and following the same argument. This gives exponential mixing of the semiflow, i.e., \cref{thm:semiflowexponential}.

\appendix
\section{Analyticity of the complex transfer operator}
\label{sec:appendix}
Recall that for $s\in \mathbb{C}$, we introduced the complex transfer operator $L_s:=\mathcal{L}_{(\delta+s)\Ret}:\Lip(\Lambda_+,\mathbb{C})\to \Lip(\Lambda_+,\mathbb{C})$ defined by, for $u\in \Lip(\Lambda_+, \C)$ and $x\in \Lambda_+$,
\begin{equation*}
L_s(u)(x)=\sum_{\gamma\in \calH} e^{-(\delta+s)\Ret(\gamma x)} u(\gamma x).
\end{equation*}
An operator-valued map on a subset of $\C$ is said to be analytic (or holomorphic) if it is Fr\'{e}chet differentiable. In this appendix, we prove \cref{lem:analyticity}. The proof is similar to \cite[Lemma 11.17]{BQ16}, and we include it for completeness of the paper.

\begin{lemma}
\label{lem:analyticity}
For any $s\in \mathbb{C}$ with $\Re{s}>-\frac{\epsilon_0}{2}$, the complex transfer operator $L_s$ depends analytically on $s$.
\end{lemma}

Note that Property~(4) in \cref{prop:Coding} states that there exists $C_1>0$ such that for any $\gamma \in \calH$ and for any $x\in \Delta_0$, $\|(d(\log \|d\gamma\|))_x\|<C_1$. Note that $\Ret(\gamma x)=-\log \|(d\gamma)_x\|$. As a result, there exists $C_6>1$ such that for any $\gamma\in\calH$ and any $x\in \Delta_0$,
\begin{gather*}
\frac{1}{C_6}\|\Ret\circ \gamma\|_{\infty} \leq \Ret(\gamma x)\leq C_6 \|\Ret\circ\gamma\|_{\infty},\\
\|\Ret\circ \gamma\|_{\infty}-C_6 \leq \Ret(\gamma x)\leq \|\Ret\circ\gamma\|_{\infty}+C_6.
\end{gather*}
Also Property~(5) in \cref{prop:Coding} states that there exists $\epsilon_0\in(0,1)$ such that $\int_{\Lambda_+}e^{\epsilon_0 \Ret}<\infty$. Using the quasi-invariant property of $\mu$, we have
\begin{equation}
\label{eqn:convergence tail}
\sum_{\gamma\in \mathcal{H}}e^{-(\delta-\epsilon_0)\|\Ret\circ \gamma\|_{\infty}}<\infty.
\end{equation}

\begin{proof}[Proof of \cref{lem:analyticity}]
We fix $s\in \mathbb{C}$ with $\Re{s}>-\frac{\epsilon_0}{2}$. Pick any $\theta\in \mathbb{C}$ with 
$C_6|\theta|<\frac{\epsilon_0}{2}$.

For $m\in \Z_{\geq 0}$, we introduce the operator $L_{s,\theta,m}: \Lip(\Lambda_+,\mathbb{C}) \to \Lip(\Lambda_+,\mathbb{C})$ given by, for $u\in \Lip(\Lambda_+,\mathbb{C})$ and $x\in \Lambda_+$,
\begin{equation*}
L_{s,\theta,m}(u)(x)=\theta^m\sum_{\gamma\in \calH} (-\Ret(\gamma x))^m e^{-(\delta+s)\Ret(\gamma x)} u(\gamma x).
\end{equation*}
Note that for $m=0$, this operator is equal to $L_{s}$. Since for any $\gamma\in \calH$ and $x\in \Lambda_+$,
\begin{equation}
\label{eqn:power series}
e^{-(\delta+s+\theta)\Ret(\gamma x)}=\sum_{m=0}^{\infty}\frac{1}{m !}(-\theta \Ret(\gamma x))^{m}e^{-(\delta+s)\Ret(\gamma x)},
\end{equation}
to get the analyticity of $L$ in a neighborhood of $s$, it suffices to check that for all $u\in \Lip(\Lambda_+,\mathbb{C})$, we have
\begin{enumerate}
\item $\sum_{m=0}^{\infty}\frac{1}{m!} L_{s,\theta,m}(u)(x)$ converges pointwise to $L_{s+\theta}(u)(x)$ for any $x\in \Lambda_+$;
\item the absolute convergence of the series
\begin{equation}
\label{absolute convergence}
\sum_{m=0}^{\infty} \frac{1}{m!} \lVert L_{s,\theta,m}(u)\rVert_{\Lip}\leq M \lVert u\rVert_{\Lip},
\end{equation}
for some constant $M$ independent of $u$ and $\theta$.
\end{enumerate}
We start with the claim: for any $x\in \Lambda_+$, the sum
\begin{equation}
\label{power series}
\sum_{m=0}^{\infty} \sum_{\gamma\in \calH} \frac{1}{m!} |\theta|^m \Ret(\gamma x)^m e^{-(\delta-\frac{\epsilon_0}{2})\Ret(\gamma x)}
\end{equation} 
converges to $\sum_{\gamma \in \calH} e^{-(\delta-\frac{\epsilon_0}{2}-|\theta|)\Ret(\gamma x)}$.

We prove the claim. For any $n\in \mathbb{N}$, we can switch the order of summation:
\begin{align*}
\sum_{m=0}^{n}\sum_{\gamma\in \calH} \frac{1}{m!} |\theta|^m \Ret(\gamma x)^m e^{-(\delta-\frac{\epsilon_0}{2})\Ret(\gamma x)}
= \sum_{\gamma\in \calH}\sum_{m=0}^n \frac{1}{m!} |\theta|^m \Ret(\gamma x)^m e^{-(\delta-\frac{\epsilon_0}{2})\Ret(\gamma x)},
\end{align*}
where the convergence of the series on the left can be deduced from \cref{eqn:power series,eqn:convergence tail} and the assumption that $C_6|\theta|<\frac{\epsilon_0}{2}$.
We estimate
\begin{align}
\label{finite power series}
&\left|\sum_{\gamma\in \calH}\sum_{m=0}^n\frac{1}{m!} |\theta|^m \Ret(\gamma x)^m e^{-(\delta-\frac{\epsilon_0}{2})\Ret(\gamma x)}-\sum_{\gamma\in \calH}e^{-(\delta-\frac{\epsilon_0}{2}-|\theta|)\Ret(\gamma x)}\right|\nonumber\\
={}&\left|\sum_{\gamma\in \calH} \left(\sum_{m=0}^n \frac{1}{m!}|\theta|^m \Ret(\gamma x)^m-e^{|\theta|\Ret(\gamma x)}\right)e^{-(\delta-\frac{\epsilon_0}{2})\Ret(\gamma x)}\right|.
\end{align}
Fix any $\epsilon>0$. It follows from the exponential tail property (see Property~(5) in \cref{prop:Coding}) that there exists $M>1$ such that
\begin{equation}
\label{exponential tail power series}
\sum_{\gamma\in \calH, \Ret(\gamma x)>M}e^{-(\delta-\frac{\epsilon_0}{2}-|\theta|)\Ret(\gamma x)}<\epsilon.
\end{equation}
At the same time, note that for all sufficiently large $n\in \mathbb{N}$, we have the uniform convergence
\begin{equation}
\label{ez}
\left|\sum_{m=0}^{n}\frac{1}{m!}z^m-e^z\right|<\epsilon
\end{equation}
for any $z\in \mathbb{C}$ with $|z|<M$. Now we can see that
\[\eqref{finite power series}\ll \epsilon\]
by dividing the sum $\sum_{\gamma\in \calH}$ into $\sum_{\gamma\in \calH, \Ret(\gamma x)\leq M}$ and $\sum_{\gamma\in \calH,\Ret(\gamma x)>M}$ and applying \cref{exponential tail power series,ez}. This finishes the proof of the claim.

Coming back to statement (1), it can be shown as in the convergence of \eqref{power series}. Next, we show the absolute convergence in \cref{absolute convergence}. We bound the sup norm: for any $x\in \Lambda_+$,
\begin{align*}
\left|\sum_{m=0}^{\infty} \frac{1}{m!}L_{s,\theta,m}(u)(x)\right| &\leq \|u\|_{\infty}\sum_{m=0}^{\infty}\frac{1}{m!} \sum_{\gamma\in \calH} |\theta|^{m}\Ret(\gamma x)^m e^{-(\delta-\frac{\epsilon_0}{2})\Ret(\gamma x)}\\
&\ll \|u\|_{\infty} \sum_{\gamma\in \calH} e^{-(\delta-\epsilon_0)\|\Ret\circ\gamma\|_{\infty}}.
\end{align*}
Hence, $\left|\sum_{m=0}^{\infty}\frac{1}{m!}L_{s,\theta,m}(u)\right|_{\infty}$ is bounded by a uniform multiple of $\|u\|_{\infty}$.
It remains to bound the Lipschitz seminorm. For any $m\in \mathbb{N}$ and any $x,y\in \Lambda_+$:
\begin{align*}
&\frac{L_{s,\theta},m(u)(x)-L_{s,\theta,m}(u)(y)}{d_{\operatorname{E}}(x,y)}=A_m+B_m+C_m,\,\,\, \text{where}\\
& A_m=\theta^m\sum_{\gamma\in \calH}\frac{(-\Ret(\gamma x))^m-(-\Ret(\gamma y))^m}{d_{\operatorname{E}}(x,y)} e^{-(\delta+s)\Ret(\gamma x)} u(\gamma x),\\
& B_m=\theta^m \sum_{\gamma\in \calH} (-\Ret(\gamma y))^{m} \frac{e^{-(\delta+s)\Ret(\gamma x)}-e^{-(\delta+s)\Ret(\gamma y)}}{d_{\operatorname{E}}(x,y)} u(\gamma x),\\
&C_m=\theta^m \sum_{\gamma\in \calH} (-\Ret(\gamma y))^m e^{-(\delta+s)\Ret(\gamma y)} \frac{u(\gamma x)-u(\gamma y)}{d_{\operatorname{E}}(x,y)}.
\end{align*}
Since for any $a,b\in \mathbb{C}$,
\begin{equation*}
|a^m-b^m|\leq m \max\{|a|,|b|\}^{m-1} |a-b|,
\end{equation*}
we get
\begin{align*}
|A_m|&\ll \|u\|_{\infty}|\theta|^m\sum_{\gamma\in \calH} m (C_6 \|\Ret\circ \gamma \|_{\infty})^{m-1} \frac{|\Ret(\gamma x)-\Ret(\gamma y)|}{d_{\operatorname{E}}(x,y)} e^{-(\delta+\Re s)\|\Ret\circ \gamma\|_{\infty}}\\
&\ll \|u\|_{\infty} \sum_{\gamma\in \calH} m |\theta|^m (C_6 \|\Ret \circ \gamma \|_{\infty})^{m-1} e^{-(\delta+\Re s)\|\Ret\circ \gamma\|_{\infty}}.
\end{align*}

To continue, we use the power series definition of the exponential function in the form of $e^{z}=\sum_{m=1}^{\infty}\frac{m}{m!}z^{m-1}$ for any $z\in \mathbb{C}$. We get
\begin{align*}
\sum_{m=1}^{\infty}\frac{1}{m!}|A_m| &\ll \sum_{m=1}^{\infty} \sum_{\gamma\in \calH} \frac{m}{m!} |\theta|^{m} (C_6\|\Ret\circ \gamma \|_{\infty})^{m-1} e^{-(\delta+\Re s)\|\Ret\circ \gamma\|_{\infty}} \|u\|_{\infty}\\
&\ll \|u\|_{\infty} \sum_{\gamma\in \calH} e^{-(\delta+\Re s-C_6|\theta|)\|\Ret\circ \gamma\|_{\infty}},
\end{align*}
where the equality is proved as in the convergence of \eqref{power series}. So, we have shown that the sum $\sum_{m=1}^{\infty}\frac{1}{m!}A_m$ is bounded by a uniform multiple of $\|u\|_{\infty}$.

For $B_m$, we use the following inequality: for any $a,b\in \mathbb{C}$,
\begin{equation*}
|e^a-e^b|\leq \max \bigl\{e^{\Re a}, e^{\Re b}\bigr\} |a-b|.
\end{equation*}
We get
\begin{align*}
|B_m|&\ll |\theta|^m \sum_{\gamma\in \calH} (\Ret(\gamma y))^m e^{-(\delta+\Re s)\|\Ret\circ\gamma\|_{\infty}} \frac{|(\delta+s)\Ret(\gamma x)-(\delta+s)\Ret(\gamma y)|}{d_{\operatorname{E}}(x,y)} \|u\|_{\infty}\\
&\ll \sum_{\gamma\in \calH} |\theta|^m (C_6\|\Ret\circ \gamma \|_{\infty})^m e^{-(\delta+\Re s)\|\Ret\circ \gamma\|_{\infty}} \|u\|_{\infty}.
\end{align*}
Hence, we have
\begin{equation*}
\sum_{m=0}^{\infty} \frac{1}{m!} |B_m|\ll \|u\|_{\infty} \sum_{\gamma\in \calH} e^{-(\delta+\Re s-C_6|\theta|)\|\Ret\circ \gamma\|_{\infty}},
\end{equation*}
which bounds the sum on the left by a uniform multiple of $\|u\|_{\infty}$.

Finally, since each $\gamma\in \calH$ acts on $\Lambda_+$ by contraction, we have
\begin{equation*}
\frac{u(\gamma x)-u(\gamma y)}{ d_{\operatorname{E}}(x,y)}=\frac{u(\gamma x)-u(\gamma y)}{d_{\operatorname{E}}(\gamma x, \gamma y)} \cdot \frac{d_{\operatorname{E}}(\gamma x,\gamma y)}{d_{\operatorname{E}}(x,y)}\leq \Lip(u).
\end{equation*}
Hence, 
\begin{equation*}
|C_m|\ll |\theta|^m \sum_{\gamma \in \calH} |C_6\|\Ret\circ \gamma\|_{\infty}|^m e^{-(\delta+\Re s)\|\Ret\circ \gamma\|_{\infty}} \Lip(u),
\end{equation*}
which implies
\begin{equation*}
\sum_{m=0}^{\infty}\frac{1}{m!}|C_m| \leq \Lip(u)\sum_{\gamma\in \calH} e^{-(\delta+\Re s-C_6|\theta|)\|\Ret\circ \gamma\|_{\infty}},
\end{equation*}
again bounding the sum on the left by a uniform multiple of $\Lip(u)$.
\end{proof}

\nocite{*}
\bibliographystyle{alpha_name-year-title}
\bibliography{References}

\end{document}

%% file: cusp.pdf_tex
%% Creator: Inkscape 1.1.1 (c3084ef, 2021-09-22), www.inkscape.org
%% Accompanies image file 'cusp.pdf' (pdf, eps, ps)
\begingroup%
  \makeatletter%
  \providecommand\color[2][]{%
    \errmessage{(Inkscape) Color is used for the text in Inkscape, but the package 'color.sty' is not loaded}%
    \renewcommand\color[2][]{}%
  }%
  \providecommand\transparent[1]{%
    \errmessage{(Inkscape) Transparency is used (non-zero) for the text in Inkscape, but the package 'transparent.sty' is not loaded}%
    \renewcommand\transparent[1]{}%
  }%
  \providecommand\rotatebox[2]{#2}%
  \newcommand*\fsize{\dimexpr\f@size pt\relax}%
  \newcommand*\lineheight[1]{\fontsize{\fsize}{#1\fsize}\selectfont}%
  \ifx\svgwidth\undefined%
    \setlength{\unitlength}{420.43622457bp}%
    \ifx\svgscale\undefined%
      \relax%
    \else%
      \setlength{\unitlength}{\unitlength * \real{\svgscale}}%
    \fi%
  \else%
    \setlength{\unitlength}{\svgwidth}%
  \fi%
  \global\let\svgwidth\undefined%
  \global\let\svgscale\undefined%
  \makeatother%
  \begin{picture}(1,0.65078187)%
    \lineheight{1}%
    \setlength\tabcolsep{0pt}%
    \put(0,0){\includegraphics[width=\unitlength,page=1]{cusp.pdf}}%
    \put(0.42109719,0.34865532){\makebox(0,0)[lt]{\lineheight{1.25}\smash{\begin{tabular}[t]{l}$\Delta_\infty'$\end{tabular}}}}%
    \put(0.75081546,0.40652781){\makebox(0,0)[lt]{\lineheight{1.25}\smash{\begin{tabular}[t]{l}$\mathbb{R}^k$\end{tabular}}}}%
  \end{picture}%
\endgroup%

%% file: Final_Paper.bbl
\begin{thebibliography}{BMMW17}

\bibitem[AM16]{AM16}
Vitor Ara\'{u}jo and Ian Melbourne.
\newblock Exponential decay of correlations for nonuniformly hyperbolic flows
  with a {$C^{1+\alpha}$} stable foliation, including the classical {L}orenz
  attractor.
\newblock {\em Ann. Henri Poincar\'{e}}, 17(11):2975--3004, 2016.

\bibitem[AG13]{AG13}
Artur Avila and S\'{e}bastien Gou\"{e}zel.
\newblock Small eigenvalues of the {L}aplacian for algebraic measures in moduli
  space, and mixing properties of the {T}eichm\"{u}ller flow.
\newblock {\em Ann. of Math. (2)}, 178(2):385--442, 2013.

\bibitem[AGY06]{AGY06}
Artur Avila, S\'{e}bastien Gou\"{e}zel, and Jean-Christophe Yoccoz.
\newblock Exponential mixing for the {T}eichm\"{u}ller flow.
\newblock {\em Publ. Math. Inst. Hautes \'{E}tudes Sci.}, (104):143--211, 2006.

\bibitem[Bab02]{Bab02}
Martine Babillot.
\newblock On the mixing property for hyperbolic systems.
\newblock {\em Israel J. Math.}, 129:61--76, 2002.

\bibitem[BV05]{BV05}
Viviane Baladi and Brigitte Vall\'{e}e.
\newblock Exponential decay of correlations for surface semi-flows without
  finite {M}arkov partitions.
\newblock {\em Proc. Amer. Math. Soc.}, 133(3):865--874, 2005.

\bibitem[BQ16]{BQ16}
Yves Benoist and Jean-Fran\c{c}ois Quint.
\newblock {\em Random walks on reductive groups}, volume~62 of {\em Ergebnisse
  der Mathematik und ihrer Grenzgebiete. 3. Folge. A Series of Modern Surveys
  in Mathematics [Results in Mathematics and Related Areas. 3rd Series. A
  Series of Modern Surveys in Mathematics]}.
\newblock Springer, Cham, 2016.

\bibitem[Bow93]{Bow93}
B.~H. Bowditch.
\newblock Geometrical finiteness for hyperbolic groups.
\newblock {\em J. Funct. Anal.}, 113(2):245--317, 1993.

\bibitem[Bri82]{Bri82}
M.~Brin.
\newblock Ergodic theory of frame flows.
\newblock In {\em Ergodic theory and dynamical systems, {II} ({C}ollege {P}ark,
  {M}d., 1979/1980)}, volume~21 of {\em Progr. Math.}, pages 163--183.
  Birkh\"{a}user, Boston, Mass., 1982.

\bibitem[BG80]{BG80}
M.~Brin and M.~Gromov.
\newblock On the ergodicity of frame flows.
\newblock {\em Invent. Math.}, 60(1):1--7, 1980.

\bibitem[BP74]{BP74}
M.~I. Brin and Ja.~B. Pesin.
\newblock Partially hyperbolic dynamical systems.
\newblock {\em Math. USSR Izv.}, 8:177--218, 1974.

\bibitem[BMMW17]{BMMW17}
Keith Burns, Howard Masur, Carlos Matheus, and Amie Wilkinson.
\newblock Rates of mixing for the {W}eil-{P}etersson geodesic flow: exponential
  mixing in exceptional moduli spaces.
\newblock {\em Geom. Funct. Anal.}, 27(2):240--288, 2017.

\bibitem[CLMS21]{CLMS21}
Mihajlo Cekić, Thibault Lefeuvre, Andrei Moroianu, and Uwe Semmelmann.
\newblock On the ergodicity of the frame flow on even-dimensional manifolds.
\newblock {\em arXiv.org}, November 2021.

\bibitem[Che98]{Che98}
N.~I. Chernov.
\newblock Markov approximations and decay of correlations for {A}nosov flows.
\newblock {\em Ann. of Math. (2)}, 147(2):269--324, 1998.

\bibitem[CS22]{CS22}
Michael Chow and Pratyush Sarkar.
\newblock Exponential mixing of frame flows for convex cocompact locally
  symmetric spaces.
\newblock arXiv:2211.14737, 2022.
\newblock Preprint.

\bibitem[DFSU21]{DFSU21}
Tushar Das, Lior Fishman, David Simmons, and Mariusz Urba\'{n}ski.
\newblock Extremality and dynamically defined measures, part {II}: measures
  from conformal dynamical systems.
\newblock {\em Ergodic Theory Dynam. Systems}, 41(8):2311--2348, 2021.

\bibitem[DZ10]{DZ10}
Amir Dembo and Ofer Zeitouni.
\newblock {\em Large deviations techniques and applications}, volume~38 of {\em
  Stochastic Modelling and Applied Probability}.
\newblock Springer-Verlag, Berlin, 2010.
\newblock Corrected reprint of the second (1998) edition.

\bibitem[Dol98]{Dol98}
Dmitry Dolgopyat.
\newblock On decay of correlations in {A}nosov flows.
\newblock {\em Ann. of Math. (2)}, 147(2):357--390, 1998.

\bibitem[Dol02]{Dol02}
Dmitry Dolgopyat.
\newblock On mixing properties of compact group extensions of hyperbolic
  systems.
\newblock {\em Israel J. Math.}, 130:157--205, 2002.

\bibitem[EO21]{EO21}
Sam Edwards and Hee Oh.
\newblock Spectral gap and exponential mixing on geometrically finite
  hyperbolic manifolds.
\newblock {\em Duke Math. J.}, 170(15):3417--3458, 2021.

\bibitem[Edw22]{Edw22}
Samuel~C. Edwards.
\newblock Effective equidistribution of the horocycle flow on geometrically
  finite hyperbolic surfaces.
\newblock {\em Int. Math. Res. Not. IMRN}, (6):4040--4092, 2022.

\bibitem[Fel71]{Fel71}
William Feller.
\newblock {\em An introduction to probability theory and its applications.
  {V}ol. {II}}.
\newblock John Wiley \& Sons, Inc., New York-London-Sydney, second edition,
  1971.

\bibitem[FS90]{FS90}
L.~Flaminio and R.~J. Spatzier.
\newblock Geometrically finite groups, {P}atterson-{S}ullivan measures and
  {R}atner's rigidity theorem.
\newblock {\em Invent. Math.}, 99(3):601--626, 1990.

\bibitem[FHP91]{FHP91}
R.~Froese, P.~Hislop, and P.~Perry.
\newblock A {M}ourre estimate and related bounds for hyperbolic manifolds with
  cusps of nonmaximal rank.
\newblock {\em J. Funct. Anal.}, 98(2):292--310, 1991.

\bibitem[Gui06]{Gui06}
Colin Guillarmou.
\newblock Resonances on some geometrically finite hyperbolic manifolds.
\newblock {\em Comm. Partial Differential Equations}, 31(1-3):445--467, 2006.

\bibitem[GK21]{GK21}
Colin Guillarmou and Benjamin K\"{u}ster.
\newblock Spectral theory of the frame flow on hyperbolic 3-manifolds.
\newblock {\em Ann. Henri Poincar\'{e}}, 22(11):3565--3617, 2021.

\bibitem[GM12]{GM12}
Colin Guillarmou and Rafe Mazzeo.
\newblock Resolvent of the {L}aplacian on geometrically finite hyperbolic
  manifolds.
\newblock {\em Invent. Math.}, 187(1):99--144, 2012.

\bibitem[KL19]{KL19}
Michael Kapovich and Beibei Liu.
\newblock Geometric finiteness in negatively pinched {H}adamard manifolds.
\newblock {\em Ann. Acad. Sci. Fenn. Math.}, 44(2):841--875, 2019.

\bibitem[Kat95]{Kat95}
Tosio Kato.
\newblock {\em Perturbation theory for linear operators}.
\newblock Classics in Mathematics. Springer-Verlag, Berlin, 1995.
\newblock Reprint of the 1980 edition.

\bibitem[Kha21]{Kha21}
Osama Khalil.
\newblock Mixing, resonances, and spectral gaps on geometrically finite locally
  symmetric spaces.
\newblock 2021.
\newblock Preprint, https://www.math.utah.edu/~khalil/publications/mixing.pdf.

\bibitem[Lal89]{Lal89}
Steven~P. Lalley.
\newblock Renewal theorems in symbolic dynamics, with applications to geodesic
  flows, non-{E}uclidean tessellations and their fractal limits.
\newblock {\em Acta Math.}, 163(1-2):1--55, 1989.

\bibitem[LP03]{LP03}
F.~Ledrappier and M.~Pollicott.
\newblock Ergodic properties of linear actions of {$(2\times 2)$}-matrices.
\newblock {\em Duke Math. J.}, 116(2):353--388, 2003.

\bibitem[Li22]{Li22}
Jialun Li.
\newblock Fourier decay, renewal theorem and spectral gaps for random walks on
  split semisimple {L}ie groups.
\newblock {\em Ann. Sci. \'{E}c. Norm. Sup\'{e}r. (4)}, 55(6):1613--1686, 2022.

\bibitem[LP22]{LP22}
Jialun Li and Wenyu Pan.
\newblock Exponential mixing of geodesic flows for geometrically finite
  hyperbolic manifolds with cusps.
\newblock {\em Invent. Math.}, page 1–91, 2022.
\newblock Advance online publication.

\bibitem[Lub10]{Lub10}
Alexander Lubotzky.
\newblock {\em Discrete groups, expanding graphs and invariant measures}.
\newblock Modern Birkh\"{a}user Classics. Birkh\"{a}user Verlag, Basel, 2010.
\newblock With an appendix by Jonathan D. Rogawski, Reprint of the 1994
  edition.

\bibitem[MMO14]{MMO14}
Gregory Margulis, Amir Mohammadi, and Hee Oh.
\newblock Closed geodesics and holonomies for {K}leinian manifolds.
\newblock {\em Geom. Funct. Anal.}, 24(5):1608--1636, 2014.

\bibitem[MO15]{MO15}
Amir Mohammadi and Hee Oh.
\newblock Matrix coefficients, counting and primes for orbits of geometrically
  finite groups.
\newblock {\em J. Eur. Math. Soc. (JEMS)}, 17(4):837--897, 2015.

\bibitem[Moo87]{Moo87}
Calvin~C. Moore.
\newblock Exponential decay of correlation coefficients for geodesic flows.
\newblock In {\em Group representations, ergodic theory, operator algebras, and
  mathematical physics ({B}erkeley, {C}alif., 1984)}, volume~6 of {\em Math.
  Sci. Res. Inst. Publ.}, pages 163--181. Springer, New York, 1987.

\bibitem[Nau05]{Nau05}
F.~Naud.
\newblock Expanding maps on {C}antor sets and analytic continuation of zeta
  functions.
\newblock {\em Ann. Sci. \'{E}cole Norm. Sup. (4)}, 38(1):116--153, 2005.

\bibitem[OS13]{OS13}
Hee Oh and Nimish~A. Shah.
\newblock Equidistribution and counting for orbits of geometrically finite
  hyperbolic groups.
\newblock {\em J. Amer. Math. Soc.}, 26(2):511--562, 2013.

\bibitem[OW16]{OW16}
Hee Oh and Dale Winter.
\newblock Uniform exponential mixing and resonance free regions for convex
  cocompact congruence subgroups of {${\rm SL}_2(\Bbb{Z})$}.
\newblock {\em J. Amer. Math. Soc.}, 29(4):1069--1115, 2016.

\bibitem[PP90]{PP90}
William Parry and Mark Pollicott.
\newblock Zeta functions and the periodic orbit structure of hyperbolic
  dynamics.
\newblock {\em Ast\'{e}risque}, (187-188):268, 1990.

\bibitem[Rat19]{Rat19}
John~G. Ratcliffe.
\newblock {\em Foundations of hyperbolic manifolds}, volume 149 of {\em
  Graduate Texts in Mathematics}.
\newblock Springer, Cham, [2019] \copyright 2019.
\newblock Third edition [of 1299730].

\bibitem[Rat87]{Rat87}
Marina Ratner.
\newblock The rate of mixing for geodesic and horocycle flows.
\newblock {\em Ergodic Theory Dynam. Systems}, 7(2):267--288, 1987.

\bibitem[Rob03]{Rob03}
Thomas Roblin.
\newblock Ergodicit\'{e} et \'{e}quidistribution en courbure n\'{e}gative.
\newblock {\em M\'{e}m. Soc. Math. Fr. (N.S.)}, (95):vi+96, 2003.

\bibitem[Rud82]{Rud82}
Daniel~J. Rudolph.
\newblock Ergodic behaviour of {S}ullivan's geometric measure on a
  geometrically finite hyperbolic manifold.
\newblock {\em Ergodic Theory Dynam. Systems}, 2(3-4):491--512 (1983), 1982.

\bibitem[Sar22a]{Sar22a}
Pratyush Sarkar.
\newblock Generalization of {S}elberg's 3/16 theorem for convex cocompact thin
  subgroups of {$\operatorname{SO}(n,1)$}.
\newblock {\em Adv. Math.}, 409(part A):Paper No. 108610, 58, 2022.

\bibitem[Sar22b]{Sar22b}
Pratyush Sarkar.
\newblock {\em Uniform {E}xponential {M}ixing of {F}rame {F}lows for
  {C}ongruence {C}overs of {C}onvex {C}ocompact {H}yperbolic {M}anifolds}.
\newblock ProQuest LLC, Ann Arbor, MI, 2022.
\newblock Thesis (Ph.D.)--Yale University.

\bibitem[SW21]{SW21}
Pratyush Sarkar and Dale Winter.
\newblock Exponential mixing of frame flows for convex cocompact hyperbolic
  manifolds.
\newblock {\em Compos. Math.}, 157(12):2585--2634, 2021.

\bibitem[Sid22]{Sid22}
Salman Siddiqi.
\newblock Decay of correlations for certain isometric extensions of {A}nosov
  flows.
\newblock {\em Ergodic Theory Dynam. Systems}, page 1–51, 2022.
\newblock Advance online publication.

\bibitem[Sto11]{Sto11}
Luchezar Stoyanov.
\newblock Spectra of {R}uelle transfer operators for axiom {A} flows.
\newblock {\em Nonlinearity}, 24(4):1089--1120, 2011.

\bibitem[TW22]{TW22}
Nattalie Tamam and Jacqueline~M. Warren.
\newblock Effective equidistribution of horospherical flows in infinite volume
  rank-one homogeneous spaces.
\newblock {\em Ergodic Theory Dynam. Systems}, page 1–61, 2022.
\newblock Advance online publication.

\bibitem[TZ23]{TZ23}
Masato Tsujii and Zhiyuan Zhang.
\newblock Smooth mixing {A}nosov flows in dimension three are exponentially
  mixing.
\newblock {\em Ann. of Math. (2)}, 197(1):65--158, 2023.

\bibitem[Win15]{Win15}
Dale Winter.
\newblock Mixing of frame flow for rank one locally symmetric spaces and
  measure classification.
\newblock {\em Israel J. Math.}, 210(1):467--507, 2015.

\bibitem[Woo82]{Woo82}
Michael Woodroofe.
\newblock {\em Nonlinear renewal theory in sequential analysis}, volume~39 of
  {\em CBMS-NSF Regional Conference Series in Applied Mathematics}.
\newblock Society for Industrial and Applied Mathematics (SIAM), Philadelphia,
  Pa., 1982.

\bibitem[You98]{You98}
Lai-Sang Young.
\newblock Statistical properties of dynamical systems with some hyperbolicity.
\newblock {\em Ann. of Math. (2)}, 147(3):585--650, 1998.

\end{thebibliography}
